\documentclass[12pt,british,refpage,intoc,bibliography=totoc,index=totoc,BCOR=7.5mm,captions=tableheading]{extarticle}
\usepackage[T1]{fontenc}
\usepackage[latin9]{inputenc}
\usepackage[a4paper]{geometry}
\usepackage{color}
\usepackage{babel}
\usepackage{amsmath}
\usepackage{amsthm}
\usepackage{amssymb}
\usepackage{graphicx}
\usepackage{subfig}
\usepackage[numbers]{natbib}
\usepackage[unicode=true,pdfusetitle,
 bookmarks=true,bookmarksnumbered=true,bookmarksopen=false,
 breaklinks=false,pdfborder={0 0 1},backref=false,colorlinks=true]
 {hyperref}
\hypersetup{
 pdfborderstyle=,linkcolor=black,citecolor=blue,urlcolor=black,filecolor=blue,pdfpagelayout=OneColumn,pdfnewwindow=true,pdfstartview=XYZ,plainpages=false}

\makeatletter

\numberwithin{equation}{section}
\numberwithin{figure}{section}
\theoremstyle{plain}
\newtheorem{thm}{\protect\theoremname}[section]
\theoremstyle{remark}
\newtheorem{rem}[thm]{\protect\remarkname}
\theoremstyle{remark}
\newtheorem*{rem*}{\protect\remarkname}
\theoremstyle{plain}
\newtheorem{cor}[thm]{\protect\corollaryname}
\theoremstyle{plain}
\newtheorem{lem}[thm]{\protect\lemmaname}

\@ifundefined{date}{}{\date{}}
\usepackage{babel}
\usepackage{caption}
\usepackage[nottoc]{tocbibind}
\allowdisplaybreaks[4]
\usepackage{dsfont}
\DeclareMathAlphabet{\mathcal}{OMS}{cmsy}{m}{n}

\newcommand{\vertiii}[1]{{\left\vert\kern-0.25ex\left\vert\kern-0.25ex\left\vert #1 \right\vert\kern-0.25ex\right\vert\kern-0.25ex\right\vert}}

\providecommand{\lemmaname}{Lemma}

\providecommand{\remarkname}{Remark}
\providecommand{\theoremname}{Theorem}

\makeatother

\providecommand{\corollaryname}{Corollary}
\providecommand{\lemmaname}{Lemma}
\providecommand{\remarkname}{Remark}
\providecommand{\theoremname}{Theorem}

\begin{document}
\title{Random vortex dynamics and Monte-Carlo\\ simulations for wall-bounded viscous flows}
\author{By V. Cherepanov\thanks{Mathematical Institute, University of Oxford, Oxford OX2 6GG. Email:
vladislav.cherepanov@maths.ox.ac.uk},$\ \; $S. W. Ertel\thanks{Institut f{\"u}r Mathematik, Technische Universit{\"a}t Berlin, Stra{\ss}e
des 17. Juni 136, 10623 Berlin, DE. Email: \protect\protect\protect\protect\protect\href{mailto:ertel@math.tu-berlin.de}{ertel@math.tu-berlin.de}},\    Z. Qian\thanks{Mathematical Institute, University of Oxford, Oxford OX2 6GG, UK.,
and OSCAR, Suzhou, China. Email: \protect\protect\protect\protect\protect\protect\href{mailto:qianz@maths.ox.ac.uk}{qianz@maths.ox.ac.uk}}\ \  and \ J. Wu\thanks{Department of Mathematical Science , Beijing Normal University - Hong Kong Baptist University United International College, Zhuhai, Guangdong 519087, China. 
Email: jianglunwu@uic.edu.cn } }

\maketitle

\begin{abstract}

Functional integral representations for solutions of the motion
equations for  wall-bounded incompressible viscous
flows, expressed (implicitly) in terms of distributions of solutions to
 stochastic differential equations of McKean-Vlasov type,  are established by using a perturbation
technique. These representations are used to obtain 
exact random vortex dynamics for wall-bounded viscous flows.
Numerical schemes therefore are proposed and the convergence 
of the numerical schemes for random vortex dynamics 
with an additional force term is established. Several numerical experiments
are carried out for demonstrating the motion of a viscous flow within
a thin layer next to the fluid boundary.

\medskip{}


\medskip{}

\end{abstract}

\section{Introduction}

In this paper, we shall establish Feynman-Kac type formulas for solutions to the motion equations of an incompressible viscous fluid flow past a solid wall. These representations are useful tools for implementing (Monte-Carlo) direct numerical simulations (DNS) for wall-bounded incompressible viscous flows. In order to explain our method, the paper is devoted to two-dimensional (2D) flows past flat plates only, the technology in the paper however can be applied to three-dimensional (3D) flows with appropriate (while substantial) modifications. However the proof of the convergence of 3D random vortex dynamics for flows past solid walls requires substantial new techniques which is beyond the scope of the present paper. 

The numerical method for calculating wall-bounded flows to be established in the paper is different from the traditional DNS (cf. \citep{OrszagPatterson1972}, \citep{Spalart1988}, \citep{MoinMashesh1998} for example), and is based on the exact stochastic formulation of the Navier-Stokes equations satisfying the non-slip condition. The approach, which may be called a random vortex method, is in spirit very similar to the mean field approach, while here we deal with a much more complicated boundary value problem of the fluid dynamical equations. The present work therefore aims to provide new tools for the study of incompressible viscous flows within thin layers next to the boundaries. Numerical experiments are carried out to demonstrate this point.

The random vortex method, initiated in Chorin \citep{Chorin 1973}, has been carried out for 2D incompressible viscous flows without boundary. The exact random vortex model for 2D viscous flows has been put forward by Goodman \citep{Goodman1987} and Long \citep{Long1988} (cf. also
\citep{CottetKoumoutsakos2000}, \citep{Majda and Bertozzi 2002} and the literature cited therein).

One of the goals of the present work is to prove the convergence of the numerical schemes for solving 2D Navier-Stokes equations with an external force applied to the fluid based on their corresponding random vortex dynamics. The convergence of the 2D random vortex dynamics was proved in Goodman \citep{Goodman1987} and Long \citep{Long1988}, and their proof was refined in Majda and Bertozzi \citep{Majda and Bertozzi 2002} and by other researchers. To our best knowledge, the convergence of random vortex dynamics treated in the literature does not contain any external vorticity in the vorticity transport equation. As we shall demonstrate in the present work, in order to deal with the boundary vorticity, we shall deform the vorticity transport equation by adding a suitable "external vorticity" term. Therefore it is important to establish the convergence result for numerical schemes built upon the vorticity transport equation with an applying vorticity force. Although our proof of the convergence follows the same line as those in the literature, the techniques used in the key steps of the proof are quite different. As a by-product, we are able to provide a transparent proof of the convergence result for 2D random vortex dynamics and to clarify several technical issues in the existing proofs of the convergence of 2D random vortex dynamics.

The second goal of the present work is to establish a useful numerical method based on the stochastic formulation of incompressible viscous flows past a solid wall, which may be called Monte-Carlo DNS. The random vortex formulation is established by using
a new functional integral representation of the solutions to certain parabolic equations. We demonstrate our method by several numerical 
experiments based on Monte-Carlo DNS we put forward.

The paper is organised as follows. In Section \ref{Random vortex method -- flows without constraint}, we state the stochastic integral representation for unconstrained flows with external force and propose random vortex schemes based on the representation. We analyse the corresponding mean-field equation in Section \ref{section mean field analysis}. In particular, we show the existence and uniqueness and prove that the velocity is indeed approximated by corresponding dynamics with regularised kernels and localised forcing which is used in simulations. Section \ref{section convergence} is devoted to the proof of the convergence of the discretised random vortex schemes proposed before. In Section \ref{Incompressible viscous flows past a wall}, we formulate the representation for viscous flows past a wall with additional "external" vorticity rendering the formulation in the form of the external force problem. In the last Section \ref{Numerical simulations of wall-bounded flows}, we provide the numerical implementation for the scheme for a wall-bounded flow and report the results of the conducted numerical experiments. 

\vskip0.3truecm

\emph{Conventions and notations}. Let us list several conventions
and notations which will be used throughout the paper.

\vskip0.3truecm

1) Let $D=\left\{ (x_{1},\ldots,x_{d}):x_{d}>0\right\} $ with its
boundary $\partial D=\left\{ x_{d}=0\right\} $. The dimension $d\geq2$
unless otherwise specified. The reflection of $\mathbb{R}^{d}$ about
the hyperspace $x_{d}=0$ maps $x=(x_{1},\ldots,x_{d})$ to $\bar{x}=(x_{1},\ldots,-x_{d})$.

\vskip0.3truecm

2) By a time dependent vector field $V$ (on $\mathbb{R}^{d}$ or
on $D$ depending on the context), whose components (with respect
to the standard coordinate system) are denoted by $V^{i}(x,t)$ (where
$i=1,\ldots,d$), we mean that for every $t$, $x\mapsto V(x,t)$
is a vector field. For this case we also say $V(x,t)$ is a time dependent
vector field in order to emphasize the space variable $x$ and time
parameter $t$.

\vskip0.3truecm

3) For a time dependent tensor field $V(x,t)$, the gradient $\nabla V$,
the divergence $\nabla\cdot V$, the Laplacian $\Delta V$ and etc.
apply only to the space variable $x$, and $t$ is considered as a
parameter. For example, if $V$ is a (time-dependent) vector field
with components $V^{i}$ (for $i=1,\ldots,d$), then $\nabla\cdot V=\sum_{i=1}^{d}\frac{\partial}{\partial x_{i}}V^{i}$.

\vskip0.3truecm

4) Unless otherwise specified, if $F(x)$ (for $x\in D$) is a tensor
field defined on $D$, then $F$ is extended to all $x\in\mathbb{R}^{d}$
such that $F(x)=0$ for $x\notin D$. The same convention will be
applied to time dependent tensor field $F(x,t)$ on $D$ so that $F(x,t)=0$
for $x\notin D$. This convention is adopted in the paper to simplify
formulas in the main context. There is an important exception to this
rule --- the velocity vector field $u(x,t)$ on $D$ is extended
to a vector field on $\mathbb{R}^{d}$ in a different manner, which
will be defined in Section \ref{Incompressible viscous flows past a wall}.

\vskip0.3truecm

5) Unless otherwise specified, repeated indices appearing in a mono-term
are summed over their region. For example, if $V$ is a vector field,
then $\nabla\cdot V=\frac{\partial}{\partial x_{i}}V^{i}$.

\vskip0.3truecm

6) The following convention for two-dimensional vectors will be applied
in this paper. To this end we make use of the following convention
on two-dimensional vectors. If $a=(a_{1},a_{2})$ and $b=(b_{1},b_{2})$,
then $a^{\bot}=(-a_{2},a_{1})$, the cross products 
\[
a\wedge b=a_{1}b_{2}-a_{2}b_{1}=a^{\bot}\cdot b
\]
and 
\[
a\wedge c=(a_{2}c,-a_{1}c)=-ca^{\bot}
\]
for any scalar $c$.

\section{Random vortex method -- flows without constraint}\label{Random vortex method -- flows without constraint}

In this section we derive a random vortex formulation for two-dimensional
incompressible viscous flows without space constraint. The mathematical
model for such a flow may be described by its velocity $u(x,t)=(u^{1}(x,t),u^{2}(x,t))$
and by the pressure $P(x,t)$ where $x\in\mathbb{R}^{2}$ and $t\geq0$.
Let $\nu>0$ be the viscosity of the fluid flow. Then $u$ and $P$
solve the Navier-Stokes equations 
\begin{equation}
\frac{\partial}{\partial t}u+(u\cdot\nabla)u-\nu\Delta u+\nabla P-F=0\label{NS2D1}
\end{equation}
and 
\begin{equation}
\nabla\cdot u=0\label{NS2D2}
\end{equation}
in $\mathbb{R}^{2}\times[0,\infty)$, where $F(x,t)=(F^{1}(x,t),F^{2}(x,t))$
is an external force applying to the fluid flow.

The vorticity $\omega=\nabla\wedge u$, which is a time dependent
scalar field for a 2D flow, evolves according to the vorticity equation
\begin{equation}
\frac{\partial}{\partial t}\omega+(u\cdot\nabla)\omega-\nu\Delta\omega-G=0\label{Vor2D1}
\end{equation}
where $G=\nabla\wedge F$.

\subsection{Stochastic integral representation}

In literature, the random vortex method is built upon two facts. Firstly,
since $u$ is divergence-free, the transition probability
density function $p_{u}(s,x;t,y)$ (for $t>s\geq0$ and $x,y\in\mathbb{R}^{2}$)
of the diffusion process with infinitesimal generator $\nu\Delta+u\cdot\nabla$
is the fundamental solution to the forward problem of the parabolic
equation 
\[
\left(\frac{\partial}{\partial t}+u\cdot\nabla-\nu\Delta\right)w=0
\]
which allows to represent the vorticity $\omega$ in terms of $p_{u}(s,x;t,y)$.
Indeed, 
\begin{equation}
\omega(x,t)=\int_{\mathbb{R}^{2}}p_{u}(0,\xi;t,x)\omega_{0}(\xi)\textrm{d}\xi+\int_{0}^{t}\int_{\mathbb{R}^{2}}p_{u}(s,\xi;t,x)G(\xi,s)\textrm{d}\xi\textrm{d}s\label{REP2D1}
\end{equation}
where $\omega_{0}=\nabla\wedge u_{0}$ is the initial vorticity and
$u_{0}$ is the initial velocity.

Secondly, the Biot-Savart law which allows, under certain assumptions
on the decay of the velocity at infinity, one recovers the velocity from the vorticity by
\begin{equation}
u(x,t)=\int_{\mathbb{R}^{2}}K(x-y)\omega(y,t)\textrm{d}y\label{BSK-T-2D1}
\end{equation}
where 
\begin{equation}
K(\xi)=\frac{\xi^{\bot}}{2\pi|\xi|^{2}},\quad\textrm{ for }\xi=(\xi_{1},\xi_{2})\in\mathbb{R}^{2}\textrm{ }\label{BSK-2D2}
\end{equation}
is called the Biot-Savart kernel in $\mathbb{R}^{2}$.

As a matter of fact, numerical schemes based on \eqref{REP2D1} are
computationally expensive, therefore we develop an alternative representation
of the vorticity $\omega$, that leads to a feasible numerical scheme.

Let $X(\xi,t)$ (also denoted later by $X_{t}^{\xi}$) for $\xi\in\mathbb{R}^{2}$
be the Taylor diffusion family defined by the SDE 
\begin{equation}
\textrm{d}X(\xi,t)=u\left(X(\xi,t),t\right)\textrm{d}t+\sqrt{2\nu}\textrm{d}B_{t},\quad X(\xi,0)=\xi\label{SDE-2D}
\end{equation}
where $B$ is a Brownian motion on some probability space. 
\begin{thm}
\label{Theorem - representation of vorticity} The vorticity $\omega$
possesses the following representation 
\begin{align}
\begin{split}
\omega(x,t) & =\int_{\mathbb{R}^{2}}p_{u}(0,\xi,t,x)\omega_{0}(\xi)\textrm{d}\xi \\
 & +\int_{0}^{t}\int_{\mathbb{R}^{2}}\mathbb{E}\left[\left.G(X(\xi,s),s)\right|X(\xi,t)=x\right]p_{u}(0,\xi,t,x)\textrm{d}\xi\textrm{d}s
 \end{split}
 \label{VREP2D1}
\end{align}
for $x\in\mathbb{R}^{2}$ and $t\geq0$. 
\end{thm}

\begin{proof}
We make use of the duality of conditional laws of diffusions, established
in \citep{ProcA-paper2}. Take arbitrary fixed $T>0$. Let $\tilde{X}(\xi,\cdot)$
be the unique (weak) solution of the following SDE: 
\begin{equation}
\textrm{d}\tilde{X}(\xi,t)=-u(\tilde{X}(\xi,t),T-t)\textrm{d}t+\sqrt{2\nu}\textrm{d}B_{t},\quad\tilde{X}(\xi,0)=\xi\label{ba-s01-1-1-1}
\end{equation}
for every $\xi\in\mathbb{R}^{2}$, where $u(x,t)=0$ for $t<0$. Let
$Y_{t}=\omega(\tilde{X}(\xi,t),T-t)$. By It\^{o}'s formula and the vorticity
transport equation to deduce that 
\begin{equation}
Y_{t}=Y_{0}+\sqrt{2\nu}\int_{0}^{t}\nabla\omega(\tilde{X}(\xi,s),T-s)\cdot\textrm{d}B_{s}-\int_{0}^{t}G(\tilde{X}(\xi,s),T-s)\textrm{d}s\label{Y-aa1-1-1}
\end{equation}
for every $t\in[0,T]$. Taking expectation on both sides we obtain
that 
\[
\omega(\xi,T)=\mathbb{E}\left[\omega_{0}(\tilde{X}(\xi,T))\right]+\int_{0}^{T}\mathbb{E}\left[G(\tilde{X}(\xi,t),T-t)\right]\textrm{d}t.
\]
The last term in the previous equation can be rewritten using the
conditional expectation, to obtain that 
\begin{align*}
\omega(\xi,T) & =\int_{\mathbb{R}^{2}}\omega_{0}(\eta)p_{-u_{T}}(0,\xi,T,\eta)\textrm{d}\eta\\
 & +\int_{0}^{T}\int_{\mathbb{R}^{2}}\mathbb{E}\left[\left.G(\tilde{X}(\xi,t),T-t)\right|\tilde{X}(\xi,T)=\eta\right]p_{-u_{T}}(0,\xi,T,\eta)\textrm{d}\eta\textrm{d}t,
\end{align*}
where $p_{-u_{T}}(\tau,\xi,t,\eta)$ is the transition probability
density function of the diffusion $\tilde{X}(\xi,\cdot)$. Since $\nabla\cdot u=0$,
\[
p_{-u_{T}}(0,\xi,T,\eta)=p_{u}(0,\eta,T,\xi),
\]
which yields that 
\begin{align}
\begin{split}
\omega(\xi,T) & =\int_{\mathbb{R}^{2}}\omega_{0}(\eta)p_{u}(0,\eta,T,\xi)\textrm{d}\eta \\
 & +\int_{0}^{T}\int_{\mathbb{R}^{2}}\mathbb{E}\left[\left.G(\tilde{X}(\xi,t),T-t)\right|\tilde{X}(\xi,T)=\eta\right]p_{u}(0,\eta,T,\xi)\textrm{d}\eta\textrm{d}t
 \end{split}
 \label{W-aa2-2-1}
\end{align}
for every $\xi\in\mathbb{R}^{2},T>0$. Thanks to the duality of conditional
laws from \citep{ProcA-paper2}, the conditional expectation
can be written in terms of the diffusion process with infinitesimal
generator $\nu\Delta+u\cdot\nabla$, which therefore leads to
\begin{align}
\begin{split}\omega(\xi,T) & =\int_{\mathbb{R}^{2}}\omega_{0}(\eta)p_{u}(0,\eta,T,\xi)\textrm{d}\eta\\
 & +\int_{0}^{T}\int_{\mathbb{R}^{2}}\mathbb{E}\left[\left.G(X(\eta,t),t)\right|X(\eta,T)=\xi\right]p_{u}(0,\eta,T,\xi)\textrm{d}\eta\textrm{d}t.
\end{split}
\label{W-aa2-1-1-1-1}
\end{align}
The proof is complete. 
\end{proof}
As a consequence of this representation of the vorticity, we have
the random vortex representation for the velocity. 
\begin{thm}
With the notations above, the velocity 
\begin{equation}
u(x,t)=\int_{\mathbb{R}^{2}}\mathbb{E}\left[K\left(x-X(\xi,t)\right)\left(\omega_{0}(\xi)+\int_{0}^{t}G(X(\xi,s),s)\textrm{d}s\right)\right]\textrm{d}\xi\label{RV-2D1}
\end{equation}
for $x\in\mathbb{R}^{2}$ and $t\geq0$, where $\omega_{0}=\nabla\wedge u_{0}$
is the initial vorticity. 
\end{thm}

\begin{proof}
The proof follows from the Biot-Savart law and the stochastic integral representation for the vorticity: 
\begin{align*}
u(x,t) & =\int_{\mathbb{R}^{2}}\left(\int_{\mathbb{R}^{2}}K(x-y)p_{u}(0,\xi,t,y)\textrm{d}y\right)\omega_{0}(\xi)\textrm{d}\xi\\
 & +\int_{0}^{t}\int_{\mathbb{R}^{2}}\left(\int_{\mathbb{R}^{2}}K(x-y)\mathbb{E}\left[\left.G(X(\xi,s),s)\right|X(\xi,t)=y\right]p_{u}(0,\xi,t,y)\textrm{d}y\right)\textrm{d}\xi\textrm{d}s\\
 & =\int_{\mathbb{R}^{2}}\left(\int_{\mathbb{R}^{2}}K(x-y)p_{u}(0,\xi,t,y)\textrm{d}y\right)\omega_{0}(\xi)\textrm{d}\xi\\
 & +\int_{0}^{t}\int_{\mathbb{R}^{2}}\left(\int_{\mathbb{R}^{2}}\mathbb{E}\left[\left.K\left(x-X(\xi,t)\right)G(X(\xi,s),s)\right|X(\xi,t)=y\right]p_{u}(0,\xi,t,y)\textrm{d}y\right)\textrm{d}\xi\textrm{d}s\\
 & =\int_{\mathbb{R}^{2}}\mathbb{E}\left[K\left(x-X(\xi,t)\right)\right]\omega_{0}(\xi)\textrm{d}\xi\\
 & +\int_{0}^{t}\int_{\mathbb{R}^{2}}\mathbb{E}\left[K\left(x-X(\xi,t)\right)G(X(\xi,s),s)\right]\textrm{d}\xi\textrm{d}s.
\end{align*}
\end{proof}

\begin{rem}
Note that even though the SDE used in the representation \eqref{RV-2D1} involves finite-dimensional noise, in fact it lives in an infinite-dimensional function space. One should thus rather view \eqref{SDE-2D} together with \eqref{RV-2D1} as a field-valued McKean--Vlasov equation. 
\end{rem}

\subsection{Random vortex schemes}

Based on the random vortex formulation \eqref{SDE-2D}, \eqref{RV-2D1}
one can derive numerical schemes for simulating solutions of viscous
fluid flows.

A standard method to approximate (field-valued) McKean--Vlasov type
of stochastic differential equations like \eqref{SDE-2D}, \eqref{RV-2D1} is to use a system of interacting SDEs, which are field valued in our case. With the strong law of large numbers in mind, we consider the following system.

Let $N\in\mathbb{N}$ and consider the system 
\begin{align}
u^{N}(x,t) & =\frac{1}{N}\sum_{i=1}^{N}\int_{\mathbb{R}^{2}}K\left(x-X^{i;N}(\eta,t)\right)\left(\omega_{0}(\eta)+\int_{0}^{t}G\left(X^{i;N}(\eta,s),s\right)\textrm{d}s\right)\textrm{d}\eta\label{u-n}
\end{align}
and 
\begin{align}
 & \textrm{d}X^{j;N}(\eta,t)=u^{N}\left(X^{j;N}(\eta,t),t\right)\textrm{d}t+\sqrt{2\nu}\textrm{d}B_{t}^{j},\quad X^{j;N}(\eta,0)=\eta\label{interacting continuous fields}
\end{align}
where $j=1,\ldots,N$ and $\left(B^{j}\right)_{j\in\mathbb{N}}$
are independent copies of the 2D Brownian motion $B$. The above
system is obtained formally by replacing the expectation appearing
in \eqref{RV-2D1} by the sample mean of `independent' copies of the
Taylor diffusion. To turn \eqref{interacting continuous fields} into
a viable numerical approximation, one needs to consider a discretisation of the spatial integral involved. For flows without space constraint, a uniform grid of size $h>0$
may be applied.

To ensure the well-posedness of the induced particle approximation,
it is easy to have the Biot--Savart kernel $K$ regularised into
a function denoted by $K_{\delta}$ (where $\delta>0$ is a parameter),
so that $K_{\delta}$ is smooth up to the origin $x=0$.

Besides the singularity of the Biot--Savart kernel, another problem
of naive discretisation of \eqref{u-n} is that even for external
vorticity $G$ with compact support, the function $\eta\mapsto\mathbb{E}\left[G(X(\eta,t),t)\right]$
will not have compact support in general. Thus, in the inhomogeneous setting,
the well-posedness of the integrals in \eqref{u-n} becomes an issue
and in naive numerical discretisation every node would interact with
infinitely many other nodes. As in practice this can not be implemented in numerical schemes, we need a cut-off function $\chi_{R}:\mathbb{R}^{2}\to[0,1]$ that vanishes outside of the ball centered at $0$ with radius $R>0$ and thus avoids infinite numerical stencils. There are various possible regularisations $K_{\delta}$ and cut-off $\chi_{R}$ and our analysis can deal with many of the reasonable choices. The assumptions we require are made precise in the beginning of the next section.

This then gives us the following Monte Carlo approximation 
\begin{align}
\begin{split}
\hat{u}_{N,h,\delta,R}\left(x,t\right):=&\frac{1}{N}\sum_{k=1}^{N}\sum_{j\in\mathbb{Z}^{2}}h^{2}K_{\delta}\left(x-\hat{X}_{N,h,\delta,R}^{k}(j,t)\right)\Bigl(\omega_{0}(jh) \\
 & \qquad+\int_{0}^{t}G\left(\hat{X}_{N,h,\delta,R}^{k}(j,s),s\right)\textrm{d}s ~ \chi_{R}(jh)\Bigl),
\end{split}
\label{monte Carlo velocity}
\end{align}
where $\hat{X}_{N,h,\delta,R}^{k}(j,\cdot)$ (for $k=1,\ldots,N$ and $j\in\mathbb{Z}^{2}$) are determined by the SDE: 
\begin{align}
\mathrm{d}\hat{X}_{N,h,\delta,R}^{k}(j,t)=\hat{u}_{N,h,\delta,R}\left(\hat{X}_{N,h,\delta,R}^{k}(j,t),t\right)\mathrm{d}t+\sqrt{2\nu}\textrm{d}B_{t}^{k},\quad\hat{X}_{N,h,\delta,R}^{k}(j,0)=jh.\label{Monte Carlo dynamics}
\end{align}
Since $\chi_{R}$ has a compact support, the second sum in \eqref{monte Carlo velocity} is indeed a finite sum.

We will refer to the combined system \eqref{monte Carlo velocity},
\eqref{Monte Carlo dynamics} as the \textbf{Field Monte Carlo Random
Vortex (FMCRV)} approximation. Due to the localisation there are only
a finite number of $\hat{X}_{N,h,\delta,R}^{k}(j,t)$ that determine
the velocity $\hat{u}_{N,h,\delta,R}\left(\cdot,t\right)$, and due
to the regularisation the velocity is Lipschitz. Thus one easily verifies that FMCRV approximation is well-posed.

The FMCRV is based on the law of large numbers for i.i.d. random variables. One can also derive approximations based on concentration inequalities for non identically distributed random variables. 

One known way to implement a Monte-Carlo method is based on the following modified system:
\begin{equation}
\mathrm{d}X_{h,\delta,R}(i,t)=u_{h,\delta,R}\left(X_{h,\delta,R}(i,t),t\right)\mathrm{d}t+\sqrt{2\nu}\mathrm{d}B^{i}(t),~X_{h,\delta,R}(i,0)=ih,
\label{hyperrandom vortex dynamics - general SDE}
\end{equation}
and
\begin{align}
\begin{split}
 u_{h,\delta,R}(x,t)&=\sum_{j\in\mathbb{Z}^{2}}h^{2}K_{\delta}\left(x-X_{h,\delta,R}(j,t)\right)\left(\omega_{0}(jh)\right.\\ &\qquad+\int_{0}^{t}G\left(X_{h,\delta,R}(j,s),s\right)\textrm{d}s\left.\chi_{R}(jh)\right),
\end{split}
\label{hyperrandom vortex dynamics - general velocity}
\end{align}
where $B^{i}$ are independent Brownian motions, which will be referred to as the \textbf{Particle Monte Carlo Random Vortex (PMCRV)} approximation.

\section{Some facts about the mean-field equation}
\label{section mean field analysis}

In this section, we prove the well-posedness of the mean-field random
vortex dynamics in the form of McKean--Vlasov type dynamics 
\begin{align}
\begin{split}\mathrm{d}X(\eta,t) & =u\left(X(\eta,t),t\right)\mathrm{d}t+\sqrt{2\nu}\mathrm{d}B_{t},\quad X(\eta,0)=\eta,\\
u(x,t) & =\int_{\mathbb{R}^{2}}\mathbb{E}\left[K(x-X(\eta,t))\left(\omega_{0}(\eta)+\int_{0}^{t}G(X(\eta,s),s)\textrm{d}s\right)\right]\textrm{d}\eta
\end{split}
\label{mean field random vortex dynamics - summarized form}
\end{align}
where $B$ is a two-dimensional Brownian motion. Furthermore we show
that on the mean-field level the velocity can be robustly approximated
by dynamics with regularised kernels and localised forcing. To this
end, we make the following assumptions that shall hold throughout this
and the next section (i.e. Sections \ref{section mean field analysis}
and \ref{section convergence}).

\textit{Assumption 1.} \label{Assumption - initial vorticity and force }
We assume that the initial vorticity $\omega_{0}\in C^{2}\left(\mathbb{R}^{2};\mathbb{R}\right)$
with a compact support. Similarly, we assume that the external vorticity
$G$ is smooth, and $G(x,t)=0$ for $|x|>R_{0}$ and for all $t$. 

\textit{Assumption 2.} \label{assumption Kdelta - overview} We consider
a family of vector kernels $\left\{ K_{\delta}:\delta>0\right\} $
which has the following properties:
\begin{itemize}
\item \label{assumption Kdelta - Lipschitz} $K_{\delta}\in C^{2}\left(\mathbb{R}^{2};\mathbb{R}^{2}\right)$
with bounded first and second derivatives, although in general the
bounds of the derivatives might depend on $\delta>0$.
\item \label{assumption Kdelta - bound} There exists a constant $C_{0}$,
such that 
\begin{align*}
\left|K_{\delta}(z)\right|\leq\frac{C_{0}}{|z|}  \quad\textrm{ for all }\delta>0\textrm{ and }z\neq0.
\end{align*}
\item \label{assumption Kdelta - epsilon existence} $K_{\delta}$ approximates
the Biot-Savart kernel $K$ in the following sense that, there is
$\epsilon_{\delta}:\mathbb{R}^{2}\to[0,\infty)$ for each $\delta>0$
such that 
\begin{align}
\left|K_{\delta}(z)-K(z)\right|\leq\frac{\epsilon_{\delta}(z)}{|z|}\quad\text{ for all }z\neq0\label{assumption Kdelta - epsilon definition-1}
\end{align}
for any $\delta>0$, where the family of functions $\epsilon_{\delta}$
has to satisfy the following condition.
\item \label{assumption Kdelta - convergence of error} It holds that 
\begin{align*}
\int_{\left\{ |z|\leq1\right\} }\frac{\epsilon_{\delta}(z)}{|z|}\mathrm{d}z+\sup_{|z|\geq1}\frac{\epsilon_{\delta}(z)}{|z|}\leq C_{\mathrm{reg}}\delta
\end{align*}
for some positive constant $C_{\mathrm{reg}}>0$. 
\end{itemize}
\begin{rem*}
A canonical example used throughout this work is the following 
\begin{align*}
K_{\delta}(z):=\left(1-e^{-\frac{|z|^{2}}{2\delta^{2}}}\right)K(z)
\end{align*}
for $\delta>0$. Then the assumptions listed above hold with $\epsilon_{\delta}(z):=\exp\left(-\frac{|z|^{2}}{2\delta^{2}}\right)$.
Moreover $K_{\delta}$ is globally Lipschitz continuous.
\end{rem*}

Furthermore we make the following assumption regarding the cutoff $\chi_{R}$.

\textit{Assumption} \textit{3}.\label{assumption chiR} The cut-off
function $\chi_{R}\in C^{2}\left(\mathbb{R}^{2},[0,1]\right)$, for
$R>0$, such that 
\begin{align}
\begin{split}\chi_{R}(z) & =1~\text{for}~|z|<R,\\
\chi_{R}(z) & =0~\text{for}~|z|>R+1
\end{split}
\label{cutoff function definition}
\end{align}
and the first two derivatives of $\chi_{R}$ are bounded. 

For simplicity, we set $K_{0}=K$ and $\chi_{\infty}=1$. It is understood that the mean-field random vortex dynamics \eqref{mean field random vortex dynamics - summarized form} are defined for $\delta=0$ or/and $R=\infty$. 

We consider the (field-valued) McKean--Vlasov equation 
\begin{align}
\mathrm{d}X_{\delta,R}(x,t)=u_{\delta,R}\left(X_{\delta,R}(x,t),t\right)\mathrm{d}t+\sqrt{2\nu}\mathrm{d}B_{t},~X_{\delta,R}(x,0)=x\label{regularised mf equation}
\end{align}
where 
\begin{align}
\begin{split}u_{\delta,R}(x,t) & =\int_{\mathbb{R}^{2}}\mathbb{E}\left[K_{\delta}\left(x-X_{\delta,R}(\eta,t)\right)\right]\omega_{0}(\eta)\textrm{d}\eta\\
 & +\int_{0}^{t}\int_{\mathbb{R}^{2}}\mathbb{E}\left[K_{\delta}\left(x-X_{\delta,R}(\eta,t)\right)G(X_{\delta,R}(\eta,s),s)\right]\chi_{R}(\eta)\textrm{d}\eta\textrm{d}s.
\end{split}
\label{mean field form of drift - regularised}
\end{align}

Firstly we show that the system \eqref{regularised mf equation}, \eqref{mean field form of drift - regularised} has a unique solution. To this end, for any given bounded drift function $b:\mathbb{R}^{2}\times\left[0,\infty\right)\to\mathbb{R}^{2}$,
we denote by $X^{b}$ the (weak) solution to 
\begin{align}
\mathrm{d}X^{b}(x,t)=b\left(X^{b}(x,t),t\right)\mathrm{d}t+\sqrt{2\nu}\mathrm{d}B_{t},\quad X^{b}(x,0)=x\label{SDE for arbitrary drift}
\end{align}
and $p_{b}\left(\tau,x,t,y\right)$ (for $t>\tau\geq0$) denotes its
transition probability density function, which is continuous. 

With this in mind, we define for any drift $b$ the operation $\mathcal{K}\left(K_{\delta},\chi_{R},\omega_{0},G\right)\diamond b$ by 
\begin{align*}
\left(\mathcal{K}\left(K_{\delta},\chi_{R},\omega_{0},G\right)\diamond b\right)(x,t) & :=\int_{\mathbb{R}^{2}}\mathbb{E}\left[K_{\delta}\left(x-X^{b}(\eta,t)\right)\right]\omega_{0}(\eta)\textrm{d}\eta\\
 & +\int_{0}^{t}\int_{\mathbb{R}^{2}}\mathbb{E}\left[K_{\delta}\left(x-X^{b}(\eta,t)\right)G\left(X^{b}(\eta,s),s\right)\right]\chi_{R}(\eta)\textrm{d}\eta\textrm{d}s
\end{align*}
for all $x\in\mathbb{R}^{2},~t\geq0$. Therefore, the fixed points of the mapping
\begin{equation*}
b\mapsto\mathcal{K}\left(K_{\delta},\chi_{R},\omega_{0},G\right)\diamond b
\end{equation*}
are solutions of \eqref{regularised mf equation}.

Before investigating the limit $\delta\to0$ and $R\to\infty$, we aim to prove the existence and uniqueness of such a fixed point, similar to Qian and Yao \citep{QianYao2022}, using a Banach fixed point argument. To this end we first generalise some bounds that were proven in \citep[Lemma 5 and 6]{QianYao2022}. A key tool here is the upper Aronson estimates. In their generic form they state there exists a constant $M(C_{K},T,\nu)>0$, such that on the time interval $[0,T]$ for all drift functions $b$ with $\sup_{t\leq T,x\in\mathbb{R}^{2}}\left|b(x,t)\right|\leq C_{K}$, it holds that 
\begin{align}
p_{b}\left(\tau,z,t,y\right)\leq\frac{M(C_{K},T,\nu)}{2\pi(t-\tau)}e^{-\frac{|y-z|^{2}}{M(C_{K},T,\nu)(t-\tau)}},~\text{for all}~y,z\in\mathbb{R}^{2},~0\leq\tau\leq t\leq T.\label{rough Aronson estimate}
\end{align}

This bound can be sharpened into the following more precise form 
\begin{align}
p_{b}(\tau,z,t,y)\leq\frac{e^{-\frac{|z-y|^{2}}{2(t-\tau)}}}{2\pi(t-\tau)}\left(1+\kappa(q)\frac{C_{K}}{\sqrt{2\nu}}\left(\sqrt{t-\tau}+\left|z-y\right|\right)e^{\frac{q-1}{2q(t-\tau)}|z-y|^{2}+\frac{C_{K}^{2}(t-\tau)}{4(q-1)\nu}}\right),\label{improved Aronson estimates}
\end{align}
for any $q\in(1,2)$. Hereby $\kappa(q)>0$ is a universal constant only depending on $q$. In what follows, we call \eqref{improved Aronson estimates} the improved Aronson estimate. Based on these estimates, we introduce  
\begin{align}
\bar{\kappa}(q,r) & :=1+\kappa(q)re^{\frac{r^{2}}{2(q-1)}}~q\left(1+\sqrt{\frac{q\pi}{2}}\right),\quad\text{for any}~q\in(1,2),r\geq0.\label{definition kappa_bar}
\end{align}

The following corollary will be useful later on.
\begin{cor}
\label{improved Aronson - integral} Let $b:\mathbb{R}^{2}\times[0,T]\to\mathbb{R}^{2}$
be a bounded function with $\sup_{t\leq T,x\in\mathbb{R}^{2}}\left|b(x,t)\right|\leq C_{K}$.
Then for any $0\leq\tau\leq t\leq T$ and any $y\in\mathbb{R}^{2}$, it holds that 
\begin{align*}
\int_{\mathbb{R}^{2}}p_{b}(\tau,z,t,y)\mathrm{d}z & \leq\int_{\mathbb{R}^{2}}\frac{e^{-\frac{|z-y|^{2}}{2(t-\tau)}}}{2\pi(t-\tau)}\left(1+\kappa(q)\frac{C_{K}}{\sqrt{2\nu}}\left(\sqrt{t-\tau}+\left|z-y\right|\right)\right. \\ & \qquad \left.\times e^{\frac{q-1}{2q(t-\tau)}|z-y|^{2}+\frac{C_{K}^{2}(t-\tau)}{4(q-1)\nu}}\right)\mathrm{d}z=\bar{\kappa}\left(q,\frac{C_{K}\sqrt{t-\tau}}{\sqrt{2\nu}}\right).
\end{align*}
\end{cor}

\begin{proof}
The first inequality is just the improved Aronson estimate \eqref{improved Aronson estimates}. The rest then follows immediately, as $\frac{e^{-\frac{|z-y|^{2}}{2(t-\tau)}}}{2\pi(t-\tau)}$ is the density of the 2D Normal distribution $\mathcal{N}\left(y,(t-\tau)\mathrm{I}\right)$
\begin{align*}
 & \int_{\mathbb{R}^{2}}\frac{e^{-\frac{|z-y|^{2}}{2(t-\tau)}}}{2\pi(t-\tau)}\left(1+\kappa(q)\frac{C_{K}}{\sqrt{2\nu}}\left(\sqrt{t-\tau}+\left|z-y\right|\right)e^{\frac{q-1}{2q(t-\tau)}|z-y|^{2}+\frac{C_{K}^{2}(t-\tau)}{4(q-1)\nu}}\right)\mathrm{d}z\\
 & \qquad =1+\kappa(q)\frac{C_{K}}{\sqrt{2\nu}}e^{\frac{C_{K}^{2}(t-\tau)}{4(q-1)\nu}}\sqrt{t-\tau}~q\int_{\mathbb{R}^{2}}\left(1+\frac{\left|z-y\right|}{\sqrt{t-\tau}}\right)\frac{e^{-\frac{\left|z-y\right|^{2}}{2(t-\tau)q}}}{2\pi(t-\tau)q}\mathrm{d}z\\
 & \qquad =1+\kappa(q)\frac{C_{K}}{\sqrt{2\nu}}e^{\frac{C_{K}^{2}(t-\tau)}{4(q-1)\nu}}\sqrt{t-\tau}~q\left(1+\sqrt{q}\int_{\mathbb{R}^{2}}\left|z\right|\frac{e^{-\frac{\left|z\right|^{2}}{2}}}{2\pi}\mathrm{d}z\right).
\end{align*}

Since the the norm $|Z|$ of a 2D standard Normal distribution $Z\sim\mathcal{N}(0,\mathrm{I})$
is $\chi(2)$-distributed, we have 
\begin{align*}
\int_{\mathbb{R}^{2}}\left|z\right|\frac{e^{-\frac{\left|z\right|^{2}}{2}}}{2\pi}\mathrm{d}z=\mathbb{E}\left[\left|Z\right|\right]=\sqrt{\frac{\pi}{2}},
\end{align*}
which concludes our proof. 
\end{proof}

Furthermore we make a remark with the following change of variables formula for
the transition kernel.
\begin{rem*}
\label{remark - transition probabilities of reflections} For any drift function $c:\mathbb{R}^{2}\times\left[0,\infty\right)\to\mathbb{R}^{2}$ and any $z\in\mathbb{R}^{2}$, define
\begin{align}
c_{z-}(x,t):=-c(z-x,t)~\text{and}~\hat{X}^{c}(\tau,\eta,t):=z-X^{c}(\tau,\eta,t).\label{definition - reflected drift}
\end{align}
Then for $0\leq\tau\leq t$, the process $\hat{X}^{c}$ solves 
\begin{align*}
\hat{X}^{c}(\tau,\eta,t) & =z-X^{c}(\tau,\eta,t)=z-\eta+\int_{\tau}^{t}(-c)(X^{c}(\tau,\eta,s),s)\mathrm{d}s-\sqrt{2\nu}(B_{t}-B_{\tau})\\
 & =(z-\eta)+\int_{\tau}^{t}c_{z-}(\hat{X}^{c}(\tau,\eta,s),s)\mathrm{d}s-\sqrt{2\nu}(B_{t}-B_{\tau}).
\end{align*}
Since $-(B_{t}-B_{\tau})$ is also a standard Brownian motion on $[\tau,+\infty)$, $\hat{X}^{c}$ is a (weak) solution to the SDE with drift $c_{z-}$, and therefore 
\begin{align}
p_{c}(\tau,\eta,t,z-y)=p_{c_{z-}}(\tau,z-\eta,t,y).\label{reflection formula for transition densities}
\end{align}
Note that $\sup_{t\leq T,x\in\mathbb{R}^{2}}\left|c_{z-}(x,t)\right|=\sup_{t\leq T,x\in\mathbb{R}^{2}}\left|c(x,t)\right|$ --- in particular, if $p_{c}$ satisfies the Aronson estimates \eqref{rough Aronson estimate} or \eqref{improved Aronson estimates}, then so does $p_{c_{z-}}$! 
\end{rem*}

This lets us prove the following key estimate.
\begin{lem}\label{Lemma -abstract bounds of convolution integral} 
Let $b:\mathbb{R}^{2}\times[0,T]\to\mathbb{R}^{2}$ be a bounded function with $\sup_{t\leq T,x\in\mathbb{R}^{2}}\left|b(x,t)\right|\leq C_{K}$, and $\gamma\in(1,2)$. Furthermore let $f:\mathbb{R}^{2}\times\mathbb{R}^{2}\to\mathbb{R}$ be a bounded function and define the seminorm 
\begin{align*}
\left\Vert f\right\Vert _{\mathrm{diag}_{1}}:=\sup_{a\in\mathbb{R}^{2}}\int_{\mathbb{R}^{2}}\left|f(z-a,z)\right|\mathrm{d}z.
\end{align*}
Let $\Theta:\mathbb{R}^{2}\times\mathbb{R}^{2}\to\mathbb{R}^{2}$
be such that 
\begin{align*}
\left|\Theta(x,y)\right|\leq\frac{\epsilon(x-y)}{|x-y|^{\gamma}}
\end{align*}
for some function $\epsilon:\mathbb{R}\to(0,\infty)$. Then for
any $q\in(1,2)$,
\begin{itemize}
\item the following inequality holds for all $x\in\mathbb{R}^{2}$, $0\leq\tau\leq t\leq T$
\begin{align}
\begin{split} & \int_{\mathbb{R}^{2}}\int_{\mathbb{R}^{2}}\left|\Theta(x,y)\right|~\left|f(y,z)\right|p_{b}(\tau,z,t,y)\mathrm{d}y\mathrm{d}z\\
 & \leq\bar{\kappa}\left(q,\frac{C_{K}\sqrt{t-\tau}}{\sqrt{2\nu}}\right)\left(\int_{|z|\leq1}\frac{\epsilon(z)}{|z|^{\gamma}}\mathrm{d}z\left\Vert f\right\Vert _{\infty}+\left(\sup_{|z|\geq1}\frac{\epsilon(z)}{|z|^{\gamma}}\right)\left\Vert f\right\Vert _{\mathrm{diag}_{1}}\right)\text{,}
\end{split}
\label{general a priori estimate - abstract kernel}
\end{align}
\item consequently, when $\left|\Theta(x,y)\right|\leq|x-y|^{-\gamma}$,
we have 
\begin{align}
\begin{split}
\int_{\mathbb{R}^{2}}\int_{\mathbb{R}^{2}}\frac{\left|f(y,z)\right|}{\left|y-x\right|^{\gamma}}p_{b}(\tau,z,t,y)\mathrm{d}z\mathrm{d}y\leq&\left(\frac{2\pi}{2-\gamma}\left\Vert f\right\Vert _{\infty}+\left\Vert f\right\Vert _{\mathrm{diag}_{1}}\right) \\ & \qquad\times \bar{\kappa}\left(q,\frac{C_{K}\sqrt{t-\tau}}{\sqrt{2\nu}}\right).\label{general a priori estimate - Biot savart kernel}
\end{split}
\end{align}
\end{itemize}
\end{lem}
\begin{proof}
To carry out the estimate, we split up the inner integral into a part close to the singularity and a part that is bounded away from the singular points $\left\{ y=x\right\} $. The integral about the singularity can be bounded as in Qian and Yao \citep{QianYao2022} in the following way
\begin{align*}
 & \int_{\mathbb{R}^{2}}\int_{\left\{ |y-x|<1\right\} }\left|f(y,z)\right|\left|\Theta(x,y)\right|p_{b}(\tau,z,t,y)\mathrm{d}y\mathrm{d}z\\
 & \leq\int_{\mathbb{R}^{2}}\int_{\left\{ |y-x|<1\right\} }\frac{\left|f(y,z)\right|\epsilon(y-x)}{\left|y-x\right|^{\gamma}}\frac{e^{-\frac{|z-y|^{2}}{2(t-\tau)}}}{2\pi(t-\tau)}\\
 & \qquad\times\left(1+\kappa(q)\frac{C_{K}}{\sqrt{2\nu}}\left(\sqrt{t-\tau}+\left|z-y\right|\right)e^{\frac{q-1}{2q(t-\tau)}|z-y|^{2}+\frac{C_{K}^{2}(t-\tau)}{4(q-1)\nu}}\right)\mathrm{d}y~\mathrm{d}z\\
 & =\int_{\mathbb{R}^{2}}\int_{\left\{ |y|<1\right\} }\frac{\left|f(y+x,y+x-z)\right|\epsilon(y)}{\left|y\right|^{\gamma}}\frac{e^{-\frac{|z|^{2}}{2(t-\tau)}}}{2\pi(t-\tau)}\\
 & \qquad\times\left(1+\kappa(q)\frac{C_{K}}{\sqrt{2\nu}}\left(\sqrt{t-\tau}+\left|z\right|\right)e^{\frac{q-1}{2q(t-\tau)}|z|^{2}+\frac{C_{K}^{2}(t-\tau)}{4(q-1)\nu}}\right)\mathrm{d}y\mathrm{d}z\\
 & \leq\left\Vert f\right\Vert _{\infty}~\int_{\left\{ |y|<1\right\} }\frac{\epsilon(y)}{\left|y\right|^{\gamma}}\mathrm{d}y~\bar{\kappa}\left(q,\frac{C_{K}\sqrt{t-\tau}}{\sqrt{2\nu}}\right).
\end{align*}

To estimate the integral away from the singular points, we note that by the change of variables mentioned above and the improved Aronson estimate \eqref{improved Aronson estimates} one has
\begin{align*}
 & \int_{\mathbb{R}^{2}}\int_{\left\{ |y-x|\geq1\right\} }\left|f(y,z)\right|\left|\Theta(x,y)\right|p_{b}(\tau,z,t,y)\mathrm{d}y\mathrm{d}z \\ 
 & \leq\left(\sup_{|z|\geq1}\frac{\epsilon(z)}{|z|^{\gamma}}\right)\int_{\mathbb{R}^{2}}\int_{\mathbb{R}^{2}}\left|f(y,z)\right|p_{b}(\tau,z,t,y)\mathrm{d}y\mathrm{d}z\\
 & =\left(\sup_{|z|\geq1}\frac{\epsilon(z)}{|z|^{\gamma}}\right)\int_{\mathbb{R}^{2}}\int_{\mathbb{R}^{2}}\left|f(z-y,z)\right|p_{b_{z-}}(\tau,0,t,y)\mathrm{d}y\mathrm{d}z\\
 & \leq\left(\sup_{|z|\geq1}\frac{\epsilon(z)}{|z|^{\gamma}}\right)\int_{\mathbb{R}^{2}}\int_{\mathbb{R}^{2}}\left|f(z-y,z)\right|\mathrm{d}z\frac{e^{-\frac{|y|^{2}}{2(t-\tau)}}}{2\pi t} \\ 
 & \qquad \times\left(1+\kappa\frac{C_{K}}{\sqrt{2\nu}}\left(\sqrt{t-\tau}+\left|y\right|\right)e^{\frac{q-1}{2q(t-\tau)}|y|^{2}+\frac{C_{K}^{2}t}{4(q-1)\nu}}\right)\mathrm{d}y\\
 & \leq\left(\sup_{|z|\geq1}\frac{\epsilon(z)}{|z|^{\gamma}}\right)\left\Vert f\right\Vert _{\mathrm{diag}_{1}}\bar{\kappa}\left(q,\frac{C_{K}\sqrt{t-\tau}}{\sqrt{2\nu}}\right),
\end{align*}
which proves \eqref{general a priori estimate - abstract kernel}. The inequality \eqref{general a priori estimate - Biot savart kernel} is then just a trivial consequence of \eqref{general a priori estimate - abstract kernel} where we take $\epsilon \equiv 1$. 
\end{proof}
Next we show a key Lipschitz estimate for the operator $\mathcal{K}\left(K_{\delta},\chi_{R},\omega_{0},G\right)\diamond$.
\begin{thm}
\label{Theorem - abstract Lipschitz estimate} Assume we are given
some time frame $T_{K}\geq0$ and some upper bound for the drift $C_{K}\geq0$.
Let the kernel function $\Theta:\mathbb{R}^{2}\times\mathbb{R}^{2}\to\mathbb{R}^{2}$
be such that for some $\gamma\in[0,2)$ and some constant $C_{\Theta}>0$, one has
\begin{align}
\left|\Theta(x,y)\right| & \leq\frac{C_{\Theta}}{|x-y|^{\gamma}}.\label{kernel condition - in Theorem}
\end{align}
Let furthermore $f:\mathbb{R}^{2}\to\mathbb{R}$ be bounded, such
that 
\begin{align}
\vertiii{f}_{T_{K}}:=\sup_{t\leq T_{K}}\int_{\mathbb{R}^{2}}\sqrt{\mathbb{E}\left[\left|f(z-\sqrt{2\nu}B_{t},z)\right|^{2}\right]}\mathrm{d}z<\infty~\text{and}~\left\Vert f\right\Vert _{\mathrm{diag}_{1}}<+\infty.\label{integral conditions}
\end{align}
Finally we set for $\alpha,\beta>1$ with $\frac{1}{\alpha}+\frac{1}{\beta}=1$
and $\alpha\gamma<2$ 
\begin{align}
\begin{split}C_{L}(R,T,f) & :=C_{\Theta}\bar{\kappa}\left(q,R\right)\left(\frac{2\pi}{2-\gamma}\left\Vert f\right\Vert _{\infty}+\left\Vert f\right\Vert _{\mathrm{diag}_{1}}\right)\frac{R\sqrt{T}}{4\nu}+C_{\Theta}e^{\frac{\alpha-1}{4\nu}R^{2}}\\
 & \qquad\times\left(\bar{\kappa}\left(q,\alpha R\right)\left(\frac{2\pi}{2-\alpha\gamma}\left\Vert f\right\Vert _{\infty}+\left\Vert f\right\Vert _{\mathrm{diag}_{1}}\right)\right)^{1/\alpha}\frac{\sqrt{T}\vertiii{f}_{T}^{1/\beta}}{\sqrt{\nu}}.
\end{split}
\label{Definition C_L}
\end{align}
Then for all times $T\leq T_{K}$ and any drift functions $b,\tilde{b}:\mathbb{R}^{2}\times[0,T_{K}]\to\mathbb{R}^{2}$
with $\left\Vert b\right\Vert _{\infty},\left\Vert \tilde{b}\right\Vert _{\infty}\leq C_{K}$,
we have 
\begin{align}
\begin{split}
&\sup_{x\in\mathbb{R}^{2},t\leq T}  \left|\int_{\mathbb{R}^{2}}\int_{\mathbb{R}^{2}}\Theta(x,y)\left(p_{b}\left(0,z,t,y\right)-p_{\tilde{b}}\left(0,z,t,y\right)\right)f(y,z)\mathrm{d}y\mathrm{d}z\right|\\
 &\qquad\leq C_{L}\left(\frac{C_{K}\sqrt{T}}{\sqrt{2\nu}},T,f\right)\left(\sqrt{T}+T\right)\sup_{x\in\mathbb{R}^{2},t\leq T}\left|b(x,t)-\tilde{b}(x,t)\right|.
\end{split}
\label{abstract local Lipschitz estimate}
\end{align}
\end{thm}

\begin{rem*}\label{Remark - admissible f} 
Note that for integrable $f$, only depending on the second variable $z$, i.e. $f(y,z)=f(z)$, it holds that 
\begin{align*}
\left\Vert f\right\Vert _{\mathrm{diag}_{1}}=\vertiii{f}=\left\Vert f\right\Vert _{1}<\infty,
\end{align*}
thus, Theorem \ref{Theorem - abstract Lipschitz estimate} is a true
generalisation of the result of Qian and Yao \citep{QianYao2022}. For $f$ only
depending on the first variable, i.e. $f(y,z)=f(y)$, we also have
$\left\Vert f\right\Vert _{\mathrm{diag}_{1}}=\left\Vert f\right\Vert _{1}$.
Assume furthermore that $f\in C_{0}^{0}\left(\mathbb{R};\mathbb{R}\right)$
has compact support. Let $R:=\sup_{z\in\mathrm{spt}f}|z|$, then for
all $z\in\mathbb{R}^{2}$, with $\left|z\right|\geq R$ 
\begin{align*}
\mathbb{E}\left[\left|f\left(z-\sqrt{2\nu}B_{t}\right)\right|^{2}\right]\leq\left\Vert f\right\Vert _{\infty}^{2}\mathbb{P}\left(\left|\sqrt{2\nu}B_{t}\right|\geq\left|z\right|-R\right)&\\
\leq\left\Vert f\right\Vert _{\infty}^{2}\mathbb{E}\left[e^{\sqrt{2\nu}\left|B_{t}\right|}\right]~e^{-\left|z\right|+R}\leq\left\Vert f\right\Vert _{\infty}^{2}\frac{e^{8\nu t+R}}{2}e^{-|z|},&
\end{align*}
where the last inequality can be derived by using Young's inequality
and the moment generating function of the $\chi^{2}$-distribution.
Since $\int_{\mathbb{R}^{2}}e^{-|z|/2}\mathrm{d}z=8\pi$, we thus
derive 
\begin{align*}
\vertiii{f}_{T_{K}} &=\sup_{t\leq T_{K}}\int_{\mathbb{R}^{2}}\sqrt{\mathbb{E}\left[\left|f(z-\sqrt{2\nu}B_{t},z)\right|^{2}\right]}\mathrm{d}z \\ &=\left\Vert f\right\Vert _{\infty}\left(R^{2}\pi+\sup_{t\leq T_{K}}\frac{e^{4\nu t+R/2}}{\sqrt{2}}8\pi\right)\\
&\leq\left\Vert f\right\Vert _{\infty}\left(\sup_{z\in\mathrm{spt}f}|z|^{2}\pi+\frac{8\pi e^{4\nu T_{k}+\frac{\sup_{z\in\mathrm{spt}f}|z|}{2}}}{\sqrt{2}}\right).
\end{align*}
Thus for any function $f$ only depending on the the first variable,
the conditions of Theorem \ref{Theorem - abstract Lipschitz estimate} are also satisfied if it has compact support. In particular both $\omega_{0}$ and $G(\cdot,t)$ for every finite $t\geq0$ are appropriate choices for $f$. 
\end{rem*}
For $f$ only depending on $y$, but not on $z$, a version of Theorem
\ref{Theorem - abstract Lipschitz estimate} was proven by Qian and
Yao in \citep{QianYao2022}. A key tool for their analysis, that we also use in our generalization, is the Cameron-Martin theorem,
which, for the convenience of the reader, we summarize in the following
remark.
\begin{rem*}
\label{Cameron-Martin theorem} The Cameron--Martin theorem determines
the Radon--Nikodym derivative of a diffusion measure w.r.t. the Wiener
measure on a finite horizon path space, i.e. for any $t\geq0$ and
any drift function $c$, the law of $\left(X^{c}(z,s)\right)_{0\leq s\leq t}$, denoted by $Q^{z,t}$, is absolutely continuous w.r.t. the law $\mathbb{P}^{B}$ of $(B_{s})_{0\leq s\leq t}$ (the Wiener measure) and its Radon--Nikodym derivative is given by 
\begin{align*}
\frac{\mathrm{d}Q^{z,t}}{\mathrm{d}\mathbb{P}^{B}}=R_{c}(z,t):=\exp\left(N_{c}(z,t)\right),
\end{align*}
where the semimartingale $N_{c}(z,t)$ is of the form 
\begin{align*}
N_{c}(z,t):=\frac{1}{\sqrt{2\nu}}\int_{0}^{t}c(z+\sqrt{2\nu}B_{r},r)\mathrm{d}B_{r}-\frac{1}{4\nu}\int_{0}^{t}\left|c(z+\sqrt{2\nu}B_{r},r)\right|^{2}\mathrm{d}r.
\end{align*}
Therefore we have for any function $g$ 
\begin{align*}
\mathbb{E}\left[g(X^{c}(t,z))\right]=\int_{\mathbb{R}^{2}}g(y)p_{c}(0,z,t,y)\mathrm{d}y=\mathbb{E}\left[g(\sqrt{2\nu}B_{t})R_{c}(z,t)\right].
\end{align*}
\end{rem*}
With this in mind let us now prove Theorem \ref{Theorem - abstract Lipschitz estimate}. 
\begin{proof}[Proof of Theorem \ref{Theorem - abstract Lipschitz estimate}]
We can not directly use the approach of \citep[Lemma 8]{QianYao2022},
since $f$ may not be integrable with respect to $z$. Therefore we
make a change of variables to note that 
\begin{align*}
 & \left|\int_{\mathbb{R}^{2}}\int_{\mathbb{R}^{2}}\Theta(x,y)\left(p_{b}\left(0,z,t,y\right)-p_{\tilde{b}}\left(0,z,t,y\right)\right)f(y,z)\mathrm{d}y\mathrm{d}z\right|\\
 & =\left|\int_{\mathbb{R}^{2}}\left(\int_{\mathbb{R}^{2}}\Theta(x,z-y)p_{b}\left(0,z,t,z-y\right)f(z-y,z)\mathrm{d}y\right.\right.\\
 & \qquad-\left.\left.\int_{\mathbb{R}^{2}}\Theta(x,z-y)p_{\tilde{b}}\left(0,z,t,z-y\right)f(z-y,z)\mathrm{d}y\right)\mathrm{d}z\right|\\
 & =\left|\int_{\mathbb{R}^{2}}\left(\int_{\mathbb{R}^{2}}\Theta(x,z-y)p_{b_{z-}}\left(0,0,t,y\right)f(z-y,z)\mathrm{d}y\right.\right.\\
 & \qquad-\left.\left.\int_{\mathbb{R}^{2}}\Theta(x,z-y)p_{\tilde{b_{z-}}}\left(0,0,t,y\right)f(z-y,z)\mathrm{d}y\right)\mathrm{d}z\right|
\end{align*}
where the last identity is a consequence of \eqref{reflection formula for transition densities}. Therefore the Cameron--Martin theorem gives us 
\begin{align}
\begin{split} & \left|\int_{\mathbb{R}^{2}}\int_{\mathbb{R}^{2}}\Theta(x,y)\left(p_{b}\left(0,z,t,y\right)-p_{\tilde{b}}\left(0,z,t,y\right)\right)f(y,z)\mathrm{d}y\mathrm{d}z\right|\\
 & =\left|\int_{\mathbb{R}^{2}}\mathbb{E}\left[\Theta(x,z-\sqrt{2\nu}B_{t})\left(R_{b_{z-}}(0,t)-R_{\tilde{b}_{z-}}(0,t)\right)f(z-\sqrt{2\nu}B_{t},z)\right]\mathrm{d}z\right|.
\end{split}
\label{Lipschitz estimate in Cameron Martin form}
\end{align}

We define 
\begin{align*}
M(z,t) & :=\frac{1}{\sqrt{2\nu}}\int_{0}^{t}\left(b_{z-}(\sqrt{2\nu}B_{r},r)-\tilde{b}_{z-}(\sqrt{2\nu}B_{r},r)\right)\mathrm{d}B_{r}\\
A(z,t) & :=-\frac{1}{4\nu}\int_{0}^{t}\left|b_{z-}(\sqrt{2\nu}B_{r},r)\right|^{2}-\left|\tilde{b}_{z-}(\sqrt{2\nu}B_{r},r)\right|^{2}\mathrm{d}r,
\end{align*}
and note that the following holds for the finite variation
part $A$ 
\begin{equation*}
\sup_{z\in\mathbb{R}^{2},s\leq t}\left|A(z,s)\right| =\sup_{z\in\mathbb{R}^{2},s\leq t}\left|\frac{1}{4\nu}\int_{0}^{s}\left|b_{z-}(\sqrt{2\nu}B_{r},r)\right|^{2}-\left|\tilde{b}_{z-}(\sqrt{2\nu}B_{r},r)\right|^{2}\mathrm{d}r\right|
\end{equation*}
which implies
\begin{equation}
\sup_{z\in\mathbb{R}^{2},s\leq t}\left|A(z,s)\right| \leq\frac{C_{K}t}{4\nu}\left\Vert b-\tilde{b}\right\Vert _{\infty},
\label{estimate - finite variation part}
\end{equation}
and for the local martingale $M$
\begin{align}
\sup_{z\in\mathbb{R}^{2}}\left\langle M(z,t)\right\rangle _{t} & \leq\frac{t}{2\nu}\left\Vert b-\tilde{b}\right\Vert _{\infty}^{2}.\label{estimate - quadratic variation}
\end{align}

This then gives us the following estimate 
\begin{align*}
 & \left|\int_{\mathbb{R}^{2}}\mathbb{E}\left[\Theta(x,z-\sqrt{2\nu}B_{t})\left(R_{b_{z-}}(0,t)-R_{\tilde{b}_{z-}}(0,t)\right)f(z-\sqrt{2\nu}B_{t},z)\right]\mathrm{d}z\right|\\
 & \leq\int_{\mathbb{R}^{2}}\mathbb{E}\left[\left|\Theta(x,z-\sqrt{2\nu}B_{t})\right|\int_{0}^{1}R_{\lambda}(z,t)\mathrm{d}\lambda\left(\left|A(z,t)\right|+\left|M(z,t)\right|\right)\left|f(z-\sqrt{2\nu}B_{t},z)\right|\right]\mathrm{d}z,
\end{align*}
where we set $R_{\lambda}(z,t):=\exp\left(\lambda N_{b_{z-}}(0,t)+(1-\lambda)N_{\tilde{b}_{z-}}(0,t)\right)$.
Next we note that if we define $b_{\lambda}:=\lambda b+(1-\lambda)\tilde{b}$
and $b_{\lambda,z-}:=\lambda b_{z-}+(1-\lambda)\tilde{b}_{z-}$, then
\begin{align}
R_{\lambda}(z,t)=R_{b_{\lambda,z-}}(0,t)~e^{-\frac{\lambda(1-\lambda)}{4\nu}\int_{0}^{t}\left|b_{z-}(\sqrt{2\nu}B_{r},r)-b_{z-}(\sqrt{2\nu}B_{r},r)\right|^{2}\mathrm{d}r}\leq R_{b_{\lambda,z-}}(0,t),\label{inequality - convex estimate Girsanov density}
\end{align}
and therefore 
\begin{align}
\begin{split} & \left|\int_{\mathbb{R}^{2}}\int_{\mathbb{R}^{2}}\Theta(x,y)\left(p_{b}\left(0,z,t,y\right)-p_{\tilde{b}}\left(0,z,t,y\right)\right)f(y,z)\mathrm{d}y\mathrm{d}z\right|\\
 & \leq\int_{0}^{1}\int_{\mathbb{R}^{2}}\mathbb{E}\left[\left|\Theta(x,z-\sqrt{2\nu}B_{t})\right|R_{b_{\lambda,z-}}(0,t)\right. \\ &\qquad\left.\times\left(\left|A(z,t)\right|+\left|M(z,t)\right|\right)\left|f(z-\sqrt{2\nu}B_{t},z)\right|\right]\mathrm{d}z\mathrm{d}\lambda.
\end{split}
\label{Lipschitz inequality - Cameron martin form}
\end{align}

We estimate the first term, involving the finite variation part $A$
of the differences in the Girsanov-log-likelihoods, using the Cameron--Martin
theorem and Lemma \ref{Lemma -abstract bounds of convolution integral}
\begin{align}
\begin{split} & \int_{\mathbb{R}^{2}}\mathbb{E}\left[\left|\Theta(x,z-\sqrt{2\nu}B_{t})\right|R_{b_{\lambda,z-}}(0,t)\left|A(z,t)\right|\left|f(z-\sqrt{2\nu}B_{t},z)\right|\right]\mathrm{d}z\\
 & \leq\frac{C_{K}t}{4\nu}\left\Vert b-\tilde{b}\right\Vert _{\infty}~\int_{\mathbb{R}^{2}}\mathbb{E}\left[\left|\Theta(x,z-\sqrt{2\nu}B_{t})\right|R_{b_{\lambda,z-}}(0,t)\left|f(z-\sqrt{2\nu}B_{t},z)\right|\right]\mathrm{d}z\\
 & =\frac{C_{K}t}{4\nu}\left\Vert b-\tilde{b}\right\Vert _{\infty}~\int_{\mathbb{R}^{2}}\int_{\mathbb{R}^{2}}\left|\Theta(x,z-y)\right|p_{b_{\lambda,z-}}(0,0,t,y)\left|f(z-y,z)\right|\mathrm{d}y\mathrm{d}z\\
 & \leq C_{\theta}\bar{\kappa}\left(q,\frac{C_{K}\sqrt{t}}{\sqrt{2\nu}}\right)\left(\frac{2\pi}{2-\gamma}\left\Vert f\right\Vert _{\infty}+\left\Vert f\right\Vert _{\mathrm{diag}_{1}}\right)\frac{C_{K}t}{4\nu}\left\Vert b-\tilde{b}\right\Vert _{\infty}.
\end{split}
\label{Lipschitz proof - finite variation part}
\end{align}

Now we use the basic Holder inequality for $\alpha,\beta>1$, such
that $\frac{1}{\alpha}+\frac{1}{\beta}=1$ and $\alpha\gamma<2$,
to derive 
\begin{align*}
 & \int_{\mathbb{R}^{2}}\mathbb{E}\left[\left|\Theta(x,z-\sqrt{2\nu}B_{t})\right|R_{b_{\lambda,z-}}(0,t)\left|M(z,t)\right|\left|f(z-\sqrt{2\nu}B_{t},z)\right|\right]\mathrm{d}z\\
 & \leq\left(\int_{\mathbb{R}^{2}}\mathbb{E}\left[\left|\Theta(x,z-\sqrt{2\nu}B_{t})\right|^{\alpha}R_{b_{\lambda,z-}}(0,t)^{\alpha}\left|f(z-\sqrt{2\nu}B_{t},z)\right|\right]\mathrm{d}z\right)^{1/\alpha}\\
 & \qquad\times\left(\int_{\mathbb{R}^{2}}\mathbb{E}\left[\left|M(z,t)\right|^{\beta}\left|f(z-\sqrt{2\nu}B_{t},z)\right|\right]\mathrm{d}z\right)^{1/\beta}.
\end{align*}

The first term can be estimated as follows. Note that 
\begin{align*}
R_{b_{\lambda,z-}}(0,t)^{\alpha}=R_{\alpha~b_{\lambda,z-}}(0,t)e^{\frac{\left(\alpha-1\right)\alpha}{4\nu}\int_{0}^{t}\left|b_{\lambda,z-}(r,z+\sqrt{2\nu}B_{r})\right|^{2}\mathrm{d}r}\leq R_{\alpha~b_{\lambda,z-}}(0,t)e^{\frac{\left(\alpha-1\right)\alpha}{4\nu}C_{K}^{2}t},
\end{align*}
which gives us 
\begin{align*}
 & \left(\int_{\mathbb{R}^{2}}\mathbb{E}\left[\left|\Theta(x,z-\sqrt{2\nu}B_{t})\right|^{\alpha}R_{b_{\lambda,z-}}(0,t)^{\alpha}\left|f(z-\sqrt{2\nu}B_{t},z)\right|\right]\mathrm{d}z\right)^{1/\alpha}\\
 & \leq e^{\frac{\alpha-1}{4\nu}C_{K}^{2}t}\left(\int_{\mathbb{R}^{2}}\mathbb{E}\left[\left|\Theta(x,z-\sqrt{2\nu}B_{t})\right|^{\alpha}R_{\alpha b_{\lambda,z-}}(0,t)\left|f(z-\sqrt{2\nu}B_{t},z)\right|\right]\mathrm{d}z\right)^{1/\alpha}\\
 & =e^{\frac{\alpha-1}{4\nu}C_{K}^{2}t}\left(\int_{\mathbb{R}^{2}}\int_{\mathbb{R}^{2}}\left|\Theta(x,z-y)\right|^{\alpha}p_{\alpha b_{\lambda,z-}}(0,0,t,y)\left|f(z-y,z)\right|\mathrm{d}z\right)^{1/\alpha}\\
 & =e^{\frac{\alpha-1}{4\nu}C_{K}^{2}t}\left(\int_{\mathbb{R}^{2}}\int_{\mathbb{R}^{2}}\left|\Theta(x,z-y)\right|^{\alpha}p_{\alpha b_{\lambda}}(0,z,t,z-y)\left|f(z-y,z)\right|\mathrm{d}y\mathrm{d}z\right)^{1/\alpha}.
\end{align*}
A change of variables and Lemma \ref{Lemma -abstract bounds of convolution integral},
which is admissible as $\alpha\gamma<2$, gives us 
\begin{align}
\begin{split} & \left(\int_{\mathbb{R}^{2}}\mathbb{E}\left[\left|\Theta(x,z-B_{t})\right|^{\alpha}R_{b_{\lambda,z-}}(0,t)^{\alpha}\left|f(z-\sqrt{2\nu}B_{t},z)\right|\right]\mathrm{d}z\right)^{1/\alpha}\\
 & \leq C_{\theta}e^{\frac{\alpha-1}{4\nu}C_{K}^{2}t}\left(\int_{\mathbb{R}^{2}}\int_{\mathbb{R}^{2}}\left|x-y\right|^{-\alpha\gamma}p_{\alpha b_{\lambda}}(0,z,t,y)\left|f(y,z)\right|\mathrm{d}y\mathrm{d}z\right)^{1/\alpha}\\
 & \leq C_{\theta}e^{\frac{\alpha-1}{4\nu}C_{K}^{2}t}\left(\bar{\kappa}\left(q,\frac{\alpha C_{K}\sqrt{t}}{\sqrt{2\nu}}\right)\left(\frac{2\pi}{2-\alpha\gamma}\left\Vert f\right\Vert _{\infty}+\left\Vert f\right\Vert _{\mathrm{diag}_{1}}\right)\right)^{1/\alpha}.
\end{split}
\label{Lipschitz proof - martingale part 1}
\end{align}

For the second term we again use Cauchy-Schwarz in combination with the BDG inequality to note that 
\begin{align*}
\begin{split} & \left(\int_{\mathbb{R}^{2}}\mathbb{E}\left[\left|M(z,t)\right|^{\beta}\left|f(z-\sqrt{2\nu}B_{t},z)\right|\right]\mathrm{d}z\right)^{1/\beta}\\
 & \leq\left(\int_{\mathbb{R}^{2}}\sqrt{\mathbb{E}\left[\left\langle M(z,t)\right\rangle ^{\beta}\right]}\sqrt{\mathbb{E}\left[\left|f(z-\sqrt{2\nu}B_{t},z)\right|^{2}\right]}\mathrm{d}z\right)^{1/\beta}\\
 & \leq\frac{\sqrt{t}}{\sqrt{\nu}}\left\Vert b-\tilde{b}\right\Vert _{\infty}\left(\int_{\mathbb{R}^{2}}\sqrt{\mathbb{E}\left[\left|f(z-\sqrt{2\nu}B_{t},z)\right|^{2}\right]}\mathrm{d}z\right)^{1/\beta},
\end{split}
\end{align*}
which, due to the definition of $\vertiii{\cdot}_{T_{K}}$, means that
\begin{equation}\left(\int_{\mathbb{R}^{2}}\mathbb{E}\left[\left|M(z,t)\right|^{\beta}\left|f(z-\sqrt{2\nu}B_{t},z)\right|\right]\mathrm{d}z\right)^{1/\beta}\leq\frac{\sqrt{t}\vertiii{f}_{T_{K}}^{1/\beta}}{\sqrt{\nu}}\left\Vert b-\tilde{b}\right\Vert _{\infty}.
\label{Lipschitz proof - martingale part 2}
\end{equation}

Combining \eqref{Lipschitz inequality - Cameron martin form} with
\eqref{Lipschitz proof - finite variation part}, \eqref{Lipschitz proof - martingale part 1}
and \eqref{Lipschitz proof - martingale part 2}, we can conclude
our proof. 
\end{proof}
As Lemma \ref{Lemma -abstract bounds of convolution integral} gives
a local Lipschitz property of $\mathcal{K}\left(K_{\delta},\chi_{R},\omega_{0},G\right)\diamond\cdot$,
we need to localise the domain of the fixed point argument, i.e. we
need to show that, at least for some timeframe $T_{K}$, $\mathcal{K}\left(K_{\delta},\chi_{R},\omega_{0},G\right)\diamond b$
has the same uniform bounds as $b$. Define 
\begin{align*}
C_{K} :=C_{0}~\bar{\kappa}\left(3/2,1\right)\left(2\pi\left\Vert \omega_{0}\right\Vert _{\infty}+\left\Vert \omega_{0}\right\Vert _{1}\right)+C_{0}~\bar{\kappa}\left(3/2,1\right)^{2}& \\ \times\sup_{s\geq0}2\pi\left(\left\Vert G(\cdot,s)\right\Vert _{\infty}+\left\Vert G(\cdot,s)\right\Vert _{1}\right)&,
\end{align*}
and
\begin{equation*}
T_{K} :=\min\left\{ \frac{2\nu}{C_{K}(T)^{2}},1\right\}.
\end{equation*}
Then the following Lemma holds.
\begin{lem}
\label{Lemma - boundedness proof} Assume that the drift $b$ satisfies
$\left|b(x,t)\right|\leq C_{K}$ for all $x\in\mathbb{R}^{2},t\in[0,T_{K}]$.
Then for arbitrary $\delta\geq0$ and $R\geq0$ it holds that 
\begin{equation*}
\left|\left(\mathcal{K}\left(K_{\delta},\chi_{R},\omega_{0},G\right)\diamond b\right)(x,t)\right|\leq C_{K}.
\end{equation*}
\end{lem}

\begin{proof}
Note that for $G=0$ and $\delta=0$ this was proven in Qian and Yao \citep{QianYao2022},
we generalise this result using Lemma \ref{Lemma -abstract bounds of convolution integral}.
Since $\left|K_{\delta}(z)\right|\leq C_{0}\left|z\right|^{-1}$,
one can deduce by Lemma \ref{Lemma -abstract bounds of convolution integral}
with $q:=3/2$ that
\begin{align}
\begin{split}
\bigg|\int_{\mathbb{R}^{2}}\mathbb{E}\left[K_{\delta}\left(x-X^{b}(\eta,t)\right)\right]&\omega_{0}(\eta)\textrm{d}\eta\bigg|\leq C_{0}\int_{\mathbb{R}^{2}}\int_{\mathbb{R}^{2}}\frac{1}{\left|y-x\right|}p_{b}(0,z,t,y)\left|\omega_{0}(z)\right|\textrm{d}y\textrm{d}z\\
 & \leq C_{0}\bar{\kappa}\left(3/2,C_{K}\sqrt{t}/\sqrt{2\nu}\right)\left(2\pi\left\Vert \omega_{0}\right\Vert _{\infty}+\left\Vert \omega_{0}\right\Vert _{1}\right)
\\ &=C_{0}\bar{\kappa}\left(3/2,1\right)\left(2\pi\left\Vert \omega_{0}\right\Vert _{\infty}+\left\Vert \omega_{0}\right\Vert _{1}\right),
\end{split}
\label{homogeneous uniform bound}
\end{align}
where the last inequality is a consequence of $t\leq T_{K}$ and $\omega_{0}$ only depending on the initial condition.

Next we note that by the Markov property of the transition kernel
$p_{b}$ of $X^{b}$ 
\begin{align*}
 & \left|\int_{0}^{t}\int_{\mathbb{R}^{2}}\mathbb{E}\left[K_{\delta}\left(x-X^{b}(\eta,t)\right)G\left(X^{b}(\eta,s),s\right)\right]\chi_{R}(\eta)\textrm{d}\eta\textrm{d}s\right|\\
 & \leq\int_{0}^{t}\int_{\mathbb{R}^{2}}\mathbb{E}\left[\left|K_{\delta}\left(x-X^{b}(\eta,t)\right)\right|\left|G\left(X^{b}(\eta,s),s\right)\right|\right]|\chi_{R}(\eta)|\textrm{d}\eta\textrm{d}s\\
 & \leq\left\Vert \chi_{R}\right\Vert _{\infty}\int_{0}^{t}\int_{\mathbb{R}^{2}}\int_{\mathbb{R}^{2}}\int_{\mathbb{R}^{2}}\left|K_{\delta}\left(y-x\right)\right|p_{b}\left(s,z,t,y\right)\textrm{d}y \\ & \qquad\qquad\qquad\qquad\qquad\times\left|G\left(z,s\right)\right|p_{b}\left(0,\eta,s,z\right)\textrm{d}z\textrm{d}\eta\textrm{d}s.
\end{align*}

Then we note that only the term $p_{b}\left(0,\eta,s,\zeta\right)$
in the last series of integrals depends on $\eta$ and can thus be
integrated out. By Corollary \ref{improved Aronson - integral}, we
have 
\begin{align*}
\int_{\mathbb{R}^{2}}p_{b}\left(0,\eta,s,z\right)\mathrm{d}\eta & \leq\bar{\kappa}\left(3/2,C_{K}\sqrt{s}/\sqrt{2\nu}\right)\leq\bar{\kappa}\left(3/2,1\right).
\end{align*}

Thus we finally obtain by Lemma \ref{Lemma -abstract bounds of convolution integral}
and the fact that $s\leq t\leq T_{K}=\min\left\{ \frac{2\nu}{C_{K}(T)^{2}},1\right\} $
\begin{align*}
 \left|\int_{0}^{t}\int_{\mathbb{R}^{2}}\mathbb{E}\left[K_{\delta}\left(x-X^{b}(\eta,t)\right)G\left(X^{b}(\eta,s),s\right)\right]\chi_{R}(\eta)\textrm{d}\eta\textrm{d}s\right|&\\
 \leq C_{0}\bar{\kappa}\left(3/2,1\right)^{2}\int_{0}^{1}\left(2\pi\left\Vert G(\cdot,s)\right\Vert _{\infty}+\left\Vert G(\cdot,s)\right\Vert _{1}\right)\textrm{d}s&
\end{align*}

Since $t\leq T_{K}\leq T$, this lets us conclude that $\left|\left(\mathcal{K}\left(K_{\delta},\chi_{R},\omega_{0},G\right)\diamond b\right)(x,t)\right|\leq C_{K}$. 
\end{proof}
\begin{rem*}
\label{a priori bound and timeframe independent of delta and chi}
Note that the a priori bound $C_{K}$, and therefore also the time
frame $T_{K}$, do not depend on the regularisation parameter $\delta$,
nor the cut-off $R\geq0$. 
\end{rem*}
Next we show a contraction of $\mathcal{K}\left(K_{\delta},\chi_{R},\omega_{0},G\right)\diamond$
on small time frames.
\begin{thm}
\label{Theorem -  contraction estimate} Let $b,\tilde{b}$ be two
drift functions such that on the time-interval $\left[0,T_{K}\right]$,
one has $\left\Vert b\right\Vert _{\infty},\left\Vert \tilde{b}\right\Vert _{\infty}\leq C_{K}$. Then for some constant $C_{D}(T_{K},C_{K},\nu,G)$, depending on $C_{K}$,$T_{K}$,
$G$ and $\nu$, the following inequality holds 
\begin{align}
\begin{split} & \sup_{x\in\mathbb{R}^{2},r\leq t}\left|\left(\mathcal{K}\left(K_{\delta},\chi_{R},\omega_{0},G\right)\diamond b\right)(r,x)~-~\left(\mathcal{K}\left(K_{\delta},\chi_{R},\omega_{0},G\right)\diamond\tilde{b}\right)(r,x)\right|\\
 & \leq\left(C_{L}(1,T_{K},\omega_{0})+t~C_{D}(T_{K},C_{K},\nu,G)\right)\left(t+\sqrt{t}\right)\sup_{x\in\mathbb{R}^{2},r\leq t}\left|b(x,r)-\tilde{b}(x,r)\right|
\end{split}
\label{contraction estimate}
\end{align}
for any $t\leq T_{K}$. 
\end{thm}

\begin{proof}
For the homogeneous part we note that by Lemma \ref{Lemma -abstract bounds of convolution integral}
and Remark \ref{Remark - admissible f} we have 
\begin{align*}
\sup_{x\in\mathbb{R}^{2}} \left|\int_{\mathbb{R}^{2}}\mathbb{E}\left[K_{\delta}\left(x-X^{b}(\eta,t)\right)\right]\omega_{0}(\eta)\textrm{d}\eta-\int_{\mathbb{R}^{2}}\mathbb{E}\left[K_{\delta}\left(x-X^{\tilde{b}}(\eta,t)\right)\right]\omega_{0}(\eta)\textrm{d}\eta\right|&\\
 \leq C_{L}(1,T_{K},\omega_{0})\left(t+\sqrt{t}\right)\left\Vert b-\tilde{b}\right\Vert _{\infty}&,
\end{align*}
with $C_{\theta}=C_{0}$.

Next we need to investigate the error in the inhomogeneity 
\begin{align*}
 & \left|\int_{0}^{t}\int_{\mathbb{R}^{2}}\biggl(\mathbb{E}\left[K_{\delta}\left(x-X^{b}(\eta,t)\right)G\left(X^{b}(\eta,s),s\right)\right]\right.\\
 & \left.-\mathbb{E}\left[K_{\delta}\left(x-X^{\tilde{b}}(\eta,t)\right)G\left(X^{\tilde{b}}(\eta,s),s\right)\right]\biggl)\chi_{R}(\eta)\textrm{d}\eta\textrm{d}s\right|
\end{align*}

Here the key difference compared to the homogeneous term is the quadratic
dependence on the diffusion process (i.e. dependence on time correlations). 

Now we note that by the Markov property 
\begin{align*}
 & \left|\int_{\mathbb{R}^{2}}\biggl(\mathbb{E}\left[K_{\delta}\left(x-X^{b}(\eta,t)\right)G\left(X^{b}(\eta,s),s\right)\right]\right.\\
 & \qquad\left.-\mathbb{E}\left[K_{\delta}\left(x-X^{\tilde{b}}(\eta,t)\right)G\left(X^{\tilde{b}}(\eta,s),s\right)\right]\biggl)\chi_{R}(\eta)\textrm{d}\eta\right|\\
 & =\left|\int_{\mathbb{R}^{2}}\int_{\mathbb{R}^{2}}K_{\delta}\left(x-y\right)p_{b}(s,z,t,y)\mathrm{d}y~G(z,s) \int_{\mathbb{R}^{2}}p_{b}(0,\eta,s,z)\chi_{R}(\eta)\mathrm{d}\eta\mathrm{d}z\right.\\
 & \qquad-\left.\int_{\mathbb{R}^{2}}\int_{\mathbb{R}^{2}}K_{\delta}\left(x-y\right)p_{\tilde{b}}(s,z,t,y)\mathrm{d}y~G(z,s) \int_{\mathbb{R}^{2}}p_{\tilde{b}}(0,\eta,s,z)\chi_{R}(\eta)\mathrm{d}\eta\mathrm{d}z\right|
\end{align*}
which we bound above by the following two terms
\begin{equation*}
 \left|\int_{\mathbb{R}^{2}}\int_{\mathbb{R}^{2}}K_{\delta}\left(x-y\right)\left(p_{b}(s,z,t,y)-p_{\tilde{b}}(s,z,t,y)\right)\mathrm{d}y \right. \left. G(z,s)\int_{\mathbb{R}^{2}}p_{b}(0,\eta,s,z)\chi_{R}(\eta)\mathrm{d}\eta\mathrm{d}z\right|,
\end{equation*}
and
\begin{equation*}
\left|\int_{\mathbb{R}^{2}}\int_{\mathbb{R}^{2}}K_{\delta}\left(x-y\right)p_{\tilde{b}}(s,z,t,y)\mathrm{d}y~G(z,s) \int_{\mathbb{R}^{2}}\left(p_{b}(0,\eta,s,z)-p_{\tilde{b}}(0,\eta,s,z)\right)\chi_{R}(\eta)\mathrm{d}\eta\mathrm{d}z\right|,
\end{equation*}
both of which can be estimated using Lemma \ref{Lemma -abstract bounds of convolution integral}. 

To estimate the first term we define 
\begin{align*}
w_{s}(z):=G(z,s)~\int_{\mathbb{R}^{2}}p_{b}(0,\eta,s,z)\chi_{R}(\eta)\mathrm{d}\eta.
\end{align*}
The supremum norm of $w$ can be estimated just as in the proof of
Lemma \ref{Lemma - boundedness proof} using the improved Aronson
estimates \eqref{improved Aronson estimates} to derive 
\begin{align*}
\left\Vert w_{s}\right\Vert _{\infty}\leq\left\Vert G(\cdot,s)\right\Vert _{\infty}\bar{\kappa}(q,C_{K}\sqrt{s}).
\end{align*}
Next we note that by $\chi_{R}\leq1$, the improved Aronson estimate
\eqref{improved Aronson estimates}, Corollary \ref{improved Aronson - integral}
\begin{align*}
\left\Vert w_{s}\right\Vert _{1} & \leq\int_{\mathbb{R}^{2}}\left|G(z,s)\right|~\int_{\mathbb{R}^{2}}p_{b}(0,\eta,s,z)\chi_{R}(\eta)\mathrm{d}\eta\mathrm{d}z\\ 
& \leq\int_{\mathbb{R}^{2}}\int_{\mathbb{R}^{2}}\left|G(\eta-z,s)\right|~p_{b}(0,\eta,s,\eta-z)\mathrm{d}\eta\mathrm{d}z\\
 & \leq\int_{\mathbb{R}^{2}}\int_{\mathbb{R}^{2}}\left|G(\eta-z,s)\right|~\frac{e^{-\frac{|z|^{2}}{2s}}}{2\pi s}\left(1+\kappa(q)C_{K}\left(\sqrt{s}+\left|z\right|\right)e^{\frac{q-1}{2qs}|z|^{2}+\frac{C_{K}^{2}s}{2(q-1)}}\right)\mathrm{d}\eta\mathrm{d}z\\
 & =\left\Vert G(\cdot,s)\right\Vert _{1}\int_{\mathbb{R}^{2}}\frac{e^{-\frac{|z|^{2}}{2s}}}{2\pi s}\left(1+\kappa(q)\frac{C_{K}}{\sqrt{2\nu}}\left(\sqrt{s}+\left|z\right|\right)e^{\frac{q-1}{2qs}|z|^{2}+\frac{C_{K}^{2}s}{4(q-1)\nu}}\right)\mathrm{d}z,
\end{align*}
which implies that
\begin{equation*}
\left\Vert w_{s}\right\Vert _{1} \leq \left\Vert G(\cdot,s)\right\Vert_{1}\bar{\kappa}\left(q,\frac{C_{K}\sqrt{s}}{\sqrt{2\nu}}\right).
\end{equation*}

Thus, by Remark \ref{Remark - admissible f}, $f(y,z):=w_{s}(z)$
is an admissible choice and we can indeed use Lemma \ref{Lemma -abstract bounds of convolution integral}
to conclude that for some constant $C_{L1}\geq\sup_{s\leq T_{K}}C_{L}(1,T_{K},w_{s})$,
which depends on $T_{K}$ and $G$, the following inequality holds
\begin{align*}
\begin{split} \left|\int_{\mathbb{R}^{2}}~\int_{\mathbb{R}^{2}}K_{\delta}\left(x-y\right)\left(p_{b}(s,z,t,y)-p_{\tilde{b}}(s,z,t,y)\right)\mathrm{d}y~G(z,s)\int_{\mathbb{R}^{2}}p_{b}(0,\eta,s,z)\chi_{R}(\eta)\mathrm{d}\eta~\mathrm{d}z\right|&\\
 \leq C_{L1}\left(\sqrt{t-s}+(t-s)\right)\left\Vert b-\tilde{b}\right\Vert _{\infty}&.
\end{split}
\end{align*}

Next we define, for fixed $s$ and $t$, the kernel $\Theta:\mathbb{R}^{2}\times\mathbb{R}^{2}\to\mathbb{R}^{2}$ by 
\begin{align}
\Theta(x,z):=\int_{\mathbb{R}^{2}}K_{\delta}\left(x-y\right)p_{\tilde{b}}(s,z,t,y)\mathrm{d}y.\label{definition theta}
\end{align}
To use Theorem \ref{Theorem - abstract Lipschitz estimate}, we have
to show that it satisfies condition \eqref{kernel condition - in Theorem}.
To this end we note that if we set $r:=\frac{\left|z-x\right|}{2}$,
then by the generic Aronson estimate \eqref{rough Aronson estimate}
we have 
\begin{align*}
\left|\Theta(x,z)\right| & \leq\int_{\mathbb{R}^{2}}\left|K_{\delta}\left(x-y\right)\right|p_{\tilde{b}}(s,z,t,y)\mathrm{d}y\\
 & \leq C_{0}\int_{\left\{ \left|y-x\right|\leq r\right\} }\left|y-x\right|^{-1}p_{\tilde{b}}(s,z,t,y)\mathrm{d}y+C_{0}\int_{\left\{ \left|y-x\right|>r\right\} }\left|y-x\right|^{-1}p_{\tilde{b}}(s,z,t,y)\mathrm{d}y\\
 & \leq C_{0}\int_{\left\{ \left|y-x\right|\leq r\right\} }\left|y-x\right|^{-1}\frac{M(C_{K},T,\nu)}{2\pi(t-s)}e^{-\frac{|y-z|^{2}}{M(C_{K},T,\nu)(t-s)}}\mathrm{d}y+C_{0}r^{-1}.
\end{align*}

Next we note that by multiplying with $1/M(C_{K},T,\nu)^{2}$, switching
to polar coordinates and $|z-x|=2r$ 
\begin{align*}
 & \int_{\left\{ \left|y-x\right|\leq r\right\} }\left|y-x\right|^{-1}\frac{e^{-\frac{|y-z|^{2}}{M(C_{K},T,\nu)(t-s)}}}{2\pi(t-s)M(C_{K},T,\nu)}\mathrm{d}y \\
 &\leq\int_{0}^{r}\frac{e^{-\frac{(2r-\lambda)^{2}}{M(C_{K},T,\nu)(t-s)}}}{(t-s)M(C_{K},T,\nu)}\mathrm{d}\lambda\leq r\frac{e^{-\frac{r^{2}}{M(C_{K},T,\nu)(t-s)}}}{(t-s)M(C_{K},T,\nu)}.
\end{align*}

Finally we set $C_{G}:=\sup_{x\in(0,\infty)}\frac{x}{e^{x}}$.
Then the previous calculations let us conclude 
\begin{align*}
\left|\Theta(x,z)\right|\leq2^{-1}\left(M(C_{K},T,\nu)^{2}C_{G}+1\right)|x-z|^{-1},
\end{align*}
which implies that we can indeed use Theorem \ref{Theorem - abstract Lipschitz estimate}.
Therefore we know that for some constant $C_{L2}\geq\sup_{s\leq T_{K}}C_{L}(1,T,G(\cdot,s))$,
depending on $C_{K}$,$T_{K}$, $G$ and $\nu$, we have 
\begin{align*}
 \left|\int_{\mathbb{R}^{2}}~\int_{\mathbb{R}^{2}}K_{\delta}\left(x-y\right)p_{\tilde{b}}(s,z,t,y)\mathrm{d}y~G(z,s)~\int_{\mathbb{R}^{2}}\left(p_{b}(0,\eta,s,z)-p_{\tilde{b}}(0,\eta,s,z)\right)\chi_{R}(\eta)\mathrm{d}\eta~\mathrm{d}z\right| &\\
 \leq C_{L2}\left(\sqrt{t-s}+(t-s)\right)\left\Vert b-\tilde{b}\right\Vert _{\infty}&. 
\end{align*}

Therefore there exists some constant $C_{D}(T_{K},C_{K},\nu,G)$ such
that \eqref{contraction estimate} holds. 
\end{proof}
The theorem we proved shows the contractivity for small timeframes, which in turn gives us the short-time well-posedness of the mean-field problem
\eqref{regularised mf equation}, \eqref{mean field form of drift - regularised}.
\begin{cor}
\label{Corollary - small time well-posedness} Let $T\leq T_{K}$
be such that 
\begin{align*}
\left(C_{L}(1,T_{K},\omega_{0})+TC_{D}(T_{K},C_{K},\nu,G)\right)\left(T+\sqrt{T}\right)<1,
\end{align*}
then there exists a unique bounded vector field $u_{\delta,R}\in L^{\infty}\left(\mathbb{R}^{2}\times[0,T];\mathbb{R}^{2}\right)$
such that $\mathcal{K}\left(K_{\delta},\chi_{R},\omega_{0},G\right)\diamond u_{\delta,R}=u_{\delta,R}$.
Furthermore on $[0,T]$, there exists a unique strong solution $X_{\delta,R}$
to the McKean--Vlasov equation \eqref{regularised mf equation}. 
\end{cor}
\begin{proof}
By Veretennikov \citep{Veretennikov}, for any $b\in L^{\infty}\left(\mathbb{R}^{2}\times[0,T];\mathbb{R}^{2}\right)$,
there is a unique strong solution of the SDE 
\begin{align*}
\mathrm{d}X=b(X(t),t)\mathrm{d}t+\sqrt{2\nu}\mathrm{d}B_{t},
\end{align*}
thus $\mathcal{K}\left(K_{\delta},\chi_{R},\omega_{0},G\right)$ is
well defined on $L^{\infty}\left(\mathbb{R}^{2}\times[0,T];\mathbb{R}^{2}\right)$.
Combining Lemma \ref{Lemma - boundedness proof} and Theorem \ref{Theorem -  contraction estimate},
shows that $\mathcal{K}\left(K_{\delta},\chi_{R},\omega_{0},G\right)$
is even a contraction and thus by the Banach fixed point theorem,
there is a unique fixed point $u_{\delta,R}$. 
\end{proof}
\begin{rem*}
\label{Remark - a priori bounds velocity} Note that the bound 
\begin{align*}
\sup_{x\in\mathbb{R}^{2},t\leq T_{K}}\left|u_{\delta}(x,t)\right|\leq C_{K}
\end{align*}
holds for all $\delta\geq0$ and all $R\geq0$. This is in particular
true for the fluid velocity $u$! 
\end{rem*}
Now that we have proven the well-posedness of \eqref{regularised mf equation}, it remains to show that the true velocity $u$ can be approximated by the the regularised velocity $u_{\delta,R}$, which we do in the following lemma.
\begin{lem}
\label{Lemma - robustness in delta and R} There exists some sufficiently
small timeframe $T>0$ and a constant $\tilde{C}(T)>0$, only depending
on $\omega_{0},G$ and $T$, such that 
\begin{align}
\sup_{t\leq T,x\in\mathbb{R}^{2}}\left|u_{\delta,R}(x,t)-u(x,t)\right|\leq2\tilde{C}(T)\left(\delta+\frac{1}{R}\right).\label{robustness in delta and R - estimate}
\end{align}
\end{lem}
\begin{proof}
As noted in Remark \ref{a priori bound and timeframe independent of delta and chi},
both $T_{K}$ and $C_{K}$ do not depend on $\delta$ or $R$. Let
the regularisation parameter $\delta\geq0$ and the cutoff threshold
$R\in[0,+\infty]$ be arbitrary, but fixed. Then we note that 
\begin{align*}
 \left|u_{\delta,R}(x,t)-u(x,t)\right|&=\left|\mathcal{K}\left(K_{\delta},\chi_{R},\omega_{0},G\right)\diamond u_{\delta,\chi}(x,t)-\mathcal{K}\left(K_{0},\chi_{\infty},\omega_{0},G\right)\diamond u(x,t)\right|\\
 & \leq\left|\mathcal{K}\left(K_{\delta},\chi_{R},\omega_{0},G\right)\diamond u_{\delta,R}(x,t)-\mathcal{K}\left(K_{\delta},\chi_{R},\omega_{0},G\right)\diamond u(x,t)\right|\\
 & \;+\left|\mathcal{K}\left(K_{\delta},\chi_{R},\omega_{0},G\right)\diamond u(x,t)-\mathcal{K}\left(K_{0},\chi_{\infty},\omega_{0},G\right)\diamond u(x,t)\right|.
\end{align*}
and, by Theorem \ref{Theorem -  contraction estimate}, the first term is bounded by
\begin{equation*}
\left(C_{L}(1,T_{K},\omega_{0})+tC_{D}(T_{K},C_{K},\nu,G)\right)\left(t+\sqrt{t}\right)\sup_{y\in\mathbb{R}^{2},r\leq t}\left|u_{\delta,R}(y,r)-u(y,r)\right|.
\end{equation*}
For estimating the last term, we proceed as follows
\begin{align*}
 & \left|\int_{\mathbb{R}^{2}}\mathbb{E}\left[K_{\delta}\left(x-X^{u}(\eta,t)\right)\left(\omega_{0}(\eta)+\int_{0}^{t}G\left(X^{u}(\eta,s),s\right)\textrm{d}s~\chi_{R}(\eta)\right)\right]\textrm{d}\eta\right.\\
 & \qquad-\left.\int_{\mathbb{R}^{2}}\mathbb{E}\left[K\left(x-X^{u}(\eta,t)\right)\left(\omega_{0}(\eta)+\int_{0}^{t}G\left(X^{u}(\eta,s),s\right)\textrm{d}s\right)\right]\textrm{d}\eta\right|\\
 & \leq\int_{\mathbb{R}^{2}}\mathbb{E}\biggr[\left|K_{\delta}\left(x-X^{u}(\eta,t)\right)-K\left(x-X^{u}(\eta,t)\right)\right|\\
 &\qquad\times\left(\left|\omega_{0}(\eta)\right|+\int_{0}^{t}\left|G\left(X^{u}(\eta,s),s\right)\right|\textrm{d}s~\chi_{R}(\eta)\right)\biggr]\textrm{d}\eta\\
 & \qquad+\int_{0}^{t}\int_{\mathbb{R}^{2}}\mathbb{E}\left[\left|K\left(x-X^{u}(\eta,t)\right)\right|\left|G\left(X^{u}(\eta,s),s\right)\right|\right]\left(1-\chi_{R}(\eta)\right)\textrm{d}\eta\textrm{d}s.
\end{align*}

First we note that by Assumption \ref{assumption Kdelta - epsilon existence}, as well as Lemma \ref{Lemma -abstract bounds of convolution integral}, we have
\begin{align}
\begin{split} & \int_{\mathbb{R}^{2}}\mathbb{E}\left[\left|K_{\delta}\left(x-X^{u}(\eta,t)\right)-K\left(x-X^{u}(\eta,t)\right)\right|\right]\left|\omega_{0}(\eta)\right|\textrm{d}\eta\\
 & \leq\int_{\mathbb{R}^{2}}\int_{\mathbb{R}^{2}}\frac{\epsilon_{\delta}(y-x)}{|y-x|}p_{u}(0,\eta,t,y)\left|\omega_{0}(\eta)\right|\textrm{d}y\textrm{d}\eta\\
 & \leq\bar{\kappa}\left(q,\frac{C_{K}\sqrt{t}}{\sqrt{2\nu}}\right)\left(\int_{|z|\leq1}\frac{\epsilon_{\delta}(z)}{|z|^{\gamma}}\mathrm{d}z~\left\Vert \omega_{0}\right\Vert _{\infty}+\left(\sup_{|z|\geq1}\frac{\epsilon_{\delta}(z)}{|z|^{\gamma}}\right)\left\Vert \omega_{0}\right\Vert _{\mathrm{diag}_{1}}\right)\\
 & \leq\bar{\kappa}\left(q,\frac{C_{K}\sqrt{t}}{\sqrt{2\nu}}\right)\left(\left\Vert \omega_{0}\right\Vert _{\infty}+\left\Vert \omega_{0}\right\Vert _{1}\right)~C_{\mathrm{reg}}~\delta.
\end{split}
\label{regularisation error - homogeneous part}
\end{align}

Next we note that by the Markov property 
\begin{align*}
 & \int_{\mathbb{R}^{2}}\mathbb{E}\left[\left|K_{\delta}\left(x-X^{u}(\eta,t)\right)-K\left(x-X^{u}(\eta,t)\right)\right|\right]\int_{0}^{t}\left|G\left(X^{u}(\eta,s),s\right)\right|\textrm{d}s~\chi_{R}(\eta)\textrm{d}\eta\\
 & =\int_{0}^{t}\int_{\mathbb{R}^{2}}\int_{\mathbb{R}^{2}}\left|K_{\delta}\left(x-y\right)-K\left(x-y\right)\right|p_{u}(s,z,t,y)\\
 &\qquad\times\left|G\left(z,s\right)\right|\int_{\mathbb{R}^{2}}p_{u}(0,\eta,s,z)\chi_{R}(\eta)\textrm{d}\eta\mathrm{d}z\mathrm{d}y\textrm{d}s.
\end{align*}
Setting $w_{s}(z):=\left|G\left(z,s\right)\right|\int_{\mathbb{R}^{2}}p_{u}(0,\eta,s,z)\chi_{R}(\eta)\textrm{d}\eta$,
then, just as in the proof of Theorem \ref{Theorem -  contraction estimate},
one derives the estimate 
\begin{align*}
\left\Vert w_{s}\right\Vert _{\infty}+\left\Vert w_{s}\right\Vert _{1}\leq\left(\left\Vert G(\cdot,s)\right\Vert _{\infty}+\left\Vert G(\cdot,s)\right\Vert _{1}\right)\bar{\kappa}\left(q,\frac{C_{K}\sqrt{s}}{\sqrt{2\nu}}\right).
\end{align*}

Thus, just as we did for the homogeneous part before, we obtain by
Assumption \ref{assumption Kdelta - epsilon existence} and Lemma \ref{Lemma -abstract bounds of convolution integral}
the following estimate 
\begin{align}
\begin{split} & \int_{\mathbb{R}^{2}}\mathbb{E}\left[\left|K_{\delta}\left(x-X^{u}(\eta,t)\right)-K\left(x-X^{u}(\eta,t)\right)\right|\right]\int_{0}^{t}\left|G\left(X^{u}(\eta,s),s\right)\right|\textrm{d}s~\chi_{R}(\eta)\textrm{d}\eta\\
 & \leq C_{\mathrm{reg}}~\delta\int_{0}^{t}~\bar{\kappa}\left(q,C_{K}\sqrt{s}\right)~C_{G}(s)~\bar{\kappa}(q,C_{K}\sqrt{t-s})~\mathrm{d}s.
\end{split}
\label{regularisation error - inhomogeneous part}
\end{align}

Finally we estimate the localisation error. Using the Markov property,
we note that 
\begin{align*}
 & \int_{0}^{t}\int_{\mathbb{R}^{2}}\mathbb{E}\left[\left|K\left(x-X^{u}(\eta,t)\right)\right|\left|G\left(X^{u}(\eta,s),s\right)\right|\right]\left(1-\chi_{R}(\eta)\right)\textrm{d}\eta\textrm{d}s\\
 & =\int_{0}^{t}\int_{\mathbb{R}^{2}}\int_{\mathbb{R}^{2}}\left|K\left(x-y\right)\right|p_{u}(s,z,t,y)\left|G\left(z,s\right)\right|\int_{\mathbb{R}^{2}}p_{u}(0,\eta,s,z)\left(1-\chi_{R}(\eta)\right)\textrm{d}\eta\textrm{d}y\textrm{d}z\textrm{d}s.
\end{align*}

We define $v_{s,R}(z):=\left|G\left(z,s\right)\right|~\int_{\mathbb{R}^{2}}p_{u}(0,\eta,s,z)\left(1-\chi_{R}(\eta)\right)\textrm{d}\eta$
and note that by the basic Aronson estimate \eqref{rough Aronson estimate}
we have 
\begin{align*}
v_{s,R}(z)\leq M(C_{K},s,\nu)^{2}\left|G\left(z,s\right)\right|\int_{\left\{ |\eta|>R\right\} }\frac{e^{-\frac{|\eta-z|^{2}}{2M(C_{K},s,\nu)s}}}{2\pi s~M(C_{K},s)}\textrm{d}\eta.
\end{align*}
Let $\mathrm{H}\sim\mathcal{N}\left(0,M(C_{K},s,\nu)s\mathrm{I}\right)$ be a normally distributed random variable, then we use the Markov
inequality to imply from the previous inequality  
\begin{align*}
v_{s,R}(z) & \leq M(C_{K},s,\nu)^{2}\left|G\left(z,s\right)\right|\mathbb{P}\left(\left|\mathrm{H}-z\right|>R\right)\\
 & \leq M(C_{K},s,\nu)^{2}\left|G\left(z,s\right)\right|\frac{\sqrt{\frac{\pi M(C_{K},s,\nu)s}{2}}+|z|}{R},
\end{align*}
and therefore if we set 
\begin{align}
C_{G}(t):=\sup_{s\leq t}\left(\left\Vert G(\cdot,s)(1+|\cdot|)\right\Vert _{\infty}+\left\Vert G(\cdot,s)(1+|\cdot|)\right\Vert _{1}\right),\label{Assumption first moment of G}
\end{align}
then Lemma \ref{Lemma -abstract bounds of convolution integral} gives
us 
\begin{align}
\begin{split} & \int_{0}^{t}\int_{\mathbb{R}^{2}}\mathbb{E}\left[\left|K\left(x-X^{u}(\eta,t)\right)\right|\left|G\left(X^{u}(\eta,s),s\right)\right|\right]\left(1-\chi_{R}(\eta)\right)\textrm{d}\eta\textrm{d}s\\
 & \leq\int_{0}^{t}\bar{\kappa}\left(q,\frac{C_{K}\sqrt{s}}{\sqrt{2\nu}}\right)\left(2\pi\left\Vert v_{s,R}(z)\right\Vert _{\infty}+\left\Vert v_{s,R}(z)\right\Vert _{1}\right)\mathrm{d}s\\
 & \leq\int_{0}^{t}2\pi\bar{\kappa}\left(q,\frac{C_{K}\sqrt{s}}{\sqrt{2\nu}}\right)M(C_{K},s,\nu)^{2}\left(1+\sqrt{\frac{\pi M(C_{K},s,\nu)s}{2}}+1\right)C_{G}(s)\mathrm{d}s~\frac{1}{R}.
\end{split}
\label{localzation error}
\end{align}
Note that by Assumption \ref{Assumption - initial vorticity and force }
we can guarantee the boundedness of $C_{G}$ on finite time intervals,
and thus the integral above can be bounded.

Combining \eqref{regularisation error - homogeneous part}, \eqref{regularisation error - inhomogeneous part} and \eqref{localzation error}, we therefore derive that for any $T>0$
there exists a constant $\tilde{C}(T)$, such that 
\begin{align*}
\sup_{t\leq T,x\in\mathbb{R}^{2}}\left|\mathcal{K}\left(K_{\delta},\chi_{R},\omega_{0},G\right)\diamond u(x,t)-\mathcal{K}\left(K_{0},\chi_{\infty},\omega_{0},G\right)\diamond u(x,x)\right|\leq\tilde{C}(T)\left(\delta+\frac{1}{R}\right).
\end{align*}
Thus for $T\leq T_{K}$, such that 
\begin{align*}
\left(C_{L}(1,T_{K},\omega_{0})+TC_{D}(T_{K},C_{K},\nu,G)\right)\left(T+\sqrt{T}\right)\leq1/2,
\end{align*}
it follows that 
\begin{align*}
\left|u_{\delta,R}(x,t)-u(x,t)\right|\leq2\tilde{C}(T)\left(\delta+\frac{1}{R}\right),
\end{align*}
which concludes our proof. 
\end{proof}
In the same manner one can derive the robustness of $u_{\delta,R}$
with respect to the initial vorticity $\omega_{0}$ and the force
$G$, as the following lemma shows.
\begin{lem}
\label{Lemma - robustness with respect to parameters} Let two initial
vorticities $\omega_{0},\tilde{\omega}_{0}$ and two external forces
$G,\tilde{G}$ be given. Then for any drift function $b$ with $\sup_{t\leq T_{K},x\in\mathbb{R}^{2}}\|b(x,t)\|\leq C_{K}$,
we have the following property 
\begin{align*}
\begin{split} \sup_{t\leq T_{K},x\in\mathbb{R}^{2}}\left|\left(\mathcal{K}\left(K_{\delta},\chi_{R},\omega_{0},G\right)\diamond b\right)(x,t)-\left(\mathcal{K}\left(K_{\delta},\chi_{R},\tilde{\omega}_{0},\tilde{G}\right)\diamond b\right)(x,t)\right| \leq C(T_{K}) &\\
 \times\sup_{s\leq T_{K}}\left(\left\Vert \omega_{0}-\tilde{\omega}_{0}\right\Vert _{\infty}+\left\Vert \omega_{0}-\tilde{\omega}_{0}\right\Vert _{1}+\left\Vert G(\cdot,s)-\tilde{G}(\cdot,s)\right\Vert _{\infty}+\left\Vert G(\cdot,s)-\tilde{G}(\cdot,s)\right\Vert _{1}\right)&.
\end{split}
\end{align*}
To emphasize the dependence on the initial vorticity and forcing,
let us denote by $u_{\delta,R}\left[\omega_{0},G\right]$ the mean
field velocity $u_{\delta,R}\left[\omega_{0},G\right]=\mathcal{K}\left(K_{\delta},\chi_{R},\omega_{0},G\right)\diamond u_{\delta,R}\left[\omega_{0},G\right]$
as found in Corollary \ref{Corollary - small time well-posedness}.
Then for some sufficiently small timeframe $0<T\leq T_{K}$, there
exists a constant $C$ such that 
\begin{align*}
\begin{split} & \sup_{t\leq T,x\in\mathbb{R}^{2}}\left|u_{\delta,R}\left[\omega_{0},G\right](x,t)-u_{\delta,R}\left[\tilde{\omega}_{0},\tilde{G}\right](x,t)\right|\\
 & \leq C\sup_{s\leq T}\left(\left\Vert \omega_{0}-\tilde{\omega}_{0}\right\Vert _{\infty}+\left\Vert \omega_{0}-\tilde{\omega}_{0}\right\Vert _{1}+\left\Vert G(\cdot,s)-\tilde{G}(\cdot,s)\right\Vert _{\infty}+\left\Vert G(\cdot,s)-\tilde{G}(\cdot,s)\right\Vert _{1}\right).
\end{split}
\end{align*}
\end{lem}

\begin{proof}
First we decompose the error into two parts 
\begin{align*}
 & \left|\mathcal{K}\left(K_{\delta},\chi_{R},\omega_{0},G\right)\diamond b(x,t)-\mathcal{K}\left(K_{\delta},\chi_{R},\tilde{\omega}_{0},\tilde{G}\right)\diamond b(x,t)\right|\\
 & =\left|\int_{\mathbb{R}^{2}}\mathbb{E}\left[K_{\delta}\left(x-X^{b}(\eta,t)\right)\left(\omega_{0}(\eta)+\int_{0}^{t}G\left(X^{b}(\eta,s),s\right)\textrm{d}s~\chi_{R}(\eta)\right)\right]\textrm{d}\eta\right.\\
 & \qquad-\left.\int_{\mathbb{R}^{2}}\mathbb{E}\left[K_{\delta}\left(x-X^{b}(\eta,t)\right)\left(\tilde{\omega}_{0}(\eta)+\int_{0}^{t}\tilde{G}\left(X^{b}(\eta,s),s\right)\textrm{d}s~\chi_{R}(\eta)\right)\right]\textrm{d}\eta\right|\\
 & \leq\left|\int_{\mathbb{R}^{2}}\mathbb{E}\left[K_{\delta}\left(x-X^{b}(\eta,t)\right)\left(\omega_{0}(\eta)-\tilde{\omega}_{0}(\eta)\right)\right]\textrm{d}\eta\right|\\
 & \qquad+\int_{0}^{t}\left|\int_{\mathbb{R}^{2}}\mathbb{E}\left[K_{\delta}\left(x-X^{b}(\eta,t)\right)\left(G\left(X^{b}(\eta,s),s\right)-\tilde{G}\left(X^{b}(\eta,s),s\right)\right)\right]~\chi_{R}(\eta)\textrm{d}\eta\right|\mathrm{d}s.
\end{align*}

The first term can be estimated by Lemma \ref{Lemma -abstract bounds of convolution integral},
as 
\begin{align*}
&\left|\int_{\mathbb{R}^{2}}\mathbb{E}\left[K_{\delta}\left(x-X^{b}(\eta,t)\right)\left(\omega_{0}(\eta)-\tilde{\omega}_{0}(\eta)\right)\right]\textrm{d}\eta\right|\\
 & \leq\int_{\mathbb{R}^{2}}\int_{\mathbb{R}^{2}}\left|K_{\delta}\left(x-y\right)\right|\left|\omega_{0}(\eta)-\tilde{\omega}_{0}(\eta)\right|\textrm{d}y\textrm{d}\eta\\
 & \leq C_{0}\bar{\kappa}\left(q,\frac{C_{K}\sqrt{t}}{\sqrt{2\nu}}\right)\left(2\pi\left\Vert \omega_{0}-\tilde{\omega}_{0}\right\Vert _{\infty}+\left\Vert \omega_{0}-\tilde{\omega}_{0}\right\Vert _{1}\right).
\end{align*}

For the second term we note that for $\Theta(x,z):=\int_{\mathbb{R}^{2}}K_{\delta}(y-x)p_{b}(s,z,t,y)\mathrm{d}y$, defined as in \eqref{definition theta} in Theorem \ref{Theorem -  contraction estimate}, we have 
\begin{align*}
\left|\Theta(x,z)\right|\leq\frac{M(C_{K},T,\nu)^{2}C_{G}+1}{2|x-z|},
\end{align*}
as proved in Theorem \ref{Theorem -  contraction estimate}. Therefore
we get 
\begin{align*}
  \int_{0}^{t}&\left|\int_{\mathbb{R}^{2}}\mathbb{E}\left[K_{\delta}\left(x-X^{b}(\eta,t)\right)\left(G\left(X^{b}(\eta,s),s\right)-\tilde{G}\left(X^{b}(\eta,s),s\right)\right)\right]~\chi_{R}(\eta)\textrm{d}\eta\right|\mathrm{d}s\\
 & =\int_{0}^{t}\int_{\mathbb{R}^{2}}\int_{\mathbb{R}^{2}}\int_{\mathbb{R}^{2}}K_{\delta}(x-y)p_{b}(s,z,t,y)\mathrm{d}y~p_{b}(0,\eta,s,z)G(z,s)\mathrm{d}z~\chi_{R}(\eta)\mathrm{d}\eta\mathrm{d}s\\
 & \leq\frac{M(C_{K},T,\nu)^{2}C_{G}+1}{2}\int_{0}^{t}\bar{\kappa}\left(q,\frac{C_{K}\sqrt{s}}{\sqrt{2\nu}}\right)\\ & \qquad \times \left(2\pi\left\Vert G(\cdot,s)-\tilde{G}(\cdot,s)\right\Vert _{\infty}+\left\Vert G(\cdot,s)-\tilde{G}(\cdot,s)\right\Vert _{1}\right)\mathrm{d}s.
\end{align*}
Thus our first claim follows immediately. The second claim follows as in Lemma \ref{Lemma - robustness in delta and R}. 
\end{proof}
The lemma we just proved does not just give us that $u_{\delta,R}$ depends continuously on the initial vorticity and the external force. As the following result shows, the robustness estimate from Lemma \ref{Lemma - robustness with respect to parameters} allows us to show the Lipschitz continuity of $u_{\delta,R}$ for all $\delta\geq0$ and $R\in[0,+\infty]$.
\begin{lem}
\label{Lemma - Lipschitz continuity} There exists a time frame $0<T\leq T_{K}$
and a constant $C$ such that 
\begin{align*}
\frac{\left|u_{\delta,R}(x+\xi,t)-u_{\delta,R}(x,t)\right|}{\left|\xi\right|}\leq C\sup_{s\leq T}\left(\left\Vert \nabla\omega_{0}\right\Vert _{\infty}+\left\Vert \nabla\omega_{0}\right\Vert _{1}+\left\Vert \nabla G(\cdot,s)\right\Vert _{\infty}+\left\Vert \nabla G(\cdot,s)\right\Vert _{1}\right),
\end{align*}
for all $\xi\in\mathbb{R}^{2}, \xi \neq 0$, $\delta\geq0$ and $R\in[0,\infty]$.
Thus, for small time frames, $u_{\delta,R}$ is Lipschitz continuous
with a Lipschitz constant independent of $\delta$ and $R$. 
\end{lem}

\begin{proof}
First we note that by a simple change of variables we have for any
$\xi\in\mathbb{R}^{2}$ and any drift function $b$
\begin{align*}
\left(\mathcal{K}\left(K_{\delta},\chi_{R},\omega_{0},G\right)\diamond b\right)(x+\xi,t)=\left(\mathcal{K}\left(K_{\delta},\chi_{R},\left(\omega_{0}\right)_{\xi-},G_{\xi-}\right)\diamond b_{\xi-}\right)(x,t),
\end{align*}
due to 
\begin{align*}
 &\int_{\mathbb{R}^{2}}\int_{\mathbb{R}^{2}}K_{\delta}(x+\xi-y)p_{b}(0,\eta,t,y)\mathrm{d}y~\omega_{0}(\eta)\mathrm{d}\eta \\&=\int_{\mathbb{R}^{2}}\int_{\mathbb{R}^{2}}K_{\delta}(x-y)p_{b}(0,\xi-\eta,t,\xi-y)\mathrm{d}y~\omega_{0}(\xi-\eta)\mathrm{d}\eta,
\end{align*}
and a similar equality for the inhomogeneous parts which follows from the identity $p_{b}(\tau,\xi-z,t,\xi-y)=p_{b_{\xi-}}(\tau,z,t,y)$ mentioned earlier in \eqref{reflection formula for transition densities}.

Since $\mathcal{K}\left(K_{\delta},\chi_{R},\omega_{0},G\right)\diamond u_{\delta,R} = u_{\delta, R}$, the difference
\begin{align*}
\left|u_{\delta,R}(x+\xi,t)-u_{\delta,R}(x,t)\right|
\end{align*}
is equal to
\begin{align*}
 & \left|\left(\mathcal{K}\left(K_{\delta},\chi_{R},\omega_{0},G\right)\diamond u_{\delta,R}\right)(x+\xi,t)-\left(\mathcal{K}\left(K_{\delta},\chi_{R},\omega_{0},G\right)\diamond u_{\delta,R}\right)(x,t)\right|\\
 & \leq\left|\left(\mathcal{K}\left(K_{\delta},\chi_{R},\left(\omega_{0}\right)_{\xi-},G_{\xi-}\right)\diamond\left(u_{\delta,R}\right)_{\xi-}\right)(x,t)-\left(\mathcal{K}\left(K_{\delta},\chi_{R},\omega_{0},G\right)\diamond\left(u_{\delta,R}\right)_{\xi-}\right)(x,t)\right|\\
 & \qquad+\left|\left(\mathcal{K}\left(K_{\delta},\chi_{R},\omega_{0},G\right)\diamond\left(u_{\delta,R}\right)_{\xi-}\right)(x,t)-\left(\mathcal{K}\left(K_{\delta},\chi_{R},\omega_{0},G\right)\diamond u_{\delta,R}\right)(x,t)\right|.
\end{align*}
Using Theorem \ref{Theorem -  contraction estimate} we conclude that the second term is bounded by 
\begin{align*}
\left(C_{L}(\omega_{0})+tC_{L}(G)\right)\left(t+\sqrt{t}\right)\sup_{y\in\mathbb{R}^{2},r\leq t}\left|u_{\delta,R}(y+\xi,r)-u_{\delta,R}(y,r)\right|.
\end{align*}
Lemma \ref{Lemma - robustness with respect to parameters} implies that for some constant $C(T_{K})>0$ the first term is bounded by
\begin{align*}
&C(T_{K})\sup_{s\leq T_{K}}\Big(\left\Vert \omega_{0}(\xi-\cdot)-\omega_{0}(\cdot)\right\Vert _{\infty}+\left\Vert \omega_{0}(\xi-\cdot)-\omega_{0}(\cdot)\right\Vert _{1}\\
 & \qquad+\left\Vert G(\xi-\cdot,s)-G(\cdot,s)\right\Vert _{\infty}+\left\Vert G(\xi-\cdot,s)-G(\cdot,s)\right\Vert _{1}\Big)\\
 & \leq C(T_{K})\sup_{s\leq T_{K}}\left(\left\Vert \nabla\omega_{0}\right\Vert _{\infty}+\left\Vert \nabla\omega_{0}\right\Vert _{1}+\left\Vert \nabla G(\cdot,s)\right\Vert_{\infty}+\left\Vert \nabla G(\cdot,s)\right\Vert_{1}\right)\left|\xi\right|.
\end{align*}

Thus for some small $T>0$ such that $\left(C_{L}(\omega_{0})+TC_{L}(G)\right)\left(T+\sqrt{T}\right)\leq \frac{1}{2}$, our claim immediately follows. 
\end{proof}
Lemma \ref{Lemma - Lipschitz continuity} in particular also holds
for $\delta=0$ and $R=+\infty$ and thus it proves the Lipschitz
continuity for the true fluid velocity $u$ as well. As Lipschitz
continuity also implies differentiability (almost everywhere) with
bounded derivatives, Lemma \ref{Lemma - Lipschitz continuity} proves
a certain regularity of $u$. 

\section{Convergence of Random Vortex Schemes}
\label{section convergence}

In this section we prove the convergence of the FMCRV scheme, an analogous proof shows the convergence of the other schemes (PMCRV, PiCRV) described in Section \ref{Random vortex method -- flows without constraint}.

To state our main result in this section, we shall repeat some definitions to establish the notation which will be used in this section.

Let $\delta>0$ and $R>0$ be two parameters. Then the (field-valued)
McKean--Vlasov stochastic differential equations 
\[
\mathrm{d}X_{\delta,R}(x,t)=u_{\delta,R}(X_{\delta,R}(x,t),t)\mathrm{d}t+\sqrt{2\nu}\mathrm{d}B_{t},~X_{\delta,R}(x,0)=x
\]
and 
\begin{align*}
u_{\delta,R}(x,t) & =\int_{\mathbb{R}^{2}}\mathbb{E}\left[K_{\delta}(X_{\delta,R}(\eta,t)-x)\right]\omega_{0}(\eta)\textrm{d}\eta\\
 & \;+\int_{0}^{t}\int_{\mathbb{R}^{2}}\mathbb{E}\left[K_{\delta}(X_{\delta,R}(\eta,t)-x)G(X_{\delta,R}(\eta,s),s)\right]\chi_{R}(\eta)\textrm{d}\eta\textrm{d}s
\end{align*}
for $x\in\mathbb{R}^{2}$ and $t\geq0$, where $B$ is a two-dimensional
Brownian motion on some probability space.

Let $h>0$ and $N\in\mathbb{N}$ are given. The Monte Carlo approximation
$\hat{u}_{N,h,\delta,R}$ is defined via the following system 
\begin{align*}
\hat{u}_{N,h,\delta,R}(x,t)=\frac{1}{N}\sum_{k=1}^{N}\sum_{j\in\mathbb{Z}^{2}}h^{2}\chi_{R}(jh)K_{\delta}(\hat{X}_{N,h,\delta,R}^{k}(j,t)-x)& \\ \times\left(\omega_{0}(jh)+\int_{0}^{t}G(\hat{X}_{N,h,\delta,R}^{k}(j,s),s)\textrm{d}s\right)&
\end{align*}
and 
\[
\mathrm{d}\hat{X}_{N,h,\delta,R}^{k}(j,t)=\hat{u}_{N,h,\delta,R}(\hat{X}_{N,h,\delta,R}^{k}(j,t),t)\mathrm{d}t+\sqrt{2\nu}\textrm{d}B_{t}^{k},\quad\hat{X}_{N,h,\delta,R}^{k}(j,0)=jh
\]
where $k=1,\ldots,N$, $j\in\mathbb{Z}^{2}$ and $x\in\mathbb{R}^{2}$
and $t\geq0$. Here $B^{k}$ ($k=1,2,\ldots$) are independent copies
of two-dimensional Brownian motion. 
\begin{thm}
Let $\delta>0$ and $R>0$ be fixed. Then 
\begin{align}
\hat{u}_{N,h,\delta,R}\rightarrow u_{\delta,R}~\text{and}~\hat{X}_{N,h,\delta,R}\rightarrow X_{\delta,R}\label{abstract Monte Carlo convergence}
\end{align}
as $N\rightarrow\infty$ and $h\rightarrow0$. 
\end{thm}

We divide the proof into several steps. The proof uses a standard
synchronous coupling approach. First we may run independent copies
$X_{\delta,R}^{k}$ (for $k=1,2,\ldots)$ of $X_{\delta,R}$ by solving the SDE: 
\begin{align}
\mathrm{d}X_{\delta,R}^{k}(x,t)=u_{\delta,R}(X_{\delta,R}^{k},t)\mathrm{d}t+\sqrt{2\nu}\mathrm{d}B_{t}^{k},~~X_{\delta,R}^{k}(x,0)=x\label{synchronous coupling}
\end{align}
for $x\in\mathbb{R}^{2}$ and $t\geq0$, where $u_{\delta,R}$ is
the unique fixed point to $\mathcal{K}\left(K_{\delta},\chi_{R},\omega_{0},G\right)\diamond u_{\delta,R}=u_{\delta,R}$
found in Corollary \ref{Corollary - small time well-posedness} in
Section \ref{section mean field analysis}. Indeed $u_{\delta,R}$
is the unique solution to \eqref{synchronous coupling} together with
\begin{align*}
u_{\delta,R}(x,t) & =\int_{\mathbb{R}^{2}}\mathbb{E}\left[K_{\delta}(X_{\delta,R}^{k}(\eta,t)-x)\right]\omega_{0}(\eta)\textrm{d}\eta\\
 & \;+\int_{0}^{t}\int_{\mathbb{R}^{2}}\mathbb{E}\left[K_{\delta}(X_{\delta,R}^{k}(\eta,t)-x)G(X_{\delta,R}^{k}(\eta,s),s)\right]\chi_{R}(\eta)\textrm{d}\eta\textrm{d}s
\end{align*}
for every $k$. The solution $u_{\delta,R}$ is independent of $k$.

\begin{rem*}
Note that since we are only dealing with additive noise, we can transform
the SDEs that we consider into random ODEs, where the randomness comes
from the Brownian motions $\left(B^{k}\right)_{k\in\mathbb{N}}$.
Thus the random field $\hat{X}_{N,h,\delta,R}$, $\left(X_{\delta,R}^{k}\right)_{k\in\mathbb{N}}$
and $\hat{u}_{N,h,\delta,R}$ are obtained by deterministic transformations
of $\left(B^{k}\right)_{k\in\mathbb{N}}$. 
\end{rem*}
Next we define the Monte-Carlo velocity computed from \eqref{synchronous coupling} by 
\begin{equation}
\begin{split}u_{N,h,\delta,R}(x,t) & :=\sum_{j\in\mathbb{Z}^{2}}\frac{1}{N}\sum_{k=1}^{N}K_{\delta}(X_{\delta,R}^{k}(jh,t)-x)\omega_{0}(jh)h^{2}\\
 & \;+\int_{0}^{t}\sum_{j\in\mathbb{Z}^{2}}\frac{1}{N}\sum_{k=1}^{N}K_{\delta}(X_{\delta,R}^{k}(jh,t)-x)G(X_{\delta,R}^{k}(jh,s),s)\chi_{R}(jh)h^{2}\textrm{d}s.
\end{split}
\label{Monte Carlo velocity - perfect sample}
\end{equation}

With this notation we note that the error in the approximation of
the vortex dynamics can be decomposed into two sources, the stability
of the evolution under empirical dynamics and the inconsistency in
the empirical approximation of the true velocity, as 
\begin{align}
\begin{split}\hat{X}_{N,h,\delta,R}^{k}(i,t)&-X_{\delta,R}^{k}(ih,t) =\int_{0}^{t}\partial_{s}\left(\hat{X}_{N,h,\delta,R}^{k}(i,s)-X_{\delta,R}^{k}(ih,s)\right)\mathrm{d}s\\
 & =\int_{0}^{t}\underbrace{\hat{u}_{N,h,\delta,R}\left(\hat{X}_{N,h,\delta,R}^{k}(i,s),s\right)-u_{N,h,\delta,R}\left(X_{\delta,R}^{k}(ih,s),s\right)}_{(\mathrm{I})~\text{stability error}}\mathrm{d}s\\
 & \;+\int_{0}^{t}\underbrace{u_{N,h,\delta,R}\left(X_{\delta,R}^{k}(ih,s),s\right)-u_{\delta,R}\left(X_{\delta,R}^{k}(ih,s),s\right)}_{(\mathrm{II})~\text{consistency error}}\mathrm{d}s.
\end{split}
\label{Monte Carlo error decomposition 1}
\end{align}

\subsection{Stability error}

First we bound the stability error $\left(\mathrm{I}\right)$, i.e.
the difference between the empirical velocity of the particle approximation and the empirical velocity computed from independent samples of the mean-field dynamics
\begin{equation*}
    \left|\hat{u}_{N,h,\delta,R}\left(\hat{X}_{N,h,\delta,R}^{l}(i,t),t\right)-u_{N,h,\delta,R}\left(X_{\delta,R}^{l}(ih,t),t\right)\right|.
\end{equation*} 
We note that this can be bounded above by the following three terms 
\begin{align*}
 &\sum_{j\in\mathbb{Z}^{2}}\frac{1}{N}\sum_{k=1}^{N}\Bigl|K_{\delta}\left(\hat{X}_{N,h,\delta,R}^{k}(j,t)-\hat{X}_{N,h,\delta,R}^{l}(i,t)\right) \\
 &\qquad-K_{\delta}\left(X_{\delta,R}^{k}(jh,t)-X_{\delta,R}^{l}(ih,t)\right)\Bigr|\left|\omega_{0}(jh)\right|h^{2}\\
 &+\int_{0}^{t}\sum_{j\in\mathbb{Z}^{2}}\frac{1}{N}\sum_{k=1}^{N}\Bigl|K_{\delta}\left(\hat{X}_{N,h,\delta,R}^{k}(j,t)-\hat{X}_{N,h,\delta,R}^{l}(i,t)\right)\\
 &\qquad-K_{\delta}\left(X_{\delta,R}^{k}(jh,t)-X_{\delta,R}^{l}(ih,t)\right)\Bigr|G\left(\hat{X}_{N,h,\delta,R}^{k}(j,s),s\right)\chi_{R}(jh)h^{2}\textrm{d}s\\
 &+\int_{0}^{t}\sum_{j\in\mathbb{Z}^{2}}\frac{1}{N}\sum_{k=1}^{N}\left|K_{\delta}\left(X_{\delta,R}^{k}(jh,t)-X_{\delta,R}^{l}(ih,s)\right)\right|\\ &\qquad\times\left|G\left(\hat{X}_{N,h,\delta,R}^{k}(j,s),s\right)-G\left(X_{\delta,R}^{k}(jh,s),s\right)\right|\chi_{R}(jh)h^{2}\textrm{d}s.
\end{align*}

Since $K_{\delta}$ is Lipschitz continuous and bounded for fixed
$\delta>0$, this implies that 
\begin{equation*}
\mathbb{E}\left[\left|\hat{u}_{N,h,\delta,R}\left(\hat{X}_{N,h,\delta,R}^{l}(i,t),t\right)-u_{N,h,\delta,R}\left(X_{\delta,R}^{l}(ih,t),t\right)\right|\right]
\end{equation*}
is bounded by 
\begin{align*}
\begin{split} 
\left(2\mathrm{Lip}\left(K_{\delta}\right)\left(\sum_{j\in\mathbb{Z}^{2}}\left|\omega_{0}(jh)\right|h^{2}+t\left\Vert G\right\Vert _{\infty}\sum_{j\in\mathbb{Z}^{2}}\chi(jh)h^{2}\right)+2t\left\Vert K_{\delta}\right\Vert _{\infty}\mathrm{Lip}\left(G\right)\sum_{j\in\mathbb{Z}^{2}}\chi(jh)h^{2}\right)&\\
\times\sup_{s\leq t,j\in\mathbb{Z}^{2}}\mathbb{E}\left[\left|\hat{X}_{N,h,\delta,R}^{k}(j,s)-X_{\delta,R}^{k}(jh,s)\right|\right]&
\end{split}
\end{align*}
for $i\in\mathbb{Z}^{2}$ and $t\in[0,T]$.

This estimate is still dependent on the spatial discretisation parameter $h$, due to the sums over the discrete grid. To eliminate this dependence we make use of the following lemma.
\begin{lem}
\label{Lemma - spatial discretisation error} For all $f\in C^{2}\left(\mathbb{R}^{2},\mathbb{R}\right)$
with compact support, the following estimate 
\begin{align*}
\left|\sum_{j\in\mathbb{Z}^{2}}f(jh)h^{2}-\int_{\mathbb{R}^{2}}f(x)\mathrm{d}x\right|\leq\frac{52}{(2\pi)^{2}}\max_{i=1,2}\left\Vert \partial_{x_{i}}^{2}f\right\Vert _{1}h^{2}
\end{align*}
holds. 
\end{lem}

\begin{proof}
See \citep[Lemma 2.2]{AndersonGreengard1985}. 
\end{proof}
As an immediate consequence we derive for $\omega_{0}$ 
\begin{align}
\begin{split}
\sum_{j\in\mathbb{Z}^{2}}\left|\omega_{0}(jh)\right|h^{2}\leq\left\Vert \omega_{0}\right\Vert _{1}+\frac{52}{(2\pi)^{2}}h^{2}\left\Vert D^{2}\omega_{0}\right\Vert _{1}.\end{split}
\label{interpolation error omega and chi}
\end{align}
Furthermore, we note that 
\begin{align*}
\sum_{j\in\mathbb{Z}^{2}}\chi_{R}(jh)h^{2}\leq\pi(R+1+h)^{2}
\end{align*}
since $\chi_{R}(z)=0$ for all $|z|>R+1$.

Thus, if we define the constant $C_{\mathrm{stab}} = C_{\mathrm{stab}}(\delta,R,T,\omega_{0},G)$ by the following
\begin{align}
\begin{split} 2\mathrm{Lip}\left(K_{\delta}\right)\left(\left\Vert \omega_{0}\right\Vert _{1}+\frac{52}{(2\pi)^{2}}h^{2}\left\Vert D^{2}\omega_{0}\right\Vert _{1}\right) +2\mathrm{Lip}\left(K_{\delta}\right)T\left\Vert G\right\Vert _{\infty}\pi(R+1+h)^{2} & \\
+2T\left\Vert K_{\delta}\right\Vert _{\infty}\mathrm{Lip}\left(G\right)\pi(R+1+h)^{2}&,
\end{split}
\label{stability constant}
\end{align}
we derive the stability estimate for $t\leq T$ 
\begin{align}
\begin{split} & \sup_{s\leq t,j\in\mathbb{Z}^{2}}\mathbb{E}\left[\left|\hat{u}_{N,h,\delta,R}\left(\hat{X}_{N,h,\delta}^{k}(j,s),s\right)-u_{N,h,\delta,R}\left(X_{\delta}^{k}(jh,s),s\right)\right|\right]\\
 & \leq C_{\mathrm{stab}}\sup_{s\leq t,j\in\mathbb{Z}^{2}}\mathbb{E}\left[\left|\hat{X}_{N,h,\delta}^{k}(j,s)-X_{\delta}^{k}(jh,s)\right|\right].
\end{split}
\label{stability of the scheme - estimate}
\end{align}

As an immediate consequence we derive from \eqref{Monte Carlo error decomposition 1} that 
\begin{align*}
\sup_{s\leq t,i\in\mathbb{Z}^{2}} \mathbb{E}\left[\left| \hat{X}_{N,h,\delta,R}^{k}(i,s)-X_{\delta,R}^{k}(ih,s)\right|\right]
\end{align*}
is bounded above by
\begin{align*}
&C_{\mathrm{stab}}\int_{0}^{t}\sup_{r\leq s,j\in\mathbb{Z}^{2}}\mathbb{E}\left[\left|\hat{X}_{N,h,\delta,R}^{k}(j,r)-X_{\delta,R}^{k}(jh,r)\right|\right]\mathrm{d}s\\
&+\int_{0}^{t}\sup_{i\in\mathbb{Z}^{2}}\mathbb{E}\left[\left|u_{N,h,\delta,R}\left(X_{\delta,R}^{k}(ih,s),s\right)-u_{\delta,R}\left(X_{\delta,R}^{k}(ih,s),s\right)\right|\right]\mathrm{d}s,
\end{align*}
and therefore by 
\begin{align}
e^{TC_{\mathrm{stab}}}\int_{0}^{T}\sup_{i\in\mathbb{Z}^{2}}\mathbb{E}\left[\left|u_{N,h,\delta,R}\left(X_{\delta,R}^{k}(ih,t),t\right)-u_{\delta,R}\left(X_{\delta,R}^{k}(ih,t),t\right)\right|\right]\mathrm{d}t,
\label{Monte Carlo error decomposition 2}
\end{align}
as follows from the Gr{\"o}nwall Lemma.

Thus it only remains to bound the inconsistency of the empirical velocity approximation.

\subsection{Consistency error}

\label{subsection - consistency error}

The consistency error can be split up into two parts: an error in the law of large numbers (i.e. the difference between the real and the empirical mean) and a discretisation error. First we investigate the law of large numbers for the homogeneous part of the empirical velocity. Note that by the tower property of the conditional expectation and the boundedness of $K_{\delta}$, we write 
\begin{align*}
 & \mathbb{E}\left[\left|\frac{1}{N}\sum_{k=1}^{N}\sum_{j\in\mathbb{Z}^{2}}K_{\delta,R}\left(X_{\delta,R}^{k}(jh,t)-X_{\delta,R}^{l}(ih,t)\right)\omega_{0}(jh)h^{2}\right.\right.\\
 & \qquad-\left.\left.\mathbb{E}\left[\sum_{j\in\mathbb{Z}^{2}}K_{\delta}\left(X_{\delta,R}^{k}(jh,t)-x\right)\omega_{0}(jh)h^{2}\right]_{x=X_{\delta,R}^{l}(ih,t)}\right|\right]\\
 & =\mathbb{E}\left[\mathbb{E}\left[\left|\frac{1}{N}\sum_{k=1}^{N}\sum_{j\in\mathbb{Z}^{2}}K_{\delta}\left(X_{\delta,R}^{k}(jh,t)-x\right)\omega_{0}(jh)h^{2}\right.\right.\right.\\
 & \qquad-\left.\left.\left.\left.\mathbb{E}\left[\sum_{j\in\mathbb{Z}^{2}}K_{\delta}\left(X_{\delta,R}^{k}(jh,t)-x\right)\omega_{0}(jh)h^{2}\right]\right|~\right|~X_{\delta,R}^{l}(ih,t)=x\right]\right]
 \end{align*}
and bound it above by
\begin{align*}
&
\sum_{j\in\mathbb{Z}^{2}}\mathbb{E}\Biggr[\mathbb{E}\Biggr[\Bigg|\frac{1}{N}\sum_{k\neq l}K_{\delta}\left(X_{\delta,R}^{k}(jh,t)-x\right)\\
&\qquad-\mathbb{E}\left[\sum_{j\in\mathbb{Z}^{2}}K_{\delta}\left(X_{\delta,R}^{k}(jh,t)-x\right)\right]\Bigg|~\Bigg|~X_{\delta,R}^{l}(ih,t)=x\Biggr]\Biggr]\left|\omega_{0}(jh)h^{2}\right|\\
&+\frac{\left\Vert K_{\delta}\right\Vert _{\infty}\left(\left\Vert \omega_{0}\right\Vert _{1}+\frac{52}{(2\pi)^{2}}h^{2}\left\Vert D^{2}\omega_{0}\right\Vert _{1}\right)}{N}.
\end{align*}

Next we note that by the independence of the samples $\left(X^{k}\right)_{k=1,\ldots,N}$,
this can be estimated by 
\begin{align*}
&\sum_{j\in\mathbb{Z}^{2}}\sup_{x\in\mathbb{R}^{2}}\mathbb{E}\Biggr[\Bigg|\frac{1}{N}\sum_{k\neq l}K_{\delta}\left(X_{\delta,R}^{k}(jh,t)-x\right) \\
&
\qquad-\mathbb{E}\left[\sum_{j\in\mathbb{Z}^{2}}K_{\delta}\left(X_{\delta,R}^{k}(jh,t)-x\right)\right]\Bigg|\Biggr]\left|\omega_{0}(jh)h^{2}\right|\\
&+\frac{\left\Vert K_{\delta}\right\Vert _{\infty}\left(\left\Vert \omega_{0}\right\Vert _{1}+\frac{52}{(2\pi)^{2}}h^{2}\left\Vert D^{2}\omega_{0}\right\Vert _{1}\right)}{N}\\
& \leq\frac{\left\Vert K_{\delta}\right\Vert _{\infty}\left(1+\frac{\sqrt{N-1}}{N}\right)\left(\left\Vert \omega_{0}\right\Vert _{1}+\frac{52}{(2\pi)^{2}}h^{2}\left\Vert D^{2}\omega_{0}\right\Vert _{1}\right)}{\sqrt{N-1}}.
\end{align*}

Hereby the last inequality follows from the standard estimation of
the law of large numbers using the variance, together with the boundedness of $K_{\delta}$.

Similarly we can bound 
\begin{align*}
 \mathbb{E}\left[\left|\frac{1}{N}\sum_{k=1}^{N}\sum_{j\in\mathbb{Z}^{2}}K_{\delta}\left(X_{\delta,R}^{k}(jh,t)-X_{\delta,R}^{l}(ih,t)\right)G\left(X_{\delta,R}^{k}(jh,s),s\right)\chi_{R}(jh)h^{2}\right.\right.&\\
 \qquad-\left.\left.\left.\mathbb{E}\left[\sum_{j\in\mathbb{Z}^{2}}K_{\delta}\left(X_{\delta,R}(jh,t)-x\right)G\left(X_{\delta,R}(jh,s),s\right)\chi_{R}(jh)h^{2}\right]\right|_{x=X_{\delta,R}^{l}(ih,t)}\right|\right]&\\
 \leq\frac{\left(1+1/\sqrt{N}\right)\left\Vert K_{\delta}\right\Vert _{\infty}\left\Vert G\right\Vert _{\infty}(R+1)^{2}\pi}{\sqrt{N-1}}&.
\end{align*}

With this in mind we define the constant $C_{\mathrm{LLN}} = C_{\mathrm{LLN}}(\delta,R,T,\omega_{0},G)$ as 
\begin{align*}
\begin{split}4 & \left\Vert K_{\delta}\right\Vert _{\infty}\left(\left\Vert \omega_{0}\right\Vert _{1}+\frac{52}{(2\pi)^{2}}h^{2}\left\Vert D^{2}\omega_{0}\right\Vert _{1}+T\left\Vert G\right\Vert _{\infty}(R+1)^{2}\pi\right)\end{split}
\end{align*}
and derive that
\begin{align*}
\begin{split} & \sup_{i\in\mathbb{Z}^{2}}\mathbb{E}\Bigg[\Bigg|\frac{1}{N}\sum_{k=1}^{N}\sum_{j\in\mathbb{Z}^{2}}K_{\delta}\left(X_{\delta,R}^{k}(jh,t)-X_{\delta,R}^{l}(jh,t)\right) \\
&\qquad\times\left(\omega_{0}(jh)+\int_{0}^{t}G\left(X_{\delta,R}^{k}(jh,s),s\right)\mathrm{d}s\chi_{R}(jh)\right)h^{2}\\
 & \quad-\mathbb{E}\Biggr[\sum_{j\in\mathbb{Z}^{2}}K_{\delta}\left(X_{\delta,R}(jh,t)-x\right)\biggl(\omega_{0}(jh) \\
 & \qquad+\int_{0}^{t}G\left(X_{\delta,R}(jh,s),s\right)\mathrm{d}s\chi_{R}(jh)\biggl)h^{2}\Biggr]\Bigg|_{x=X_{\delta,R}^{l}(jh,t)}\Bigg|\Biggr] \leq \frac{C_{\mathrm{LLN}}}{\sqrt{N}}. 
\end{split}
\end{align*}

Next we need to estimate the consistency error of the spatial discretisation using Lemma \ref{Lemma - spatial discretisation error}. To
this end, we first note that for every $i=1,2$ and $k=1,2$ 
\begin{align*}
\partial_{x_{i}}^{k}u_{\delta,R}(x,t)=\int_{\mathbb{R}^{2}}\mathbb{E}\left[\partial_{x_{i}}^{k}K_{\delta}\left(X_{\delta,R}(\eta,t)-x\right)\left(\omega_{0}(\eta)+\int_{0}^{t}G\left(X_{\delta,R}(\eta,s),s\right)\mathrm{d}s~\chi_{R}(\eta)\right)\right]\mathrm{d}\eta,
\end{align*}
which in turn implies 
\begin{align}
\sup_{t\leq T,x\in\mathbb{R}^{2}}\left|\partial_{x_{i}}^{k}u_{\delta,R}(x,t)\right|\leq\left\Vert \partial_{x_{i}}^{k}K_{\delta}\right\Vert _{\infty}\left(\left\Vert \omega_{0}\right\Vert _{L^{1}}+T\left\Vert G\right\Vert _{\infty}(R+1)^{2}\pi\right),\label{bound for derivatives of uDelta}
\end{align}
for $i=1,2$ and $k=1,2$.

Next we note that since 
\begin{align*}
X_{\delta,R}(x,t)=\int_{0}^{t}u_{\delta,R}\left(X_{\delta,R}(x,s),s\right)\mathrm{d}s+\sqrt{2\nu}B_{t},
\end{align*}
we derive, using the Einstein sum convention, 
\begin{align*}
\partial_{x_{i}}X_{\delta,R}(x,t)=e_{i}+\int_{0}^{t}\partial_{x_{k}}u_{\delta,R}\left(X_{\delta,R}(x,s),s\right)\partial_{x_{i}}X_{\delta,R}^{k}(x,s)\mathrm{d}s
\end{align*}
and 
\begin{align*}
\partial_{x_{i}x_{j}}^{2}X_{\delta,R}(x,t) & =\int_{0}^{t}\partial_{x_{k}x_{l}}^{2}u_{\delta,R}\left(X_{\delta,R}(x,s),s\right)\partial_{x_{i}}X_{\delta,R}^{k}(x,s)\partial_{x_{j}}X_{\delta,R}^{l}(x,s)\mathrm{d}s\\
 & \;+\int_{0}^{t}\partial_{x_{k}}u_{\delta,R}\left(X_{\delta,R}(x,s),s\right)\partial_{x_{i}x_{j}}^{2}X_{\delta,R}(x,s)\mathrm{d}s.
\end{align*}

As a consequence we note that by the Gr{\"o}nwall Lemma 
\begin{align*}
\left\Vert DX_{\delta,R}(\cdot,t)\right\Vert _{\infty} \leq\exp\left(\int_{0}^{t}\left\Vert Du_{\delta,R}(\cdot,s)\right\Vert _{\infty}\mathrm{d}s\right),
\end{align*}
and, similarly, 
\begin{align*}
\left\Vert D^{2}X_{\delta,R}(\cdot,t)\right\Vert _{\infty} \leq\exp\left(3\int_{0}^{t}\left\Vert Du_{\delta,R}(\cdot,s)\right\Vert _{\infty}\mathrm{d}s\right)\int_{0}^{t}\left\Vert D^{2}u_{\delta,R}(\cdot,s)\right\Vert _{\infty}\mathrm{d}s.
\end{align*}
Therefore, 
\begin{align*}
\left\Vert \partial_{\eta_{i}}K_{\delta}\left(X_{\delta,R}(\cdot,t)-x\right)\right\Vert _{\infty} & \leq\left\Vert DK_{\delta}\right\Vert _{\infty}\exp\left(\int_{0}^{t}\left\Vert Du_{\delta,R}(\cdot,s)\right\Vert _{\infty}\mathrm{d}s\right)\\
\left\Vert \partial_{\eta_{i}\eta_{i}}^{2}K_{\delta}\left(X_{\delta,R}(\cdot,t)-x\right)\right\Vert _{\infty} & \leq\left(\left\Vert DK_{\delta}\right\Vert _{\infty}+\left\Vert DK_{\delta}\right\Vert _{\infty}\right)\left(1+\int_{0}^{t}\left\Vert D^{2}u_{\delta,R}(\cdot,s)\right\Vert _{\infty}\mathrm{d}s\right)\\
 & \qquad\times\exp\left(3\int_{0}^{t}\left\Vert Du_{\delta,R}(\cdot,s)\right\Vert _{\infty}\mathrm{d}s\right).
\end{align*}

Then, setting 
\begin{align*}
f(\eta):=\mathbb{E}\left[K_{\delta}\left(X_{\delta,R}(\eta,t)-x\right)\left(\omega_{0}(\eta)+\int_{0}^{t}G\left(X_{\delta,R}(\eta,s),s\right)\mathrm{d}s~\chi_{R}(\eta)\right)\right],
\end{align*}
its second derivative $\partial_{\eta_{i}\eta_{i}}^{2}f(\eta)$ is bounded by
\begin{align*}
 &\left(\left\Vert K_{\delta}\right\Vert _{\infty}+\left\Vert DK_{\delta}\right\Vert _{\infty}+\left\Vert D^{2}K_{\delta}\right\Vert _{\infty}\right)\left(1+\int_{0}^{t}\left\Vert D^{2}u_{\delta,R}(\cdot,s)\right\Vert _{\infty}\mathrm{d}s\right)^{2}\\
 & \;\exp\left(6\int_{0}^{t}\left\Vert Du_{\delta,R}(\cdot,s)\right\Vert _{\infty}\mathrm{d}s\right)\left(1+\int_{0}^{t}\left\Vert G(\cdot,s)\right\Vert _{\infty}+\left\Vert DG(\cdot,s)\right\Vert _{\infty}+\left\Vert D^{2}G(\cdot,s)\right\Vert _{\infty}\mathrm{d}s\right)\\
 & \;\left(\left|\omega_{0}(\eta)\right|+\left|\partial_{\eta_{i}}\omega_{0}(\eta)\right|+\left|\partial_{\eta_{i}\eta_{i}}^{2}\omega_{0}(\eta)\right|+\left|\chi_{R}(\eta)\right|+\left|\partial_{\eta_{i}}\chi_{R}(\eta)\right|+\left|\partial_{\eta_{i}\eta_{i}}^{2}\chi_{R}(\eta)\right|\right).
\end{align*}

Using Lemma \ref{Lemma - spatial discretisation error} and the bound
\eqref{bound for derivatives of uDelta} for the derivatives of $u_{\delta,R}$, we therefore can bound the discretisation error 
\begin{align}
 & \left|\sum_{j\in\mathbb{Z}^{2}}f(jh)h^{2}-\int_{\mathbb{R}^{2}}f(\eta)\mathrm{d}\eta\right|\leq\frac{52}{(2\pi)^{2}}\max_{i=1,2}\left\Vert \partial_{x_{i}}^{2}f\right\Vert _{1}h^{2}\leq C_{\mathrm{dis}}h^{2},\label{discretisation error estimate}
\end{align}
where the constant $C_{\mathrm{dis}} = C_{\mathrm{dis}}(\delta,R,t,\omega_{0},G)$ is given by the following expression
\begin{align*}
\begin{split} 
& \frac{52}{(2\pi)^{2}} \max_{i=1,2}\left(\left\Vert \omega_{0}\right\Vert _{1}+\left\Vert \partial_{\eta_{i}}\omega_{0}\right\Vert _{1}+\left\Vert \partial_{\eta_{i}\eta_{i}}^{2}\omega_{0}\right\Vert _{1}+\left\Vert \chi_{R}(\eta)\right\Vert _{1}+\left\Vert \partial_{\eta_{i}}\chi_{R}\right\Vert _{1}+\left\Vert \partial_{\eta_{i}\eta_{i}}^{2}\chi_{R}\right\Vert _{1}\right)\\
 & \exp\left(6t\left\Vert DK_{\delta}\right\Vert _{\infty}\left(\left\Vert \omega_{0}\right\Vert _{L^{1}}+t\left\Vert G\right\Vert _{\infty}(R+1)^{2}\pi\right)\right) \biggl(1+\int_{0}^{t}\Bigl(\left\Vert G(\cdot,s)\right\Vert _{\infty}+\left\Vert DG(\cdot,s)\right\Vert _{\infty}\\
 & +\left\Vert D^{2}G(\cdot,s)\right\Vert _{\infty}\Bigl)\mathrm{d}s\biggl) \left(\left\Vert K_{\delta}\right\Vert _{\infty}+\left\Vert DK_{\delta}\right\Vert _{\infty}+\left\Vert D^{2}K_{\delta}\right\Vert _{\infty}\right)\bigl(1+t\left\Vert D^{2}K_{\delta}\right\Vert _{\infty}\\
 & 
 \left(\left\Vert \omega_{0}\right\Vert _{L^{1}}+t\left\Vert G\right\Vert _{\infty}(R+1)^{2}\pi\right)\bigl)^{2}.
\end{split}
\end{align*}
Note that this bound is independent of $x\in\mathbb{R}^{2}$! Thus
for the consistency error we derive the estimate 
\begin{align}
\begin{split} \sup_{t\leq T,i\in\mathbb{Z}^{2}}\mathbb{E}\left[\left|u_{N,h,\delta,R}\left(X_{\delta}^{l}(ih,t),t\right)-u_{\delta,R}\left(X_{\delta}^{l}(ih,t),t\right)\right|\right] \leq\frac{C_{\mathrm{LLN}}}{\sqrt{N}}+C_{\mathrm{dis}}h^{2}.
\end{split}
\label{consistency estimate}
\end{align}

\subsection{Total error of interacting vortex dynamics and their empirical velocities}

Combining the inconsistency \eqref{consistency estimate} with \eqref{Monte Carlo error decomposition 2}, we thus derive the following error estimate for the Monte-Carlo approximation
of the (regularised and localised) vorticity dynamics 
\begin{align}
\begin{split} \sup_{t\leq T,i\in\mathbb{Z}^{2},k\leq N}&\mathbb{E}\left[\left|\hat{X}_{N,h,\delta,R}^{k}(i,t)-X_{\delta,R}^{k}(ih,t)\right|\right]
\leq T e^{TC_{\mathrm{stab}}} \left(\frac{C_{\mathrm{LLN}}}{\sqrt{N}}+C_{\mathrm{dis}}h^{2}\right).
\end{split}
\label{error vortex dynamics - Monte Carlo}
\end{align}

\begin{rem*}
\label{Remark - pointwise convergence of particles} Note that \eqref{error vortex dynamics - Monte Carlo} only gives us pointwise convergence of $\hat{X}_{N,h,\delta,R}$ to $X_{\delta,R}^{k}$ and not the usual pathwise convergence that is often investigated in Propagation of Chaos results. However, pathwise convergence of the particles is actually not needed to show the convergence of the computed Monte Carlo velocity. 
\end{rem*}

The estimate \eqref{error vortex dynamics - Monte Carlo} can in turn
be used to bound the error of the Monte Carlo approximation to the
(regularised and localised) velocity $u_{\delta,R}$, i.e.
\begin{align*}
\sup_{t\leq T,x\in\mathbb{R}^{2}}\mathbb{E}\left[\left|\hat{u}_{N,h,\delta,R}(x,t)-u_{\delta,R}(x,t)\right|\right],
\end{align*}
as the same calculations that were used to derive the key stability estimate \eqref{stability of the scheme - estimate} give us the following upper bound
\begin{align*}
&\sup_{t\leq T,x\in\mathbb{R}^{2}}\mathbb{E}\left[\left|\hat{u}_{N,h,\delta,R}(x,t)-u_{N,h,\delta,R}\left(x,t\right)\right|\right]+\sup_{t\leq T,x\in\mathbb{R}^{2}}\mathbb{E}\left[\left|u_{N,h,\delta,R}\left(x,t\right)-u_{\delta,R}\left(x,t\right)\right|\right]\\
& \leq C_{\mathrm{stab}}\sup_{t\leq T,x\in\mathbb{R}^{2}}\mathbb{E}\left[\left|\hat{X}_{N,h,\delta,R}^{k}(i,t)-X_{\delta,R}^{k}(ih,t)\right|\right]+\sup_{t\leq T,x\in\mathbb{R}^{2}}\mathbb{E}\left[\left|u_{N,h,\delta,R}\left(x,t\right)-u_{\delta,R}\left(x,t\right)\right|\right].
\end{align*}

The second term quantifies the inconsistency of the empirical velocity,
it can be estimated just as in Subsection \ref{subsection - consistency error}. Together with the bounds for the error of the dynamics \eqref{error vortex dynamics - Monte Carlo} we thus derive the error bound 
\begin{align}
\begin{split} \sup_{t\leq T,x\in\mathbb{R}^{2}}\mathbb{E}\left[\left|\hat{u}_{N,h,\delta,R}(x,t)-u_{\delta,R}(x,t)\right|\right] \leq\left(TC_{\mathrm{stab}}e^{TC_{\mathrm{stab}}}+1\right)\left(\frac{C_{\mathrm{LLN}}}{\sqrt{N}}+C_{\mathrm{dis}}h^{2}\right).
\end{split}
\label{error velocities - Monte Carlo, pointwise}
\end{align}

By the Markov inequality, this indeed gives uniform stochastic guarantees,
in the sense that the following tail bound holds for all $t\leq T$
and $x\in\mathbb{R}^{2}$ 
\begin{align*}
\mathbb{P}\left(\left|\hat{u}_{N,h,\delta,R}(x,t)-u_{\delta,R}(x,t)\right|>\epsilon\right)  \leq\frac{1}{\epsilon} \left(TC_{\mathrm{stab}}e^{TC_{\mathrm{stab}}}+1\right)\left(\frac{C_{\mathrm{LLN}}}{\sqrt{N}}+C_{\mathrm{dis}}h^{2}\right).
\end{align*}

Combining \eqref{error velocities - Monte Carlo, pointwise} with
the robustness result in Lemma \ref{Lemma - robustness in delta and R}
shows that $\hat{u}_{N,h,\delta,R}$ approximates $u$ pointwise.
Let us formally state this fact in the following Theorem.
\begin{thm}
\label{Theorem - convergence result} For all $\delta,R>0$, there
exist appropriately chosen $N(\delta,R)\in\mathbb{N}$ and $h(\delta,R)>0$,
with 
\begin{align*}
N(\delta,R)\xrightarrow{\delta\to0,R\to\infty}\infty\quad\text{ and }\quad h(\delta,R)\xrightarrow{\delta\to0,R\to\infty}0
\end{align*}
such that 
\begin{align}
\sup_{\substack{N\geq N(\delta,R) \\ h\leq h(\delta,R)}}\sup_{\substack{t\leq T \\ x\in\mathbb{R}^{2}}}~\mathbb{E}\left[\left|\hat{u}_{N,h,\delta,R}(x,t)-u_{\delta,R}(x,t)\right|\right]\xrightarrow{\delta\to0,R\to\infty}0.\label{convergence result in limit form}
\end{align}
\end{thm}

The proof of Theorem \ref{Theorem - convergence result} is a trivial
combination of the estimate \eqref{error velocities - Monte Carlo, pointwise} and Lemma \ref{Lemma - robustness in delta and R}. It shows that if the number of particles $N$ and the grid size $h$ are appropriately large/small in relation to the regularisation and cutoff parameters $\delta$ and $R$, then our Monte-Carlo velocity $\hat{u}_{N,h,\delta,R}$ converges to the true velocity. Using the concrete form of $C_{\mathrm{stab}},C_{\mathrm{LLN}}$ and $C_{\mathrm{dis}}$ as defined in this section also would allow us to bound $N(h,\delta)$ and $h(\delta,R)$ and give concrete convergence rates in \eqref{convergence result in limit form}. However these bounds are far from optimal and for the sake of (notational) simplicity we refrain from computing them.

\begin{rem*}
The estimate \eqref{error velocities - Monte Carlo, pointwise} only gives us a pointwise convergence of the Monte Carlo approximation $\hat{u}_{N,h,\delta,R}$ to $u_{\delta,R}$. To derive uniform convergence one would need a suitable law of large numbers that is uniform in time and space, i.e. one would need a suitable estimate of 
\begin{align*}
\sup_{t\leq T,x\in\mathbb{R}^{2}}\left|u_{N,h,\delta,R}(x,t)-\mathbb{E}\left[u_{N,h,\delta,R}(x,t)\right]\right|
\end{align*}
that converges to zero for $N\to\infty$ in a suitable probabilistic
sense. Such uniform laws of large numbers have been achieved in various
works in the context of kernel density estimation and statistics on
separable Banach spaces. In the same manner one proves the convergence
of the PiCRV scheme. The only difference is that in the consistency
estimate one uses a law of large numbers for non identically but independent
random variables, just as in the derivation of the scheme in Section
\ref{Random vortex method -- flows without constraint}. This then gives the error estimate
\begin{align}
\begin{split} \sup_{t\leq T,x\in\mathbb{R}^{2}}\mathbb{E}\left[\left|\tilde{u}_{h,\delta,R}(x,t)-u_{\delta,R}(x,t)\right|\right]\leq C\left(Ce^{TC}~T+1\right)\left(h+h^{2}\right),
\end{split}
\label{error velocities - PiC pointwise}
\end{align}
for some constant $C = C(\delta,R,T,\omega_{0},G)>0$. Note that the term
of order $h$ comes from the estimate of the standard deviation of
the discretised velocity! 
\end{rem*}

\section{Incompressible viscous flows past a wall}
\label{Incompressible viscous flows past a wall}

In the next sections we consider incompressible viscous flows past
a solid wall. We aim to develop Monte-Carlo method for such flows
based on a random vortex formulation which will be established in
this section.

Suppose the flow with viscosity constant $\nu>0$ is constrained in
the upper half space $D\subset\mathbb{R}^{d}$ where $x_{d}>0$, so
that the solid wall is modelled by the hyperspace $\partial D$ where
$x_{d}=0$. Let $V(x,t)$, for $x\in D$ and $t\geq 0$, be a time-dependent
vector field, which is bounded and differentiable up to the boundary
$\partial D$, satisfying $V(x,t)=0$ for $x\in\partial D$ and
for $t\geq0$. $V(x,t)$ is extended to the whole space $\mathbb{R}^{d}$
via reflection, so that 
\begin{equation}
V^{i}(x,t)=V^{i}(\bar{x},t)\quad\textrm{ for }i=1,\ldots,d-1\textrm{ and }V^{d}(x,t)=-V^{d}(\bar{x},t).\label{v1-01}
\end{equation}

Let $L_{V}=\nu\Delta+V\cdot\nabla$ which is an elliptic operator
of second order on $\mathbb{R}^{d}$, and $p_{V}^{D}(\tau,\xi;t,x)$
denote the transition probability density function of the $L_{V}$-diffusion
killed on leaving the domain $D$,  which is the Green function of
$\frac{\partial}{\partial t}+L_{V}$ on $D$ satisfying the Dirichlet
condition, cf. \citep{Freidlin1985} and \citep{StroockVaradhan}
for example.

From now on we assume that $\nabla\cdot V(x,t)=0$ for $x_{d}>0$.
Then the extension of \textbf{$V(x,t)$} via the reflection about
$x_{d}=0$ is divergence-free in $\mathbb{R}^{d}$, i.e. $\nabla\cdot V(\cdot,t)=0$
in $\mathbb{R}^{d}$ in the distribution sense. Since $L_{V}^{\star}=L_{-V}$,
$p_{V}^{D}(\tau,\xi;t,x)$ is the Green function of $\frac{\partial}{\partial t}-L_{-V}$
on $D$ subject to the Dirichlet boundary condition. Since $V(\bar{x},t)=\overline{V(x,t)}$
for all $x$ and $t\geq0$, $p_{V}(\tau,\xi;t,x)=p_{V}(\tau,\bar{\xi};t,\bar{x})$
for $\xi,x\in\mathbb{R}^{d}\textrm{ and }t>\tau\geq0$, and 
\begin{equation}
p_{V}^{D}(\tau,\xi;t,x)=p_{V}(\tau,\xi;t,x)-p_{V}(\tau,\xi;t,\bar{x})\label{f-04}
\end{equation}
for $\xi,x\in D\textrm{ and }t>\tau\geq0$.

Let $f(x,t)$ be a solution to the initial and boundary value problem
of the parabolic equation 
\begin{equation}
\left(L_{-V}-\frac{\partial}{\partial t}\right)f+g=0\textrm{ in }D\times[0,\infty),\quad f(\cdot,0)=\varphi,\label{eq:cauchy prm1-1}
\end{equation}
subject to the Dirichlet boundary condition that 
\begin{equation}
\left.f(x,t)\right|_{\partial D}=0\quad\textrm{ for }x\in\partial D\;\textrm{ and }t>0,\label{D-c-s1}
\end{equation}
where $g(x,t)$ is a continuous and bounded function on $D\times[0,\infty)$.

We next derive a stochastic integral representation theorem for solutions to the problem \eqref{eq:cauchy prm1-1}, \eqref{D-c-s1}. Let $X^{\xi}=(X_{t}^{\xi})_{t\geq0}$ (for $\xi\in\mathbb{R}^{d}$)
be the solution to SDE: 
\begin{equation}
\textrm{d}X_{t}=V(X_{t},t)\textrm{d}t+\sqrt{2\nu}\textrm{d}B_{t},\quad X_{0}=\xi,\label{b-SDE}
\end{equation}
where $(B_{t})$ is a standard $d$-dimensional Brownian motion on
some probability space.
\begin{thm}
\label{thm23-5} The solution $f(x,t)$ to the problem \eqref{eq:cauchy prm1-1}, \eqref{D-c-s1} possesses the following stochastic integral representation
\begin{align}
\begin{split}
f(x,t) & =\int_{D}\left(p_{V}(0,\eta;t,x)-p_{V}(0,\bar{\eta};t,x)\right)\varphi(\eta)\textrm{d}\eta \\
 & +\int_{0}^{t}\int_{D}\mathbb{E}\left[\left.1_{\{s>\gamma_{t}(X^{\eta})\}}g(X_{s}^{\eta},s)\right|X_{t}^{\eta}=x\right]p_{V}(0,\eta;t,x)\textrm{d}\eta\textrm{d}s\label{W-aa2-1-1}
\end{split}
\end{align}
for $t>0$ and $x\in D$, where $\gamma_{t}(X^{\eta}) = \sup\{s\in(0,t):X_{s}^{\eta}\in D\}$ respectively. 
\end{thm}

\begin{proof}
The proof relies on the duality of conditional laws of diffusions,
established in \citep{ProcA-paper2}. Take arbitrary fixed $t>0$.
For simplicity $\tilde{X}_{s}^{\xi}$ denotes the unique (weak) solution
$\tilde{X}(t,\xi;s)$ to SDE 
\begin{equation}
\textrm{d}\tilde{X}=-V^{t}(\tilde{X},s)\textrm{d}s+\sqrt{2\nu}\textrm{d}B,\quad\tilde{X}(t,x;0)=x.\label{ba-s01-1-1}
\end{equation}
Let $Y_{s}=f(\tilde{X}_{s\wedge T_{\xi}}^{\xi},t-s)=1_{\{s<T_{\xi}\}}f(\tilde{X}_{s}^{\xi},t-s)$
for $s\leq t$, where $T_{\xi}=\inf\left\{ s:\tilde{X}_{s}^{\xi}\notin D\right\} $
and the second equality follows from the assumption that $f(x,s)$ vanishes
at the boundary $\partial D$. By It\^o's formula and \eqref{eq:cauchy prm1-1}, one has
\begin{equation}
Y_{s}=Y_{0}+\sqrt{2\nu}\int_{0}^{s}1_{\{r<T_{\xi}\}}\nabla f(\tilde{X}_{r}^{\xi},t-r)\cdot\textrm{d}B_{r}-\int_{0}^{s}1_{\{r<T_{\xi}\}}g(\tilde{X}_{r}^{\xi},t-r)\textrm{d}r\label{Y-aa1-1}
\end{equation}
for every $s\in[0,t]$. After taking expectation both sides we may
conclude that 
\[
f(\xi,t)=\mathbb{E}\left[f(\tilde{X}_{t}^{\xi},0)1_{\{t<T_{\xi}\}}\right]+\int_{0}^{t}\mathbb{E}\left[1_{\{s<T_{\xi}\}}g(\tilde{X}_{s}^{\xi},t-s)\right]\textrm{d}s.
\]
Note that $T_{\xi}$ is a stopping time with respect to the filtration
generated by $\tilde{X}$, so that $1_{\{s<T_{\xi}\}}$ is therefore
measurable with respect to $\tilde{X}$ running up to time $s\leq t$.
Therefore we may take conditional expectation on $\tilde{X}_{t}=\eta$,
to obtain that 
\begin{align*}
f(\xi,t) & =\int_{\mathbb{R}^{d}}\mathbb{E}\left[\left.1_{\{t<T_{\xi}\}}\right|\tilde{X}_{t}^{\xi}=\eta\right]\varphi(\eta)p_{-V^{t}}(0,\xi;t,\eta)\textrm{d}\eta\\
 & +\int_{0}^{t}\int_{\mathbb{R}^{d}}\mathbb{E}\left[\left.1_{\{s<T_{\xi}\}}g(\tilde{X}_{s}^{\xi},t-s)\right|\tilde{X}_{t}^{\xi}=\eta\right]p_{-V^{t}}(0,\xi;t,\eta)\textrm{d}\eta\textrm{d}s
\end{align*}
for every $\xi\in D$, where $p_{-V^{t}}(\tau,\xi;s,\eta)$ is the
transition probability density function of the diffusion $\tilde{X}^{\xi}$.
Since $\nabla\cdot V=0$, we can replace $p_{-V^{t}}(0,\xi;t,\eta)$
by $p_{V}(0,\eta;t,\xi)$. Hence 
\begin{align}
\begin{split}
f(\xi,t) & =\int_{D}\mathbb{E}\left[\left.1_{\{t<T_{\xi}\}}\right|\tilde{X}_{t}^{\xi}=\eta\right]\varphi(\eta) 
p_{V}(0,\eta;t,\xi)\textrm{d}\eta \\
 & +\int_{0}^{t}\int_{D}\mathbb{E}\left[\left.1_{\{s<T_{\xi}\}}g(\tilde{X}_{s}^{\xi},t-s)\right|\tilde{X}_{t}^{\xi}=\eta\right] 
 p_{V}(0,\eta;t,\xi)\textrm{d}\eta\textrm{d}s\label{W-aa2-2}
\end{split}
\end{align}
for every $\xi\in D$ and $t>0$. Next we utilise the duality of conditional
laws established in \citep{ProcA-paper2} and rewrite the conditional
expectations in terms of the diffusion process $X^{\xi}$ (starting
from $\xi$ at $\tau=0$) whose infinitesimal generator is $\nu\Delta+V\cdot\nabla$.
Hence 
\begin{align}
\begin{split}
f(\xi,t) & =\int_{D}\mathbb{E}\left[\left.1_{\{t<\zeta(X^{\eta})\}}\right|X_{t}^{\eta}=\xi\right]\varphi(\eta)p_{V}(0,\eta;t,\xi)\textrm{d}\eta \\
 & +\int_{0}^{t}\int_{D}\mathbb{E}\left[\left.1_{\{s>\gamma_{t}(X^{\eta})\}}g_{\varepsilon}(X_{s}^{\eta},s)\right|X_{t}^{\eta}=\xi\right]p_{V}(0,\eta;t,\xi)\textrm{d}\eta\textrm{d}s\label{W-aa2-1-1-1}
 \end{split}
\end{align}
where $\zeta(X^{\eta})=\inf\{s:X^{\eta}\notin D\}$ and $\gamma_{t}(X^{\eta}) = \sup\{s\in(0,t):X_{s}^{\eta}\in D\}$, which completes
the proof. 
\end{proof}

The Biot-Savart law for $D$ may be derived from the Green formula
and integration by parts. While the Green function for $D$, satisfying
the Dirichlet condition at $\partial D$, can be identified as 
\begin{equation}
\varGamma_{D}(y,x)=\varGamma(y,x)-\varGamma(y,\overline{x})\quad\textrm{ for }y\neq x\textrm{ or }\overline{x},\label{Green-03}
\end{equation}
where $\varGamma$ is the Green function on $\mathbb{R}^{d}$.

For the case where $d=2$, the Biot-Savart law can be stated as
follows. Suppose $u\in C^{2}(D)\cap C^{1}(\overline{D})$, $\left.u\right|_{\partial D}=0$,
$u\rightarrow0$ and $|\nabla u|\rightarrow0$ at the infinity, and
$\nabla\cdot u=0$ in $D$. Let $\omega=\nabla\wedge u$ and assume
that both $\omega,\nabla\omega\in L^{1}(D)$. Then 
\begin{equation}
u(x)=\int_{D}K(x,y)\omega(y)\textrm{d}y\quad\forall x\in D,\label{B-S law}
\end{equation}
where 
\begin{equation}
K(x,y)=\frac{1}{2\pi}\left(\frac{(x-y)^{\bot}}{|x-y|^{2}}-\frac{(\bar{x}-y)^{\bot}}{|\overline{x}-y|^{2}}\right)\label{eq:qq13-1}
\end{equation}
(for $y\neq x$ or $\overline{x}$) is the Biot-Savart kernel. 

We are now in a position to establish the random vortex formulation
for two-dimensional flows past a solid wall. Therefore from now on
we assume that $d=2$. The fluid flow is described by its velocity
vector field $u(x,t)$, which satisfies the Navier-Stokes equations:
\begin{equation}
\frac{\partial}{\partial t}u+(u\cdot\nabla)u-\nu\Delta u+\nabla P-F=0\quad\textrm{ in }D\label{meq-01}
\end{equation}
and 
\begin{equation}
\nabla\cdot u=0\quad\textrm{ in }D,\label{meq-02}
\end{equation}
subject to the non-slip condition $u(x,t)=0$ for $x\in\partial D$,
where $P(x,t)$ is the pressure, $F(x,t)$ is an external force, and
$\nu>0$ is the kinematic viscosity constant of the fluid. Recall
that $u(x,t)$ is extended to the whole space so that $u^{1}(x,t)=u^{1}(\overline{x},t)$
and $u^{2}(x,t)=-u^{2}(\overline{x},t)$.

The vorticity $\omega=\nabla\wedge u$ evolves according to the transport
equation: 
\begin{equation}
\frac{\partial}{\partial t}\omega+(u\cdot\nabla)\omega-\nu\Delta\omega-G=0\quad\textrm{ in }D,\label{2d-v4}
\end{equation}
where $G=\nabla\wedge F$. Since $u(x,t)$ is subject to the non-slip
condition, according to the Biot-Savart law, 
\begin{equation}
u(x,t)=\int_{D}K(x,y)\omega(y,t)\textrm{d}y\quad\forall x\in D,\label{B-S law-1}
\end{equation}
where the singular integral kernel $K(y,x)$ is given by \eqref{eq:qq13-1}. The vorticity $\omega(x,t)$ vanishes at the infinity, while $\omega(x,t)$ in general may assume non-trivial values at the rigid wall $\partial D$, and boundary vorticity is given by 
\[
\theta(x_{1},t):=-\frac{\partial u^{1}}{\partial x_{2}}(x_{1},0,t)\quad\textrm{ for }x_{1}\in\mathbb{R}\textrm{ and }t\geq 0.
\]

Let us introduce a family of perturbations of the vorticity $\omega$
modified near the boundary using a cut-off function, so that the
modified vorticity vanishes along the boundary $\partial D$. More
precisely $\theta$ is extended to the interior of $D$ as 
\begin{equation}
\sigma_{\varepsilon}(x_{1},x_{2},t)=\theta(x_{1},t)\phi(x_{2}/\varepsilon),\label{ext-01}
\end{equation}
for each $\varepsilon>0$, where $\phi:[0,\infty)\rightarrow[0,1]$
is a proper cut-off function to be chosen later, and $\phi$ is smooth,
such that $\phi(r)=1$ for $r\in[0,1/3)$ and $\phi(r)=0$ for $r\geq2/3$,
so that $\left.\omega\right|_{\partial D}=\left.\sigma_{\varepsilon}\right|_{\partial D}$
for every $\varepsilon>0$. Let $W^{\varepsilon}=\omega-\sigma_{\varepsilon}$, then 
\begin{equation}
\left(L_{-u}-\frac{\partial}{\partial t}\right)W_{\varepsilon}+g_{\varepsilon}=0\quad\textrm{ in }D\quad\textrm{ and }\left.W_{\varepsilon}\right|_{\partial D}=0,\label{eq:qq4}
\end{equation}
where 
\begin{align}
\begin{split}
g_{\varepsilon}(x,t) & =G(x,t)+\frac{\nu}{\varepsilon^{2}}\phi''(x_{2}/\varepsilon)\theta(x_{1},t)-\frac{1}{\varepsilon}\phi'(x_{2}/\varepsilon)u^{2}(x,t)\theta(x_{1},t)\\
 & +\phi(x_{2}/\varepsilon)\left(\nu\frac{\partial^{2}\theta}{\partial x_{1}^{2}}(x_{1},t)-\frac{\partial\theta}{\partial t}(x_{1},t)\right)-\phi(x_{2}/\varepsilon)u^{1}(x,t)\frac{\partial\theta}{\partial x_{1}}(x_{1},t)\label{g-xt-01}
 \end{split}
\end{align}
for any $x=(x_{1},x_{2})$, $x_{2}\geq0$. The initial data for $W^{\varepsilon}$
is given by 
\[
W_{0}^{\varepsilon}(x)=\omega_{0}(x_{1},x_{2})-\omega_{0}(x_{1},0)\phi(x_{2}/\varepsilon)\quad\textrm{ for }x\in D
\]
where $\omega_{0}=\nabla\wedge u_{0}$ is the initial vorticity. 
\begin{thm}
\label{thm-newint} The following representation holds: 
\begin{align}
\begin{split}
u(x,t) & =\int_{D}\mathbb{E}\left[1_{D}(X_{t}^{\xi})K(x,X_{t}^{\xi})-1_{D}(X_{t}^{\overline{\xi}})K(x,X_{t}^{\overline{\xi}})\right]\omega_{0}(\xi)\textrm{d}\xi \\
 & +\int_{0}^{t}\int_{D}\mathbb{E}\left[1_{\{s>\gamma_{t}(X^{\eta})\}}1_{D}(X_{t}^{\xi})K(x,X_{t}^{\xi})G(X_{s}^{\xi},s)\right]\textrm{d}\xi\textrm{d}s \\
 & +\lim_{\varepsilon\downarrow0}\int_{0}^{t}\int_{D}\mathbb{E}\left[1_{\{s>\gamma_{t}(X^{\eta})\}}1_{D}(X_{t}^{\xi})K(x,X_{t}^{\xi},x)\tilde{g}_{\varepsilon}(X_{s}^{\xi},s)\right]\textrm{d}\xi\textrm{d}s\label{u-aa3}
\end{split}
\end{align}
for every $x=(x_{1},x_{2})\in D$, $t>0$, and 
\[
u^{1}(x,t)=u^{1}(\overline{x},t),\quad u^{2}(x,t)=-u^{2}(\overline{x},t)
\]
for $x=(x_{1},x_{2})$ with $x_{2}>0$, $t>0$, where 
\[
\tilde{g}_{\varepsilon}(x,t)=\frac{\nu}{\varepsilon^{2}}\chi(x_{2}/\varepsilon)\theta(x_{1},t)\quad\textrm{ for }x=(x_{1},x_{2})
\]
$\gamma_{t}(\psi)$ denotes $\sup\left\{ s\in(0,t):\psi(s)\in D\right\} $
and 
\begin{equation}
\chi(r)=\begin{cases}
162\left(2r-1\right) & \textrm{ for }r\in[1/3,2/3],\\
0 & \textrm{ otherwise. }
\end{cases}\label{phi-2d-1}
\end{equation}
\end{thm}

\begin{proof}
For every $\varepsilon>0$, the deformed vorticity $W_{\varepsilon}$
satisfies the parabolic equation \eqref{eq:qq4}: 
\begin{equation}
\left(L_{-u}-\frac{\partial}{\partial t}\right)W_{\varepsilon}+g_{\varepsilon}=0\quad\textrm{ in }D\label{eq:qq4-a}
\end{equation}
subject to the Dirichlet boundary condition  $\left.W_{\varepsilon}\right|_{\partial D}=0.$
According to Theorem \ref{thm23-5} 
\begin{align}
\begin{split}
W_{\varepsilon}(\xi,t) & =\int_{D}\left(p_{u}(0,\eta,t,\xi)-p_{u}(0,\bar{\eta},t,\xi)\right)W_{\varepsilon}(\eta,0)\textrm{d}\eta \\
 & +\int_{0}^{t}\int_{D}\mathbb{E}\left[\left.1_{\{s>\gamma_{t}(X^{\eta})\}}g_{\varepsilon}(X_{s}^{\eta},s)\right|X_{t}^{\eta}=\xi\right]p_{u}(0,\eta;t,\xi)\textrm{d}\eta\textrm{d}s
 \end{split}
 \label{W-aa2-1}
\end{align}
where $\gamma_{t}(X^{\eta}) = \sup\{s\in(0,t):X_{s}^{\eta}\in D\}$, so by using the Biot-Savart law we obtain that 
\begin{align*}
u(x,t) & =\int_{D}K(x,y)\sigma_{\varepsilon}(y,t)\textrm{d}y\\
 & +\int_{D}\mathbb{E}\left[1_{D}(X_{t}^{\xi})K(x,X_{t}^{\xi})-K(x,X_{t}^{\overline{\xi}})1_{D}(X_{t}^{\overline{\xi}})\right]W_{\varepsilon}(\xi,0)\textrm{d}\xi\\
 & +\int_{0}^{t}\int_{D}\mathbb{E}\left[1_{\{s>\gamma_{t}(X^{\eta})\}}1_{D}(X_{t}^{\eta})K(x,X_{t}^{\eta})\wedge g_{\varepsilon}(X_{s}^{\eta},s)\right]\textrm{d}\eta\textrm{d}s
\end{align*}
for every $x\in D$, $t>0$ and for every $\varepsilon>0$. The result
follows by setting $\varepsilon\downarrow0$. 
\end{proof}

\section{Numerical simulations of wall-bounded flows}
\label{Numerical simulations of wall-bounded flows}

Let us make several comments on the simulations of fluid flows passing a wall based on the functional integral representations we have established in previous sections. Then we describe several numerical schemes. Our numerical experiment is carried out with a simplified scheme to demonstrate the usefulness of the methods advocated in the present
work.

Let us recall the important notations. $D$ represents the upper semi-plane where $x_{2}>0$. The reflection $x\mapsto\overline{x}$ with $\overline{x}=(x_{1},-x_{2})$ for $x=(x_{1},x_{2})$, and $u(\overline{x},t)=(u^{1}(x,t),-u^{2}(x,t))$ for every $x\in\mathbb{R}^{2}$.

For a two-dimensional vector $a=(a_{1},a_{2})$ and a scalar $c$, $a\wedge c=(a_{2}c,-a_{1}c)$.

Let us reserve two parameters $\varepsilon>0$ and $\delta>0$ which will be used to realise a practical approximation based on the functional integral representation in Theorem \ref{thm-newint}. Define
\[
K_{D,\delta}(y,x)=\frac{1}{2\pi}1_{D}(y)\left(1-e^{-|y-x|^{2}/\delta^{2}}\right)\left(\frac{(y-x)^{\bot}}{|y-x|^{2}}-\frac{(y-\overline{x})^{\bot}}{|y-\overline{x}|^{2}}\right),
\]
the regularised Biot-Savart kernel for the half-plane. The regularisation
is required in order to update the boundary stress. We may introduce
the boundary regularised function 
\begin{equation}
\chi(r)=\begin{cases}
162(2r-1) & \textrm{ for }r\in[1/3,2/3],\\
0 & \textrm{ otherwise. }
\end{cases}\label{phi-2d-1-1}
\end{equation}

Then, according to Theorem \ref{thm-newint}, the velocity will be computed by the following formula: 
\begin{align*}
u(x,t) & =\int_{D}\mathbb{E}\left[K_{D,\delta}(x,X_{t}^{\xi})-K_{D,\delta}(x,X_{t}^{\overline{\xi}})\right]\omega_{0}(\xi)\textrm{d}\xi\\
 & +\int_{0}^{t}\int_{D}\mathbb{E}\left[1_{\{s>\gamma_{t}(X^{\xi})\}}K_{D,\delta}(x,X_{t}^{\xi})G(X_{s}^{\xi},s)\right]\textrm{d}\xi\textrm{d}s\\
 & +\frac{\nu}{\varepsilon^{2}}\int_{0}^{t}\int_{D}\mathbb{E}\left[\chi\left(\frac{X_{2;s}^{\xi}}{\varepsilon}\right)1_{\{s>\gamma_{t}(X^{\xi})\}}K_{D,\delta}(x,X_{t}^{\xi})\theta(X_{1;s}^{\xi},s)\right]\textrm{d}\xi\textrm{d}s,
\end{align*}
for every $x\in D$ and $t>0$, where $\gamma_{t}(\psi) = \sup\left\{ s\in(0,t):\psi(s)\in D\right\}$. In the above formula $X_{1;s}^{\xi}$ and $X_{2;s}^{\xi}$ denote the first and second coordinate of $X_{s}^{\xi}$ respectively. In what follows, we will omit the coordinate index by denoting $\chi(x) = \chi(x_2)$ and $\theta(x, t) = \theta(x_1,t)$.

Let us describe the numerical schemes based on the last model. We divide the schemes into two steps. During the first step we set up the discretisation procedure for dealing with (finite-dimensional) integrals in time and in the space variables. The discretisation of this type appears in any numerical methods, but a bit care is needed due to the appearance of boundary layer phenomena for wall-bounded flows. According to Prandtl \citep{Prandtl1904}, there is a thin layer near the boundary within which the main stream velocity decreases to zero sharply, hence there is substantial stress at the wall, which in turn generates significant boundary vorticity. Turbulence may be generated near the solid wall if the Reynolds number is large. The boundary layer thickness $\delta>0$ of a fluid flow with small viscosity is given by $\frac{\delta}{L}\sim\sqrt{\frac{1}{\textrm{Re}}}$, where $\textrm{Re}$ is the Reynolds number $\textrm{Re}=\frac{VL}{\nu}$,
$V$ and $L$ are the typical velocity magnitude and typical length.

In numerical simulations, we can always first make a reduction, to make the fluid dynamics equations dimensionless. That is, the typical velocity and length may be fixed to the good size for printing the results. While we still prefer to use the Reynolds number $\textrm{Re}$,
so that $\nu\sim\frac{1}{\textrm{Re}}$. Suppose $\nu$ is small, then near the boundary, the mesh $h_{2}$ for the $x_{2}$-coordinate has to be far smaller than the boundary layer thickness $\delta\sim\sqrt{\frac{1}{\textrm{Re}}}$, which leads to the first constraint: 
\begin{equation}
h_{2}\ll L\sqrt{\frac{1}{\textrm{Re}}}.\label{h-ratio01}
\end{equation}
The mesh size $h_{1}$ for $x_{1}$-coordinate should be comparable to $h_{2}$ but not need to be much larger than $h_{2}$. For the region outside the boundary layer, we can use reasonable size of the mesh $h_{0}$, and in general we choose $h_{0}\geq h_{1}\geq h_{2}$. We
may choose integers $N_{1}$, $N_{2}$ and $N_{0}$ according to the following constraints: 
\begin{equation}
N_{1}h_{1}\sim L,\ \,L\geq\delta\quad\textrm{ and }N_{0}h_{0}\sim L.\label{meshes}
\end{equation}
However in order to capture the flow in the boundary layer, $N_{2}$ has to be chosen such that $h_{2}N_{2}>\delta$.

We introduce regular rectangular lattices for the boundary and the outer layer, and calculate the values at the lattice points 
\begin{align*}
x_{b}^{i_{1},i_{2}}=
(i_{1}h_{1},i_{2}h_{2}) & \quad\textrm{for }|i_{1}|\leq N_{1}\textrm{ and }|i_{2}|\leq N_{2};\\
x_{o}^{i_{1},i_{2}}=(i_{1}h_{0},i_{2}h_{0}) & \quad\textrm{for }|i_{1}|\leq N_{0}\textrm{ and }|i_{2}|\leq N_{0},
\label{xii-01}
\end{align*}
which in the following will be written as $x^{i_{1},i_{2}}$ to denote both the boundary and outer lattices.

Notice that the lattices are symmetric with respect to the second coordinate as we need the corresponding values for $X_{t}^{\overline{\xi}}$ in the first integral in representation \eqref{u-aa3}.
The total number of the lattice points needed for the lattices equals 
\[
N=(2 N_{1}+1)(2 N_{2}+1)+(2 N_{0}+1)^{2},
\]
which largely determines the computational cost. 

Thus far we have described the discretisation of space variables. For time $t$, we can use a unified scheme, at each step we use time duration $h>0$ to be chosen properly, with $t_{k}=kh$ for $k=0,1,2,\ldots.$ The Taylor diffusions \eqref{b-SDE} are initialised at the lattice points in the numerical scheme at time $t_{0}=0$. Then we propose a simple scheme to compute them at further times $t_{k}=kh$ for $k=1,2,\ldots$ --- it is given by the discretisation of the SDE which can be set up as follows: 
\[
X_{0}^{i_{1},i_{2}}=x^{i_{1},i_{2}},
\]
\[
X_{t_{k+1}}^{i_{1},i_{2}}=X_{t_{k}}^{i_{1},i_{2}}+hu(X_{t_{k}}^{i_{1},i_{2}},t_{k})+\sqrt{2\nu}(B_{t_{k+1}}-B_{t_{k}}),
\]
for $k=0,1,\ldots$, where $u(x,t)$ are updated according to 
\begin{align}
\begin{split}
u(x,t_{k+1}) & =\sum_{\substack{i_{1},i_{2}\\ i_{2}>0}}A_{i_{1},i_{2}}\mathbb{E}\left[K_{D,\delta}\left(x,X_{t_{k+1}}^{i_{1},i_{2}}\right)-K_{D,\delta}\left(x,X_{t_{k+1}}^{i_{1},-i_{2}}\right)\right]\omega_{i_{1},i_{2}} \\
 & +\sum_{\substack{i_{1},i_{2}\\ i_{2}>0}}A_{i_{1},i_{2}}\sum_{l=0}^{k}h\mathbb{E}\left[1_{\left\{t_l>\gamma_{t_k}\left(X^{i_{1},i_{2}}\right)\right\}}K_{D,\delta}\left(x,X_{t_{k+1}}^{i_{1},i_{2}}\right)G\left(X_{t_{l}}^{i_{1},i_{2}},t_{l}\right)\right] \\
 & +\frac{\nu}{\varepsilon^{2}}\sum_{\substack{i_{1},i_{2}\\ i_{2}>0}}A_{i_{1},i_{2}}\sum_{l=0}^{k}h\mathbb{E}\Biggr[1_{\left\{t_l>\gamma_{t_k}\left(X^{i_{1},i_{2}}\right)\right\}}\chi\left(\frac{X_{t_{l}}^{i_{1},i_{2}}}{\varepsilon}\right)  \\ & \qquad\qquad \times K_{D,\delta}\left(x,X_{t_{k+1}}^{i_{1},i_{2}}\right)\theta\left(X_{t_{l}}^{i_{1},i_{2}},t_{l}\right)\Biggr],
 \end{split}
 \label{g-s-01}
\end{align}
for $x=(x_{1},x_{2})$ with $x_{2}>0$, and 
\[
(u^{1}(x,t),u^{2}(x,t))=(u^{1}(\overline{x},t),-u^{2}(\bar{x},t))\quad\textrm{ if }x_{2}<0.
\]
In the above formula, the summation over indices $i_1, i_2$ is taken over both the boundary and outer lattice, and therefore 
\begin{equation}
A_{i_{1},i_{2}}=\begin{cases}
h_{1}h_{2} & \textrm{ for the sum over the boundary lattice},\\
h_{0}h_{0} & \textrm{ for the sum over the outer lattice},
\end{cases}\label{Aii-01}
\end{equation}
and
\begin{equation}
\omega_{i_{1},i_{2}}=\omega_{0}(x^{i_{1},i_{2}}),\label{ome-d}
\end{equation}
where $\omega_{0}(x) = \omega(x, 0)$.

The representation \eqref{g-s-01} depends on the boundary vorticity 
\[
\theta(x_{1},t)=-\left.\frac{\partial u^{1}(x,t)}{\partial x_{2}}\right|_{x_{2}=0+}
\]
which can be, for instance, approximated by a finite difference. However, as the values of the velocity can be unstable within the boundary layer, in practice it is better to update the boundary stress by differentiating the iteration \eqref{g-s-01}, so that the boundary stress is updated according to the following procedure:
\begin{align}
\begin{split}
\theta(x_{1},t_{k+1}) & =\sum_{\substack{i_{1},i_{2}\\ i_{2}>0}}A_{i_{1},i_{2}}\mathbb{E}\left[\varTheta_{D,\delta}\left(x,X_{t_{k+1}}^{i_{1},i_{2}}\right)-\varTheta_{D,\delta}\left(x,X_{t_{k+1}}^{i_{1},-i_{2}}\right)\right]\omega_{i_{1},i_{2}} \\
 & +\sum_{\substack{i_{1},i_{2}\\ i_{2}>0}}A_{i_{1},i_{2}}\sum_{l=0}^{k}h\mathbb{E}\left[1_{\left\{t_l>\gamma_{t_k}\left(X^{i_{1},i_{2}}\right)\right\}}\varTheta_{D,\delta}\left(x,X_{t_{k+1}}^{i_{1},i_{2}}\right)G\left(X_{t_{l}}^{i_{1},i_{2}},t_{l}\right)\right] \\
 & +\frac{\nu}{\varepsilon^{2}}\sum_{\substack{i_{1},i_{2}\\ i_{2}>0}}A_{i_{1},i_{2}}\sum_{l=0}^{k}h\mathbb{E}\Biggr[1_{\left\{t_l>\gamma_{t_k}\left(X^{i_{1},i_{2}}\right)\right\}}\chi\left(\frac{X_{t_{l}}^{i_{1},i_{2}}}{\varepsilon}\right)  \\ & \qquad\qquad  \times\varTheta_{D,\delta}\left(x,X_{t_{k+1}}^{i_{1},i_{2}}\right)\theta\left(X_{t_{l}}^{i_{1},i_{2}},t_{l}\right)\Biggr],
 \end{split}
 \label{g-s-01-1}
\end{align}
for $k=0,1,2,\ldots$, where 
\begin{equation*}
\varTheta_{D,\delta}(y,x_{1}) =-\left.\frac{\partial}{\partial x_{2}}\right|_{x_{2}=0+}K_{D,\delta}^{1}(y,x).
\end{equation*}
We compute the derivative explicitly as follows
\begin{align*}
\varTheta_{D,\delta}(y,x_{1}) =1_{\{y_{2}>0\}}(y)\frac{1}{\pi}\left(\left(1-\textrm{e}^{-\frac{|y_{1}-x_{1}|^{2}+|y_{2}|^{2}}{\delta^2}}\right)\frac{|y_{1}-x_{1}|^{2}-|y_{2}|^{2}}{\left(|y_{1}-x_{1}|^{2}+|y_{2}|^{2}\right)^{2}}\right.& \\
\left.\frac{1}{\delta^2} \textrm{e}^{-\frac{|y_{1}-x_{1}|^{2}+|y_{2}|^{2}}{\delta^2}} \frac{|y_2|^2}{|y_1-x_1|^2+|y_2|^2} \right). &
\end{align*}

In the second step, one has to handle the mathematical expectations in representations in \eqref{g-s-01} and \eqref{g-s-01-1}, which is the core of the Monte-Carlo schemes. We propose two slightly different approaches which  lead to the following schemes.

\subsection{Numerical scheme 1}\label{NS1section}

We are now in a position to formulate our first Monte-Carlo scheme in which we drop the expectation by running independent (two-dimensional) Brownian motions. More precisely, we run the following stochastic differential equations: 
\[
X_{0}^{i_{1},i_{2}}=x^{i_{1},i_{2}},
\]
\[
X_{t_{k+1}}^{i_{1},i_{2}}=X_{t_{k}}^{i_{1},i_{2}}+hu(X_{t_{k}}^{i_{1},i_{2}},t_{k})+\sqrt{2\nu}\sqrt{t_{k+1}-t_{k}}\Phi_{k}
\]
where 
\[
\Phi_{k}=\left(\Phi_{k}^{1},\Phi_{k}^{2}\right)\equiv\frac{B_{t_{k+1}}-B_{t_{k}}}{\sqrt{t_{k+1}-t_{k}}},
\]
with $\Phi_{k}^{1}$, $\Phi_{l}^{2}$ (for $k,l=0,1,\ldots$) being
independent standard Gaussian random variables. $u(x,t)$ is updated
according to the following: 
\begin{align*}
u(x,t_{k+1}) & =\sum_{\substack{i_{1},i_{2}\\ i_{2}>0}}A_{i_{1},i_{2}}\left(K_{D,\delta}\left(x,X_{t_{k+1}}^{i_{1},i_{2}}\right)-K_{D,\delta}\left(x,X_{t_{k+1}}^{i_{1},-i_{2}}\right)\right)\omega_{i_{1},i_{2}}\nonumber \\
 & +\sum_{\substack{i_{1},i_{2}\\ i_{2}>0}}A_{i_{1},i_{2}}\sum_{l=0}^{k}h1_{\left\{t_l>\gamma_{t_k}\left(X^{i_{1},i_{2}}\right)\right\}}K_{D,\delta}\left(x,X_{t_{k+1}}^{i_{1},i_{2}}\right)G\left(X_{t_{l}}^{i_{1},i_{2}},t_{l}\right)\nonumber \\
 & +\frac{\nu}{\varepsilon^{2}}\sum_{\substack{i_{1},i_{2}\\ i_{2}>0}}A_{i_{1},i_{2}}\sum_{l=0}^{k}h1_{\left\{t_l>\gamma_{t_k}\left(X^{i_{1},i_{2}}\right)\right\}}\chi\left(\frac{X_{t_{l}}^{i_{1},i_{2}}}{\varepsilon}\right) \nonumber \\ & \qquad\qquad\times K_{D,\delta}\left(x,X_{t_{k+1}}^{i_{1},i_{2}}\right)\theta\left(X_{t_{l}}^{i_{1},i_{2}},t_{l}\right),
\end{align*}
and
\begin{align*}
\theta(x_{1},t_{k+1}) & =\sum_{\substack{i_{1},i_{2}\\ i_{2}>0}}A_{i_{1},i_{2}}\left(\varTheta_{D,\delta}\left(x,X_{t_{k+1}}^{i_{1},i_{2}}\right)-\varTheta_{D,\delta}\left(x,X_{t_{k+1}}^{i_{1},-i_{2}}\right)\right)\omega_{i_{1},i_{2}}\nonumber \\
 & +\sum_{\substack{i_{1},i_{2}\\ i_{2}>0}}A_{i_{1},i_{2}}\sum_{l=0}^{k}h 1_{\left\{t_l>\gamma_{t_k}\left(X^{i_{1},i_{2}}\right)\right\}}\varTheta_{D,\delta}\left(x,X_{t_{k+1}}^{i_{1},i_{2}}\right)G\left(X_{t_{l}}^{i_{1},i_{2}},t_{l}\right)\nonumber \\
 & +\frac{\nu}{\varepsilon^{2}}\sum_{\substack{i_{1},i_{2}\\ i_{2}>0}}A_{i_{1},i_{2}}\sum_{l=0}^{k}h1_{\left\{t_l>\gamma_{t_k}\left(X^{i_{1},i_{2}}\right)\right\}}\chi\left(\frac{X_{t_{l}}^{i_{1},i_{2}}}{\varepsilon}\right) \nonumber \\ & \qquad\qquad \times\varTheta_{D,\delta}\left(x,X_{t_{k+1}}^{i_{1},i_{2}}\right)\theta\left(X_{t_{l}}^{i_{1},i_{2}},t_{l}\right),
\end{align*}
for $x=(x_{1},x_{2})$ with $x_{2}>0$, and 
\[
(u^{1}(x,t),u^{2}(x,t))=(u^{1}(\overline{x},t),-u^{2}(\bar{x},t))\quad\textrm{ when }x_{2}<0.
\]

\subsection{Numerical scheme 2}\label{NS2section}

In this scheme we appeal to the strong law of large numbers, that is we replace the expectations by averages. Therefore we run the
following stochastic differential equations: 
\begin{equation*}
X_{0}^{m;i_{1},i_{2}}=x^{i_{1},i_{2}},
\end{equation*}
\begin{equation*}
X_{t_{k+1}}^{m;i_{1},i_{2}}=X_{t_{k}}^{m;i_{1},i_{2}}+hu(X_{t_{k}}^{m;i_{1},i_{2}},t_{k})+\sqrt{2\nu}\sqrt{t_{k+1}-t_{k}}\Phi_{k}^{m}
\end{equation*}
where 
\[
\Phi_{k}^{m}=\left(\Phi_{k}^{m,1},\Phi_{k}^{m,2}\right)\equiv\frac{B_{t_{k+1}}^{m}-B_{t_{k}}^{m}}{\sqrt{t_{k+1}-t_{k}}},
\]
and $\Phi_{k}^{m,1}$, $\Phi_{l}^{n,2}$ (for $m,n=1,2,\ldots$ and
$k,l=0,1,\ldots$) are independent standard Gaussian random variables.
Here $u(x,t)$ is updated according to the following: 
\begin{align*}
u(x,t_{k+1}) & =\sum_{\substack{i_{1},i_{2}\\ i_{2}>0}}A_{i_{1},i_{2}}\frac{1}{N}\sum_{m=1}^{N}\left(K_{D,\delta}\left(x,X_{t_{k+1}}^{m;i_{1},i_{2}}\right)-K_{D,\delta}\left(x,X_{t_{k+1}}^{m;i_{1},-i_{2}}\right)\right)\omega_{i_{1},i_{2}}\nonumber \\
 & +\sum_{\substack{i_{1},i_{2}\\ i_{2}>0}}A_{i_{1},i_{2}}\sum_{l=0}^{k}h\frac{1}{N}\sum_{m=1}^{N}1_{\left\{t_l>\gamma_{t_k}\left(X^{m;i_{1},i_{2}}\right)\right\}}K_{D,\delta}\left(x,X_{t_{k+1}}^{m;i_{1},i_{2}}\right)G\left(X_{t_{l}}^{m;i_{1},i_{2}},t_{l}\right)\nonumber \\
 & +\frac{\nu}{\varepsilon^{2}}\sum_{\substack{i_{1},i_{2}\\ i_{2}>0}}A_{i_{1},i_{2}}\sum_{l=0}^{k}h\frac{1}{N}\sum_{m=1}^{N}1_{\left\{t_l>\gamma_{t_k}\left(X^{m;i_{1},i_{2}}\right)\right\}}\chi\left(\frac{X_{t_{l}}^{m;i_{1},i_{2}}}{\varepsilon}\right) \nonumber \\ & \qquad\qquad \times K_{D,\delta}\left(x,X_{t_{k+1}}^{m;i_{1},i_{2}}\right)\theta\left(X_{t_{l}}^{m;i_{1},i_{2}},t_{l}\right),
\end{align*}
and
\begin{align*}
\theta(x_{1},t_{k+1}) & =\sum_{\substack{i_{1},i_{2}\\ i_{2}>0}}A_{i_{1},i_{2}}\frac{1}{N}\sum_{m=1}^{N}\left(\varTheta_{D,\delta}\left(x,X_{t_{k+1}}^{m;i_{1},i_{2}}\right)-\varTheta_{D,\delta}\left(x,X_{t_{k+1}}^{m;i_{1},-i_{2}}\right)\right)\omega_{i_{1},i_{2}}\nonumber \\
 & +\sum_{\substack{i_{1},i_{2}\\ i_{2}>0}}A_{i_{1},i_{2}}\sum_{l=0}^{k}h\frac{1}{N}\sum_{m=1}^{N}1_{\left\{t_l>\gamma_{t_k}\left(X^{m;i_{1},i_{2}}\right)\right\}}\varTheta_{D,\delta}\left(x,X_{t_{k+1}}^{m;i_{1},i_{2}}\right)G\left(X_{t_{l}}^{m;i_{1},i_{2}},t_{l}\right)\nonumber \\
 & +\frac{\nu}{\varepsilon^{2}}\sum_{\substack{i_{1},i_{2}\\ i_{2}>0}}A_{i_{1},i_{2}}\sum_{l=0}^{k}h\frac{1}{N}\sum_{m=1}^{N}1_{\left\{t_l>\gamma_{t_k}\left(X^{m;i_{1},i_{2}}\right)\right\}}\chi\left(\frac{X_{t_{l}}^{m;i_{1},i_{2}}}{\varepsilon}\right) \nonumber \\ & \qquad\qquad \times\varTheta_{D,\delta}\left(x,X_{t_{k+1}}^{m;i_{1},i_{2}}\right)\theta\left(X_{t_{l}}^{m;i_{1},i_{2}},t_{l}\right),
\end{align*}
for $x=(x_{1},x_{2})$ with $x_{2}>0$, and 
\[
(u^{1}(x,t),u^{2}(x,t))=(u^{1}(\overline{x},t),-u^{2}(\bar{x},t))\quad\textrm{ when }x_{2}<0.
\]

\subsection{Numerical experiment}

Recall the Reynolds number is given by $\textrm{Re}=\frac{U_{0}L}{\nu}$
with the main stream velocity $U_{0}$ determined by the initial velocity
in our experiments, and $L$ is the typical size. In our experiment, we set 
\begin{equation}
\nu=0.01\label{viscosity-c1}
\end{equation}
for the viscosity constant, and the typical length scale 
\begin{equation}
L=6.\label{typ-s-01}
\end{equation}
We also set the initial main stream velocity 
\begin{equation}
U_{0}=1,\label{k-value}
\end{equation}
which altogether imply that  $\textrm{Re} = 600$ in our setup.

In the experiment, we choose the initial velocity to be $u_0(x_1, x_2) = U_0 1_{x_2}$, i.e. the uniform flow that is set to satisfy the no-slip condition at the wall. Some non-trivial motion is then incited by the external force term $G(x, t) \equiv G_0$ which we choose as $G_0 = -0.2$. The kernel regularisation parameter $\sigma = 0.01$, and the cutoff parameter $\varepsilon=0.05$.

As described above, we choose an outer lattice to be a rectangular lattice in the domain $\{-6 \leq x_1 \leq 6, 0 \leq x_2 \leq 6\}$. The mesh size for the lattice is $h_0=0.15$ and the number $N_0=40$. For the lattice in the boundary layer, we the mesh size has to satisfy 
\[
h_{2}\ll\delta=L\sqrt{\frac{1}{\textrm{Re}}}\sim 0.2449.
\]
In the simulation the chosen parameters are  $h_{1}=0.1$, $h_{2}=0.00125$ with numbers $N_{1}=60$, $N_{2}=80$. This choice gives $26,042$ points in total. 

We implement the numerical scheme 1 from Subsection \ref{NS1section}. The time mesh is taken as $h=0.01$ and the iteration is conducted for $300$ steps. The simulation results for the boundary flow are presented in Figure \ref{ExpFigBLayer}. In the figures, the streamlines are coloured by the magnitude of the velocity and the background colour gives the value of the vorticity. For the outer outer flow, we present the results in Figure \ref{ExpFigOLayer}. 

\begin{figure}
    \centering
    \subfloat[\centering $t=0.01$]{{\includegraphics[width=.45\linewidth]{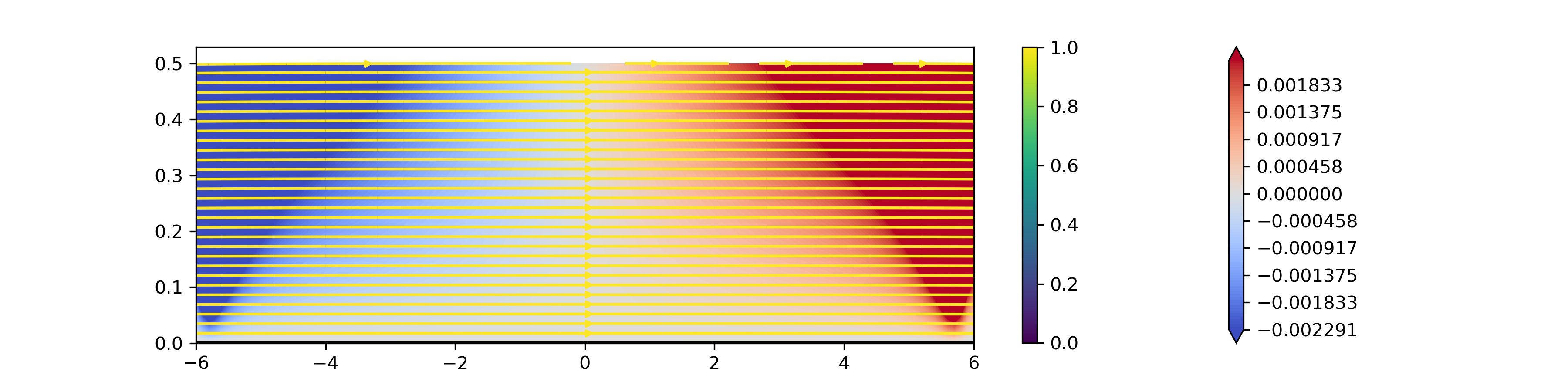} }}
    \subfloat[\centering $t=0.25$]{{\includegraphics[width=.45\linewidth]{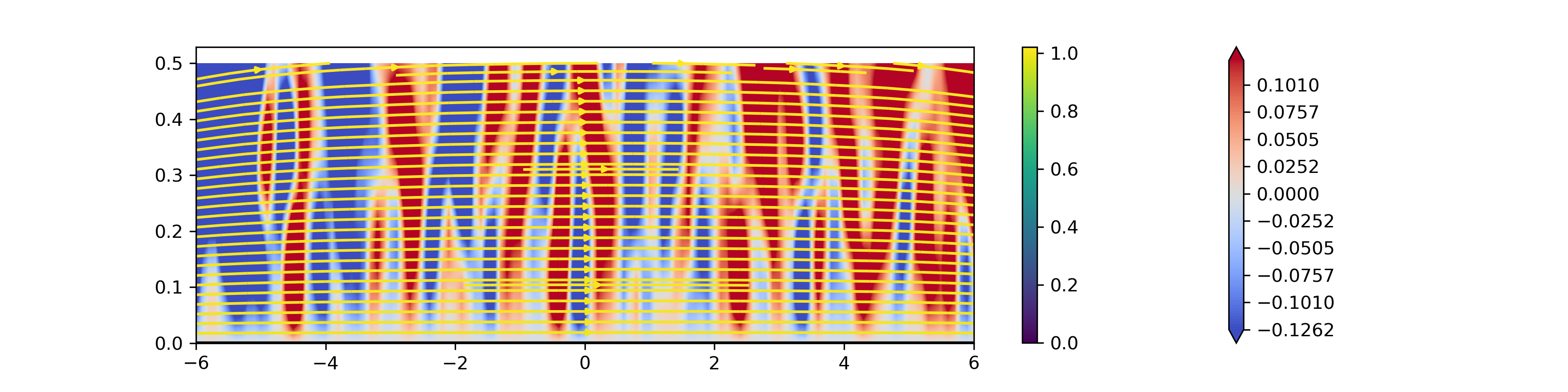}}}
    \qquad
    \subfloat[\centering $t=0.5$]{{\includegraphics[width=.45\linewidth]{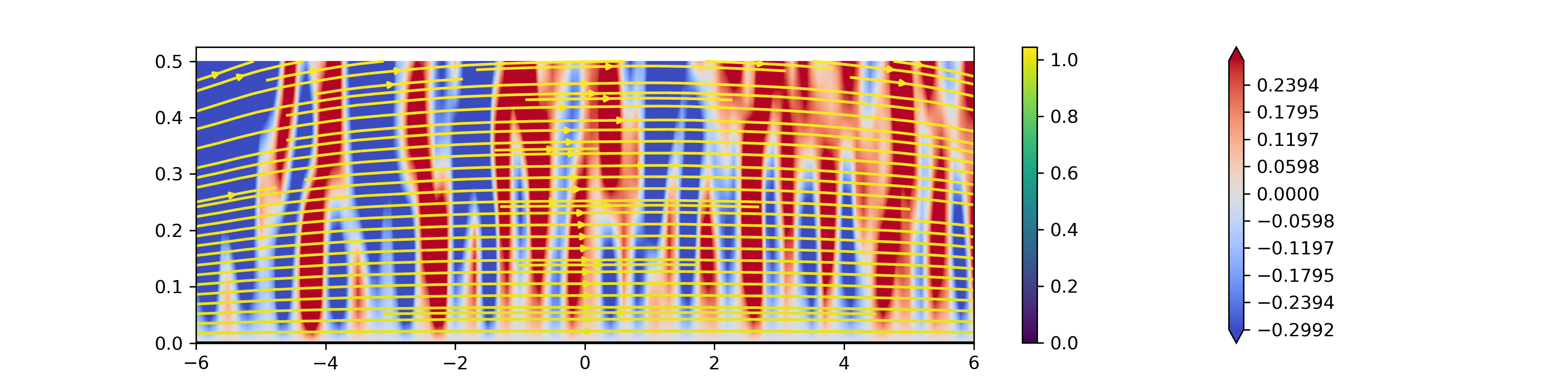} }}
    \subfloat[\centering $t=1.0$]{{\includegraphics[width=.45\linewidth]{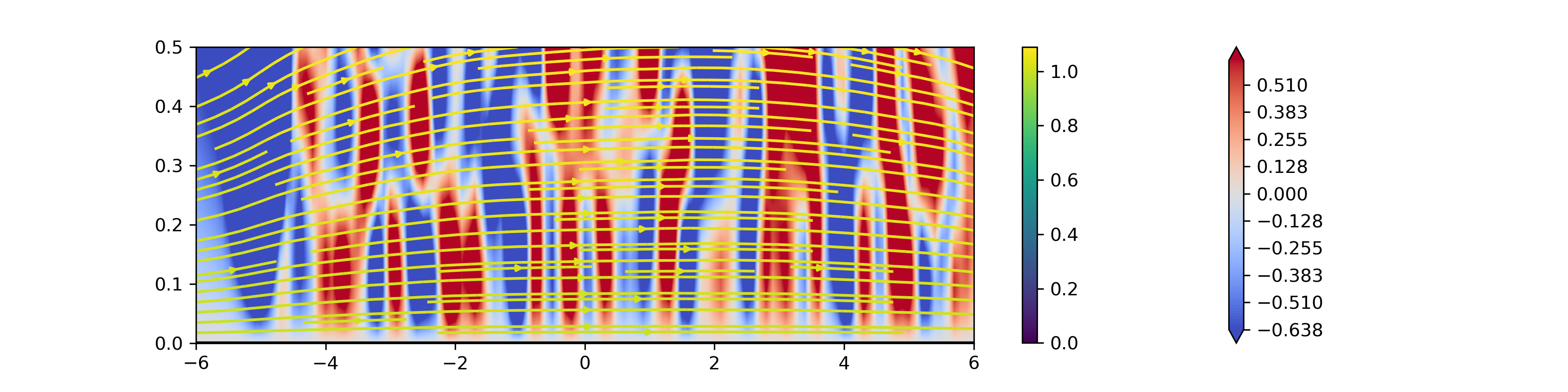}}}
    \qquad
    \subfloat[\centering $t=1.5$]{{\includegraphics[width=.45\linewidth]{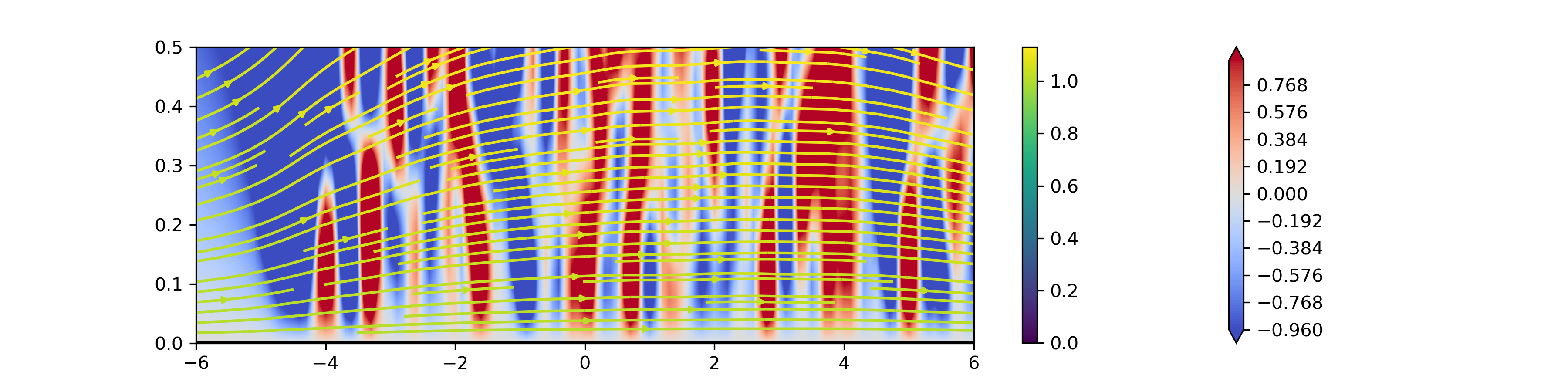} }}
    \subfloat[\centering $t=2.0$]{{\includegraphics[width=.45\linewidth]{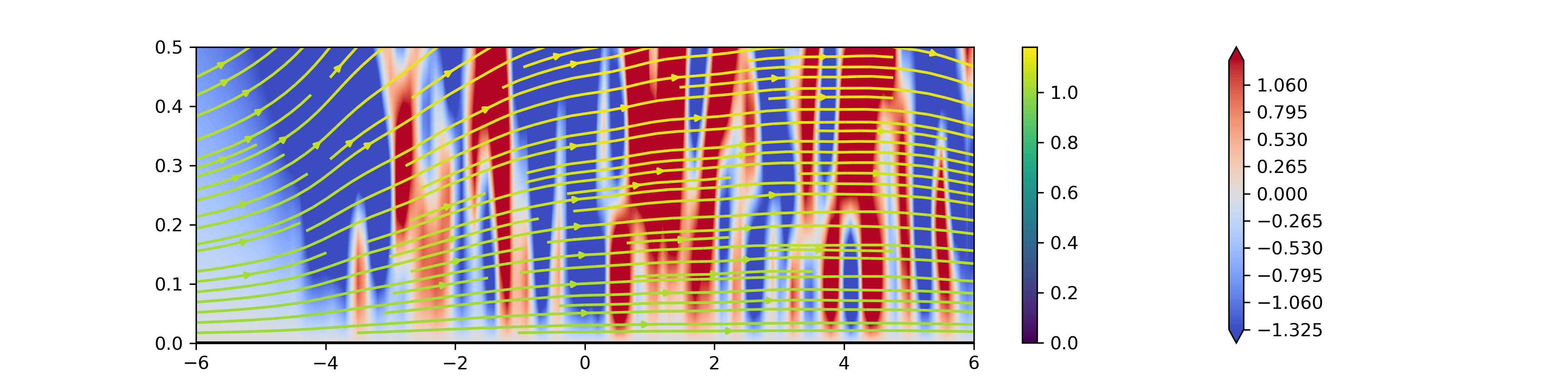}}}
    \qquad
    \subfloat[\centering $t=2.5$]{{\includegraphics[width=.45\linewidth]{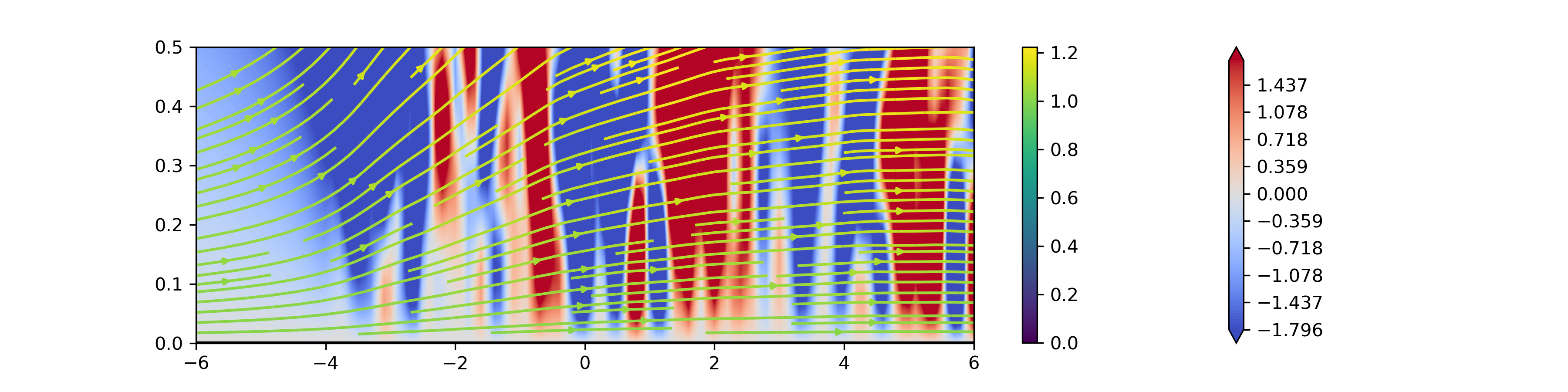} }}
    \subfloat[\centering $t=3.0$]{{\includegraphics[width=.45\linewidth]{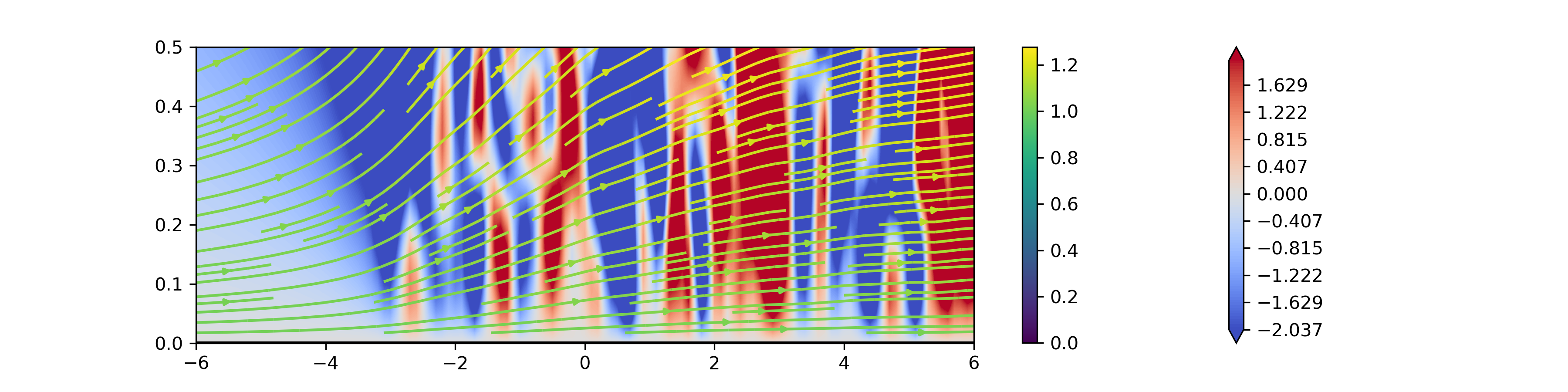}}}
    \caption{The boundary layer flow at different times $t$.}
    \label{ExpFigBLayer}
\end{figure}

\begin{figure}
    \centering
    \subfloat[\centering $t=0.01$]{{\includegraphics[width=.45\linewidth]{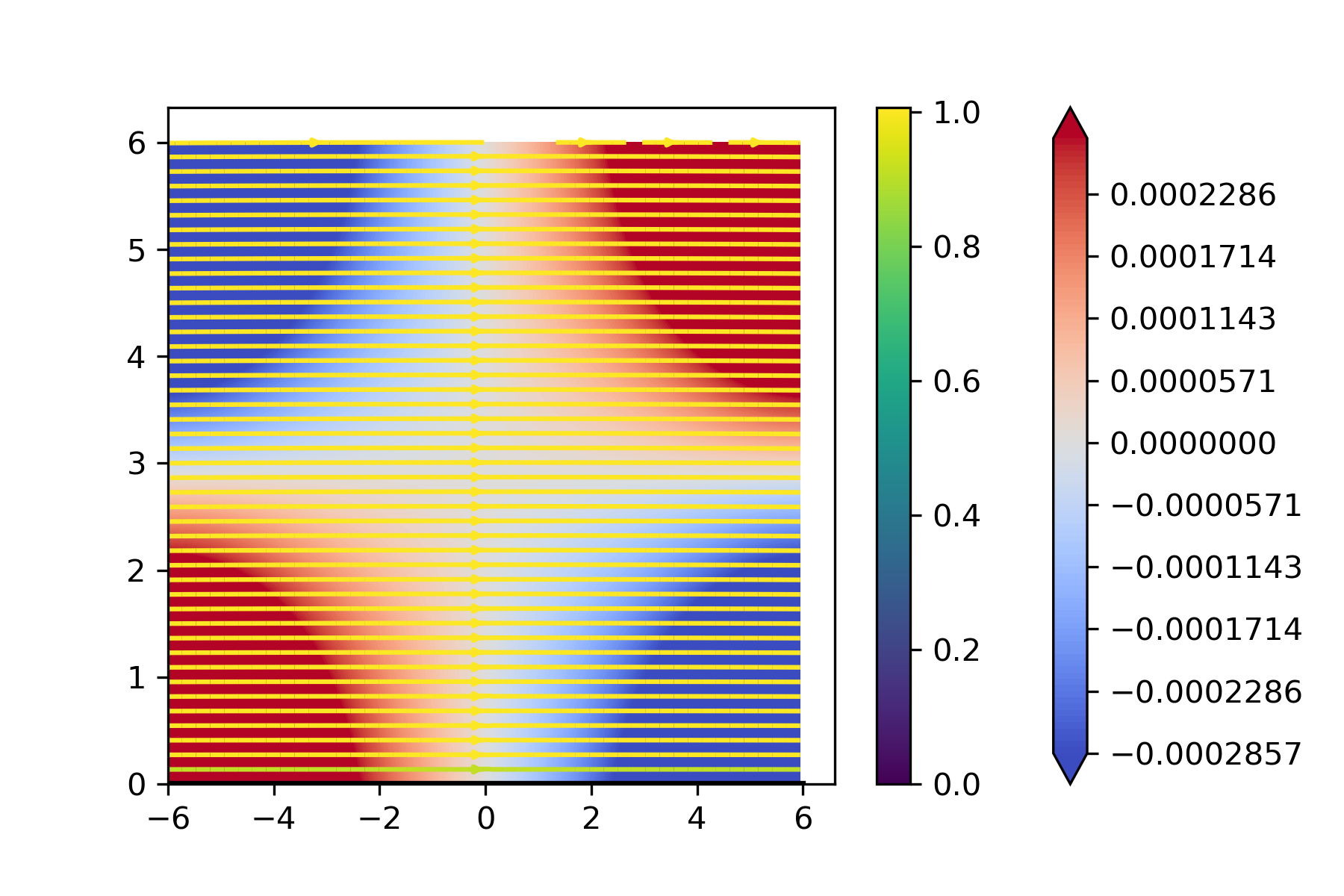} }}
    \subfloat[\centering $t=1.0$]{{\includegraphics[width=.45\linewidth]{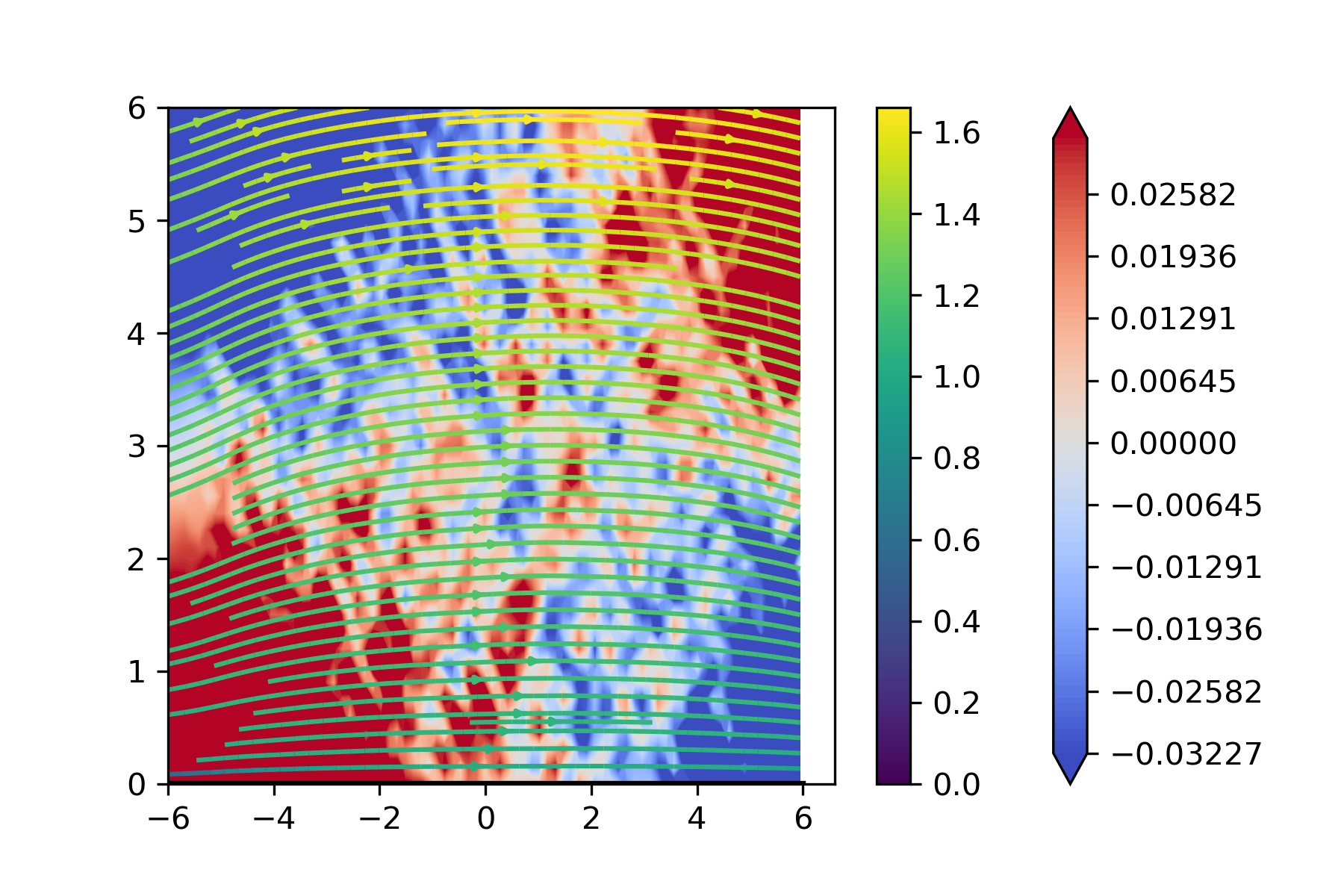}}}
    \qquad
    \subfloat[\centering $t=2.0$]{{\includegraphics[width=.45\linewidth]{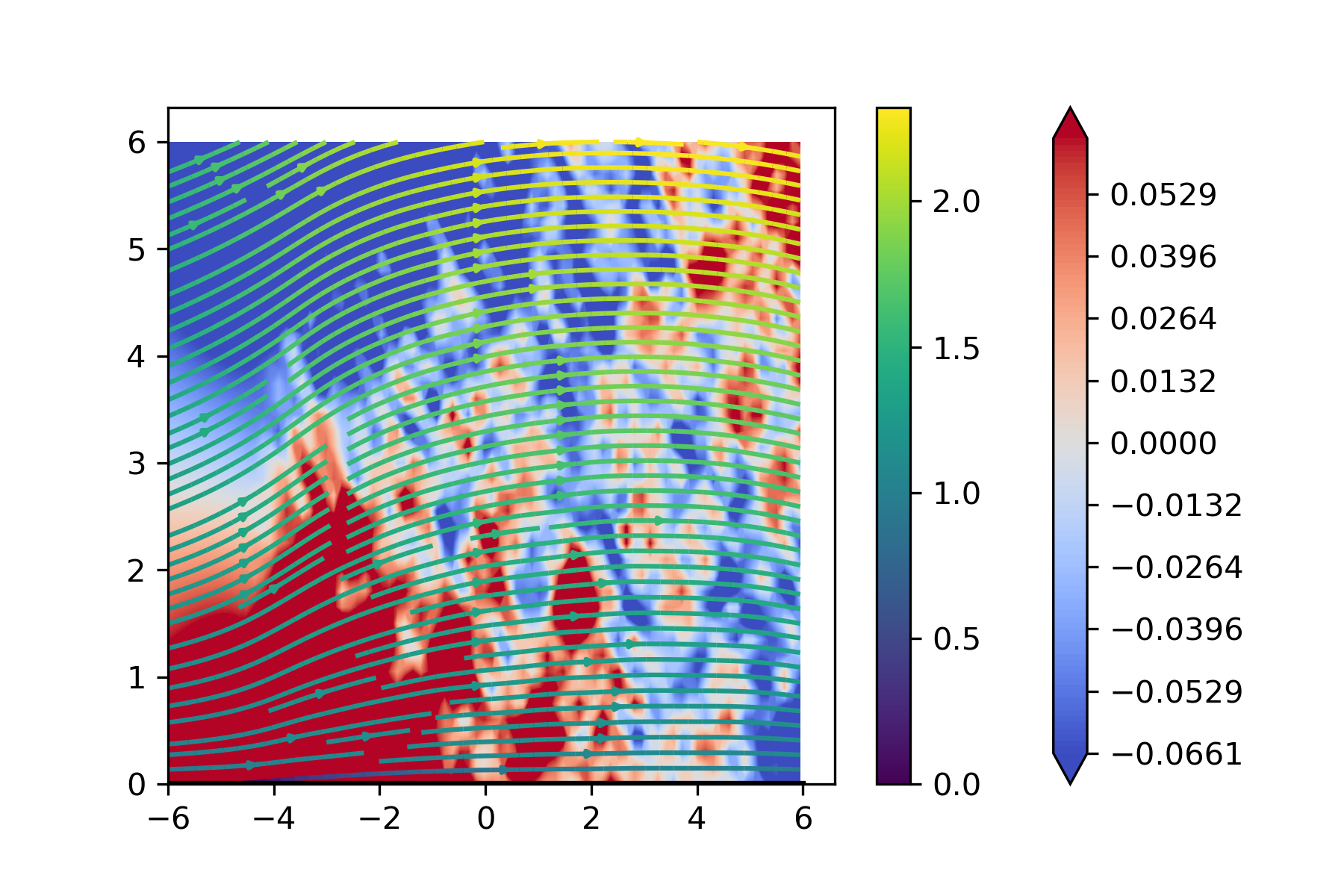} }}
    \subfloat[\centering $t=3.0$]{{\includegraphics[width=.45\linewidth]{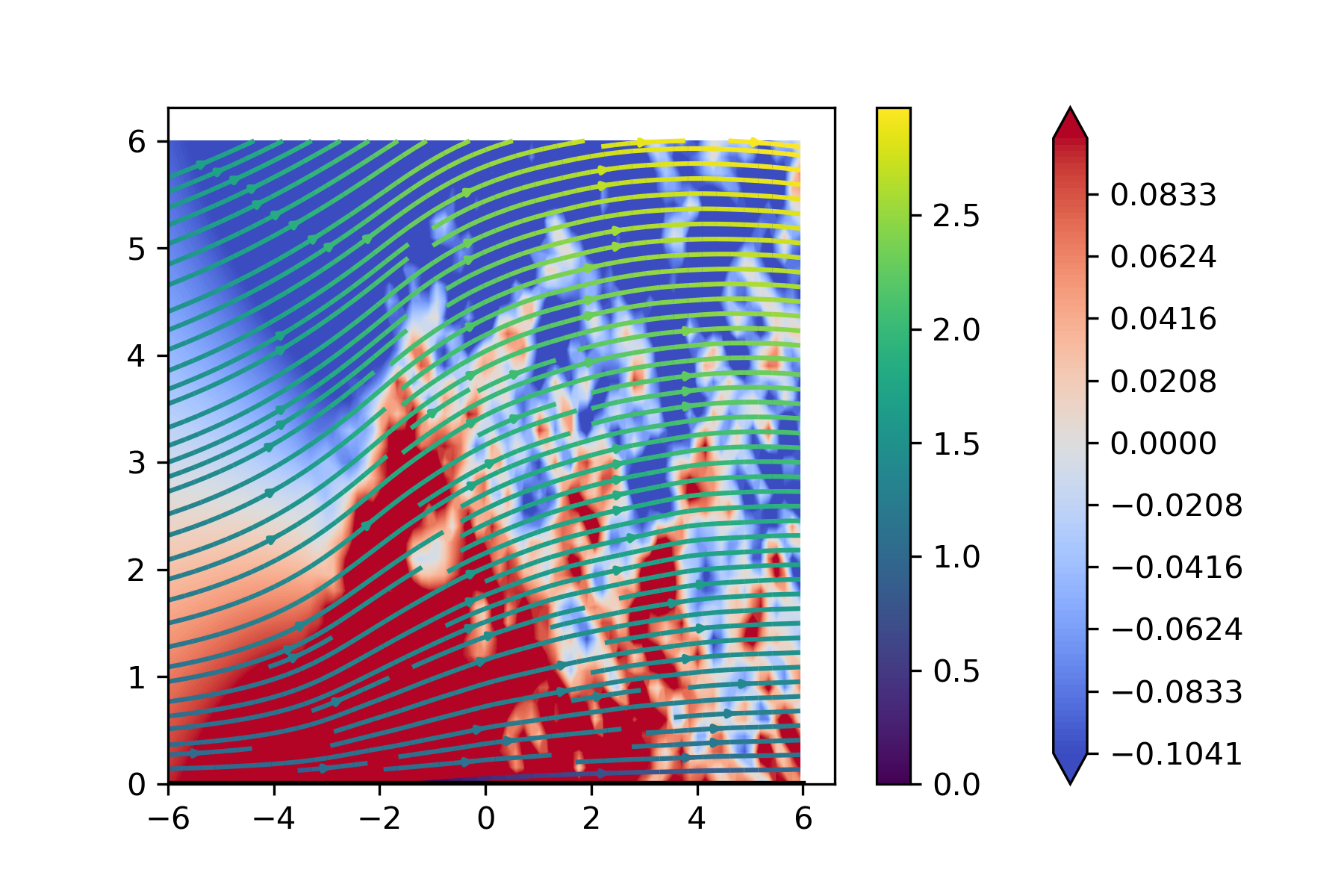}}}
    \caption{The outer layer flow at different times $t$.}
    \label{ExpFigOLayer}
\end{figure}

The outside boundary layer flow close to the rigid boundary can be
seen in Figure 2.

\newpage{}

\section*{Data Availability Statement}

The data that support the findings of this study are available from
the corresponding author upon reasonable request.

\section*{Acknowledgement}

The authors would like to thank Oxford Suzhou Centre for Advanced
Research for providing the excellent computing facility VC and ZQ are supported
(fully and partially, respectively) by the EPSRC Centre for Doctoral Training in Mathematics
of Random Systems: Analysis, Modelling and Simulation (EP/S023925/1).
 SWE is funded by Deutsche Forschungsgemeinschaft (DFG) - Project-ID 410208580 - IRTG2544 ('Stochastic Analysis in Interaction').

\end{document}